\documentclass{article}
\usepackage{graphicx} % Required for inserting images
\usepackage{geometry}
\usepackage{bm}
\usepackage{cuted}
\usepackage[normalem]{ulem}
\usepackage{moreverb,url}

\usepackage[colorlinks,bookmarksopen,bookmarksnumbered,citecolor=red,urlcolor=red]{hyperref}

\usepackage{graphicx}
\graphicspath{{svg-inkscape/}}
\usepackage[per-mode=fraction]{siunitx}
\usepackage{multirow}
\usepackage{eurosym}
\usepackage{pifont}
\usepackage[utf8]{inputenc}
\usepackage{mathtools}%includes amsmath
\mathtoolsset{showonlyrefs}
\usepackage{amsfonts}
\usepackage{amssymb}
\usepackage{paralist}
\usepackage{textcomp,gensymb}
\usepackage{subcaption}
\usepackage{float}
\usepackage[english]{babel}
\usepackage{multicol}
\usepackage{wrapfig}
\usepackage{pdfpages}
\usepackage{xcolor}
\usepackage{nicefrac}
\usepackage{algorithm}
\usepackage{algpseudocode}
\usepackage{svg}
\usepackage{interval}
\usepackage{tabularx}
\newcolumntype{b}{X}
\newcolumntype{s}{>{\hsize=.5\hsize}X}

\definecolor{BrickRed}{rgb}{0.8, 0.25, 0.33}

\makeatletter
\def\@opargbegintheorem#1#2#3{\trivlist
   \item[]{\bfseries #1\ #2\ (#3)} \itshape}
\makeatother

\newtheorem{theorem}{Theorem}[section]
\newtheorem{definition}[theorem]{Definition}
\newtheorem{proof}{Proof}[theorem]
\newtheorem{remark}[theorem]{Remark}
\newtheorem{lemma}[theorem]{Lemma}
\newtheorem{proposition}[theorem]{Proposition}

\algtext*{EndWhile}% Remove "end while" text
\algtext*{EndIf}% Remove "end if" text
\algtext*{EndFor}
\algtext*{Procedure}
\algtext*{EndProcedure}

\makeatletter
\let\OldStatex\Statex
\renewcommand{\Statex}[1][3]{%
  \setlength\@tempdima{\algorithmicindent}%
  \OldStatex\hskip\dimexpr#1\@tempdima\relax}
\makeatother

\newcommand{\M}{\mathcal{M}}                                % Manifold
\newcommand{\X}{\mathfrak{X}}                               % Vector fields
\newcommand{\g}{\mathfrak{g}}                               % Lie algebra
\newcommand{\h}{\mathfrak{h}}                               % Lie algebra
\newcommand{\R}{\mathbb{R}}                                 % Real numbers
\newcommand{\appX}{\bar{X}}
\newcommand{\state}{x}
\newcommand{\stateb}{y}
\newcommand{\appstate}{\bar{x}}

\newcommand{\ext}{\text{d}}                                 % Exterior derivative
\newcommand{\der}[2][3]{\frac{\partial{#1}}{\partial{#2}}}
\newcommand{\Coord}{Q}                                      % Coordinate chart-map
\newcommand{\Diff}{\mathrm{Diff}}

\newcommand{\pmat}[1]{\begin{pmatrix}#1\end{pmatrix}}

\title{Model order reduction via Lie groups}
\author{Yannik Wotte, Patrick Buchfink, Silke Glas, \\ Federico Califano, Stefano Stramigioli}
\date{\today}

\begin{document}

\maketitle

% \yannik{Colorful}
% \patrick{comments}
% \silke{for}
% \federico{all}
% \stefano{authors.}
 
% \twocolumn[]
    
\begin{abstract}
    Lie groups and their actions are ubiquitous in the description of physical systems, and we explore implications in the setting of model order reduction (MOR).
    We present a novel framework of MOR via Lie groups, called MORLie, in which high-dimensional dynamical systems on manifolds are approximated by low-dimensional dynamical systems on Lie groups.  
    In comparison to other Lie group methods we are able to attack non-equivariant dynamics, which are frequent in practical applications, and we provide new non-intrusive MOR methods based on the presented geometric formulation.
    We also highlight numerically that MORLie has a lower error bound than the Kolmogorov $N$-width, which limits linear-subspace methods.  
    The method is applied to various examples: 1. MOR of a simplified deforming body modeled by noisy point cloud data following a sheering motion, where MORLie outperforms a naive POD approach in terms of accuracy and dimensionality reduction. 2. Reconstructing liver motion during respiration with data from edge detection in MRI scans, where MORLie reaches performance approaching the state of the art, while reducing the training time from hours on a computing cluster to minutes on a mobile workstation. 3. An analytic example showing that the method of freezing (a previous Lie group method for MOR introduced in~\cite{Ohlberger2013}) is analytically recovered as a special case, showing the generality of the geometric framework.    
\end{abstract}

\begin{multicols}{2}
    
\section{Introduction}\label{sec:introduction}

% Currently:
% - Subspace
% - Non-separable
% - Lie group
% - Contrast to submanifold methods
% \yannik{remove equation numbers of unreferred equations}
\begin{figure*}
    \centering
    \def\svgwidth{0.8\linewidth}\scriptsize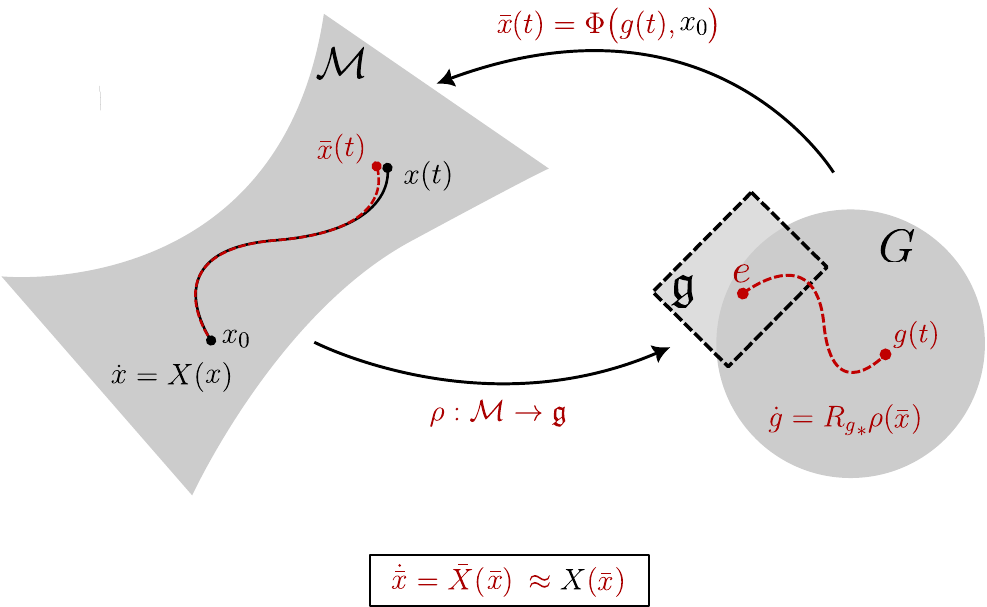
    %\includesvg[width=0.8\textwidth]{Figures/MorLie_Overview_x.svg}
    \caption{Summary of MORLie. %The idea of MORLie is to replace high-dimensional state evolution by low-dimensional group evolution. The reduced state $g(t)\in G$ evolves on a Lie group, and the full state is reconstructed via a group action. The induced motion on $\M$ is constrained to lie on the orbit of the initial condition. 
    Given a full order model (FOM) defined by a vector field $X$ on a state-manifold $\M$, we construct a reduced order model (ROM) on a lower-dimensional Lie group $G$. A group action $\Phi:G\times\M\rightarrow\M$ is used to reconstruct states via $\appstate(t) = \Phi(g(t),\state_0)$, where the evolution of $g(t)$ is determined by a reduced vector field $\rho:\M\rightarrow\mathfrak{g}$. The ROM dynamics on $G$ are defined as $\dot{g} = {R_g}_*\rho(\appstate)$, and induce approximated dynamics on $\M$ that are constrained to group orbits. The goal is to choose $(G,\Phi,\rho)$, such that the reconstructed solutions $\appstate(t)$ closely follows the FOM dynamics $\dot{\appstate} \approx X(\appstate)$.
    }\label{fig:MorLie_Overview}
\end{figure*}

Model order reduction (MOR) is an essential step for the repeated simulation and optimization of distributed and multi-scale engineering systems, such as soft robotic manipulators~\cite{Sadati2023,Mathew2025,Pustina2024}, 
hybrid reaction-diffusion systems~\cite{Chinesta2011,Grepl2012}, fluids~\cite{Lassila2014} and systems with fluid-structure interactions~\cite{Lee2024}. These systems are frequently modeled by (parameterized) partial differential equations (PDEs), which yield high dimensional ordinary differential equations (ODEs) upon spatial discretization~--~often considered as full order models (FOMs). The core idea of MOR is to approximate these FOMs by low-dimensional surrogate models, called reduced order models (ROMs), that can be evaluated with far lower computational complexity. 

% Address slowly decaying Kolmogorov $N$-width for point-cloud data:
% \begin{itemize}
%     \item Linear-subspace methods struggle when real-world data does not evolve on subspaces, e.g., particle streams from real experiment
%     \item Submanifold methods likewise struggle to address such issues~--~e.g, when random initial clouds of particles are not restricted to any particular submanifold
% \end{itemize}
% Transition to importance of Lie group theory, which, among others, can describe such data. 

A big class of MOR techniques begins with large datasets of snapshots, which could be solutions from high-fidelity simulations of FOMs or state measurements from real world experiments, for an interesting range of system parameters and initial conditions. In the following, we denote the high-dimensional state-manifold of the FOM as $\M$, a set of solution snapshots (classically, the solution manifold) as $S \subseteq \M$, and we lump parameters and initial conditions as $\mu \in \mathcal{P}$, with $\mathcal{P}$ a compact set. %as columns of a snapshot matrix $S \in \R^{n\times m}$. 
%
% \begin{equation}
%     S = \{ \state(t,\mu) \;|\; t \in [0,T], \mu \in \mathcal{P}\}\,.
% \end{equation}
%
In linear-subspace methods~\cite{Hesthaven2015,Quarteroni2015}, $\M$ is a Hilbert space, and $S$ is approximated by a subspace $W \subseteq \M$, e.g., spanned by suitable reduced basis elements $\{\varphi_1,\cdots,\varphi_N\},\; \varphi_i \in W$~\cite{Sirovich1987}. Introducing time-dependent coefficients $c_{\mu}^i(t)\in \R$ indexed by $i$, real solutions $u_{\mu}(t) \in S$ are approximated as 
\begin{equation}\label{eq:subspace-approximation}
    u_{\mu}(t) \approx \sum_{i=1}^N c_{\mu}^i(t)\varphi_i\,. 
\end{equation}
Linear-subspace methods scale to high-dimensional systems, can preserve physical structure and can capture input-output behavior for control applications~\cite{Volkwein2013,Besselink2013}. Theoretical error bounds are well-known: the Kolmogorov $N$-width~\cite{Pinkus1985} gives, depending on $S$, a lower bound to the approximation error that can be achieved by an $N$-dimensional subspace $W$ of $\M$. However, linear-subspace methods are ill-suited for $S$ with a slowly decaying Kolmogorov $N$-width, e.g., in certain transport dominated problems from fluid dynamics or heat transport, where solutions do not evolve on fixed subspaces~\cite{Ohlberger2013}. %We later highlight that also dynamic point-cloud data, e.g., extracted from videos of rigid bodies, deforming bodies, and particle streams from wind-tunnels, fall into this category of problems.

This slow decaying Kolmogorov $N$-width is often referred to as the Kolmogorov barrier, and to break it, nonlinear MOR methods are required~\cite{Peherstorfer2022}, for example by allowing also the basis elements to change over time: 
% These methods allow for time-dependent reduced bases, i.e.,
\begin{equation}\label{eq:non-separable-approximation}
    u_\mu(t) \approx \sum_{i=1}^N c_{\mu}^i(t)\varrho_{\mu,i}(t)\,. % u(x,t,\mu) \approx \sum_{i=0}^N c^i(t,\mu)\varrho_i\big(x,\bm{\alpha}(t,\mu)\big)\,.
\end{equation}
Research on MOR methods going beyond the Kolmogorov barrier include: adaptive basis methods~\cite{Peherstorfer2020,Selvaraj2024}, which recompute the reduced basis over time, polynomial approximation methods~\cite{Sharma2023,Barnett2023,Geelen2023}, and machine learning approaches~\cite{Lee2020,Buchfink2023c,Chen2023,Barnett2023,Romor2025}, which are sufficiently complex to express relations of the form~\eqref{eq:non-separable-approximation} but typically incur large computational cost. 

A recent differential geometric description of MOR encompasses many of the aforementioned methods in a common description as submanifold methods~\cite{Buchfink2023a}: given the FOM as a dynamical system on a $\M$, the ROM is a dynamical system on a submanifold (rather than a subspace) $\mathcal{N}$ of $\M$.

In the present article we use a new geometric approach, utilizing the action of Lie groups. Lie groups act on the underlying configuration and state manifolds for a wide variety of physical systems, including rigid bodies, flexible bodies and fluids, where they allow to describe and exploit symmetries of the systems' kinematics and dynamics~\cite{Holm2009}. With reference to Figure~\ref{fig:MorLie_Overview} our approach reduces a dynamical system on a high-dimensional state-manifold $\M$ to a lower-dimensional dynamical system on a Lie group $G$. Let $\Phi:G\times\M\rightarrow\M$ denote an action of $G$ on $\M$, and $u_{\mu,0} \in \M$ the initial condition of an FOM trajectory, then we work towards approximations of the type
\begin{equation}\label{eq:mfd-lie-group-approximation}
    u_\mu(t) \approx \Phi(g_\mu(t),u_{\mu,0})\,.% = g(t) \cdot u_0\,.
\end{equation}
While submanifold methods for MOR restrict trajectories to lie on a fixed submanifold, we show that~\eqref{eq:mfd-lie-group-approximation} instead restricts the dynamics to directions tangent to group orbits. Thus, rather than constraining where states can lie, we constrain how they are allowed to move. In this way, initial conditions are preserved: for any initial state $u_0 \in \M$, the approximation evolves on the orbit through $u_0$, so no projection onto a reduced manifold is required. More formally, our method induces a low-dimensional distribution on the state-space manifold, rather than selecting a submanifold of it.
%While submanifold methods for MOR restrict trajectories to submanifolds, we show that~\eqref{eq:mfd-lie-group-approximation} restricts directions to lie tangent to group orbits. Every initial condition defines its own orbit, along which the approximation evolves. Therefore, no projection step is required: initial states $u_0 \in \M$ are not projected to a submanifold, and approximations can reproduce initial states exactly. More formally, we show that our method induces a low-dimensional distribution on the state-manifold, rather than picking a submanifold of it: \textit{submanifolds constrain where states can lie, distributions constrain how states can move.}

Early Lie group methods for MOR~\cite{Ohlberger2013,Beyn2004,Rowley2000} split dynamics into group and manifold parts (vertical and horizontal components on a principal bundle, respectively) and applied submanifold methods to reduce the latter. The splitting explicitly required equivariance of the underlying dynamical systems, and MOR methods were intrusive, i.e., based on explicit knowledge of the FOM dynamics. Our method does not require equivariance, and allows a novel non-intrusive (i.e., data-driven) approach that exploits the presence of approximate group motions. Recently~\cite{Kleikamp2022} used a learning approach to find a ROM on a subgroup of $G = \Diff(\R^d)$ to reduce Burgers equation with discontinuous shocks. Our method works towards a general geometric approach enabling other choices of $G$, in a fashion that is compatible with machine-learning, but does not rely on it. 

Concretely, our contributions are:
\begin{enumerate}
    \item Development of a new MOR procedure via Lie groups, which we name MORLie (Section~\ref{ssec:MORLie_Theory}).
    \item Formalization of a \textit{group Kolmogorov $N$-width}, providing a lower bound for the approximation error of MORLie (Section~\ref{ssec:group-N-width}).
    \item Development of optimization problems to compute ROMs, both intrusive and non-intrusive, and numerical methods to solve them (Section~\ref{sec:optimization}).
    \item Multiple example applications: analytically recovering the method of freezing~\cite{Ohlberger2013} (Section~\ref{sec:analytic_examples}), point clouds undergoing sheering motions and liver tracking during respiration (Section~\ref{sec:numeric_examples}).
\end{enumerate}

The article is ordered as follows. We conclude the introduction with related literature and notation. Section~\ref{sec:background} provides background on Lie groups and their actions on manifolds. Section~\ref{sec:problem_formulation} introduces the main problem, i.e., to describe an MOR method that results in a ROM on a Lie group and to conceptually connect such a method to existing theory on MOR on manifolds.
Next, Section~\ref{sec:morlie} formally describes MORLie and introduces the group Kolmogorov $N$-width. Afterwards, Section~\ref{sec:optimization} presents intrusive and non-intrusive optimization methods, and initial ideas for hyperreduction. Section~\ref{sec:analytic_examples} provides applications of MORLie to linear transport and the method of freezing~\cite{Ohlberger2013}. Section~\ref{sec:numeric_examples} provides numerical example applications and a discussion of the results. A conclusion and outlook are given in Section~\ref{sec:conclusion}. %The Appendix~\ref{sec:appendix} presents relevant proofs to avoid cluttering the main article.

\subsection{Related literature}
% Reviews on MOR~\cite{Besselink2013,Ohlberger2016,Lu2021,Scarciotti2024}:
The review~\cite{Besselink2013} presents a comparative study of linear-subspace MOR in structural mechanics, describing modal truncation techniques, acceleration methods, Krylov subspace based MOR and balanced truncation. 
%For a review on linear-subspace MOR in structural mechanics see~\cite{Besselink2013}. 
In~\cite{Lu2021}, center manifold reduction as well as the Galerkin method, modal synthesis method and proper orthogonal decomposition (POD) are reviewed, where the latter three are linear-subspace methods. 
%In~\cite{Lu2021}, also center manifold reduction is reviewed, a technique that falls into submanifold MOR.
The review~\cite{Scarciotti2024} focusses on interconnection-based MOR for both linear-subspace cases and nonlinear applications.
In~\cite{Ohlberger2016} various linear-subspace techniques and nonlinear methods are reviewed, also putting a focus on a posteriori error bounds and the Kolmogorov $N$-width of the to-be-reduced systems. %The review also includes the method of freezing~\cite{Ohlberger2013}, which employs Lie groups.

% Submanifold MOR (ManiMOR, Transformation-based methods, etc.):
Submanifold methods for MOR are formalized in a differential geometric language in~\cite{Buchfink2023a}. Most nonlinear approaches to MOR can be seen as particular cases of submanifold methods: transformation-based methods~\cite{Cagniart2019,Cagniart2020_thesis,Black2020,Schulze2023} apply linear or nonlinear, time-dependent transformations to a fixed and reduced basis (still allowing for time-varying coefficients). A similar formalism can also be derived by transforming snapshot matrices with nonlinear transformations~\cite{Taddei2020}, which can reduce the Kolmogorov $N$-width. By transforming a reduced basis, such methods effectively represent the system on a submanifold of the state-space. %Adaptive basis approaches~\cite{Peherstorfer2022,Peherstorfer2020} have also been shown to be subject to a different error bound, presented as the Kolmogorov $(N,M)$-width in~\cite{Rim2020}. 
The work of~\cite{Reiss2018,Krah2025} uses a shifted POD, where the shift can be seen as a particular choice of Lie group action~--~we generalize towards accommodating arbitrary Lie group actions, see e.g., Section~\ref{ssec:method_of_freezing}.
%and dictionary-methods
In~\cite{Pagliantini2021} integration on Lie groups $SO(n), SP(n)$ is used for geometrically exact integration of time-varying reduced bases for Hamiltonian systems. We interpret this as a submanifold method, since the Lie groups are made to act on a Stiefel (sub-)manifold~\cite{Feppon2018} that the time-varying reduced basis of the total system.  % Discussion: comment how~\cite{Feppon2018} identify DO methods with the geometry of Stiefel manifolds, and how Lie group methods take this further, abandoning Stiefel manifolds when orthogonal decompositions are no longer well-defined.

Various machine learning approaches fall into the category of submanifold methods for MOR:\ % add~\cite{Lee2020,Fresca2021,Buchfink2023c}
in~\cite{Lee2020} deep convolutional autencoders are applied for manifold Galerkin projection and manifold least-squares Petrov-Galerkin projection, the approach overcomes the Kolmogorov $n$-width, outperforming linear subspace methods on Benchmark 1D Burgers equation and chemically reacting flows, and aposteriori error estimates are provided. ROMs for nonlinear parameterized PDEs are learned by separately identifying deep neural nets parameterizing the submanifold and the reduced dynamics, in~\cite{Fresca2021}. Structure-preserving MOR method for Hamiltonian systems is presented in~\cite{Buchfink2023c}, where deep autoencoders are used and symplecticity is enforced softly as an additional cost-function term, strongly improving convergence over approaches that do not preserve structure. 
The machine learning approach~\cite{Kaszs2022} learns a 2D subspectral manifold to represent Couette flow and its bifurcations at low Reynolds numbers. Autoencoders, which naturally parameterize submanifolds, are applied to achieve MOR of compressible and incompressible Navier-Stokes equations for a large array of examples, including flow around a wing for different Mach numbers, in~\cite{Romor2025}. In~\cite{Viswanath2025} reduced-order neural operators on sparse graphs are learned in a physics-informed manner, and in~\cite{Steeven2024} data on fixed point-clouds is used to learn physics from partial information, using a reduced latent representation that parameterizes a submanifold. In~\cite{Barnett2023} a multi-stage neural network based approach for mitigating the Kolmogorov $N$-width is applied to shock-dominated, unsteady flow and hyperreduction is achieved. A submanifold is composed of an initial reduced basis and additional, nonlinear input from the neglected basis elements.

% Lie groups in MOR (method of freezing, etc.)
% -> Find origin in:
% Symmetry reduction and applications
% Registration methods

The method of freezing~\cite{Ohlberger2016} highlighted earlier, splits certain equivariant PDEs into a Lie group and a vector-space part, applying linear-subspace MOR to the latter. It has its background in reduction of equivariant dynamic systems~\cite{Blankenstein2001,Cendra2001,Rowley2003}, which had early applications to MOR of equivariant PDEs in~\cite{Rowley2000,Beyn2004}.% as well as geometric image registration~\cite{Bruveris2011,Bruveris2013}.  

Some recent machine learning approaches to fluid- and continuum mechanics do not fit into a submanifold framework, but can be interpreted as instances of the general Lie group methods we present:~\cite{Chen2023} learn deformation maps for continuum mechanics applications, achieving a reduced-order latent-state representation that is grid-agnostic, allowing hyperresolution during reconstruction of the full state. In~\cite{Cucchiara2024} a similar approach is applied to both viscid/ inviscid and compressible/incompressible fluid dynamics. Both approaches implicitly learn representations of the Lie group $\Diff(\R^3)$. In~\cite{Kleikamp2022} the Lie group $\Diff(\R^d)$ is explicitly parameterized by integrating a finite set of neural net parameterized vector fields, which are optimized in a non-intrusive fashion for the example of the 2D Burgers equation, where the resulting method is able to represent shocks.
% Lie groups in ML: Approximate symmetry in machine learning
% MorLie in the wild (material point method, etc.)

Apart from~\cite{Kleikamp2022}, applications of Lie groups in MOR were restricted to equivariant dynamic systems, and according to our best knowledge no work in MOR explicitly treats approximate equivariance. Within the machine learning community, both equivariant and approximately equivariant systems were investigated in detail:~\cite{Wang2022} investigates approximately equivariant networks for learning approximately equivariant dynamics, and apply this to learning dynamics of smoke plumes and inlet flows. In~\cite{Huang2023} approximately equivariant graph neural networks are studied, and~\cite{Park2025} study approximate equivariance in reinforcement learning. The notion of approximate equivariance was formalized by~\cite{Petrache2023}, who also show the importance to correctly identify the degree of equivariance (or relaxation thereof) in a given problem. In~\cite{Otto2023}, a unified framework is presented to discover, enforce and promote symmetry in machine learning applications, including linear algebraic relations to find symmetry Lie-subgroups of datasets, functions and dynamic systems, given the action of a larger Lie group that fails to be symmetric. 

For a review of the Kolmogorov $N$-width and recent nonlinear widths extending the concept, see the dedicated Section~\ref{ssec:Kolmogorov-N-width}.

\subsection{Notation}
See~\cite{Lee2012} for background on differential geometry and~\cite{Holm2009} for background on Lie groups.
Calligraphic letters $\M,\,\mathcal{N}$ denote smooth manifolds. Given a point $\state \in \M$, we let $T_\state\M$ denote the tangent space at $\state$ and the tangent bundle is the disjoint union $T\M = \dot{\bigcup}_{\state\in\M} T_\state \M$. Then $\Gamma(T\M)$ is the set of sections of $T\M$ which consists of vector fields $X,Y \in \Gamma(T\M)$ over $\M$. The vector field $[X,Y]$ is called the Lie-bracket of vector fields $X,Y$.

Denote as $C^k(\M,\mathcal{N})$ the set of $k$-times differentiable maps from $\M$ to $\mathcal{N}$. We define $C^k(\M) := C^k(\M,\R)$. For $\phi \in C^k(\M,\mathcal{N})$ with $k\geq 1$ the push-forward is $\phi_*: T_\state\mathcal{M} \rightarrow T_{\phi(\state)}\mathcal{N}$. 

The group of diffeomorphisms is $\Diff(\M) \subseteq C^\infty(\M,\M)$, and contains the smooth maps whose inverse is also smooth. %:= \{\phi \| \phi,\phi^{-1} \in C^\infty(\M,\M) \}$. 

Further $G$ denotes a Lie group, $g,h$ denote arbitrary elements of $G$ and $e$ denotes the group identity. %The left and right translations are the maps $L_g(h) := gh$ and $R_g(h) := hg$, respectively. 
We denote by $\g = T_e G$ the Lie algebra of $G$, by $\widetilde{A},\widetilde{B} \in \g$ its elements, and by $[\widetilde{A},\widetilde{B}]$ their (left) Lie-bracket. 
The exponential map is $\exp:\g\rightarrow G$, and we similarly denote $\exp(\widetilde{A}) =: e^{\widetilde{A}}$. %
Lastly, $\Phi:G\times \M \rightarrow \M$ denotes a group action of $G$ on $\M$. For $\state\in\M$ and when the choice of $\Phi$ is clear from context, we denote $\Phi_g(\state) = g\cdot \state$.

% As an operator $[\cdot,\cdot]:\g \times \g \rightarrow \g$ denotes the left Lie bracket on $\g$.

% $\Delta^k$ denotes a $k$-dimensional distribution that assigns to each $\state$ a $k$-dimensional subspace $\Delta^k_\state\subseteq T_\state\mathcal{M}$.

\section{Background}\label{sec:background}

\subsection{Lie groups and their actions} 
% A number of properties of the action of a Lie group $G$ on a manifold $\M$ are reviewed.
Given $G$ an $n$-dimensional Lie group, and $g,h \in G$. The left and right translations are, respectively: 
\begin{align}
    L_g(h) := gh \,,\\
    R_g(h) := hg \,.
\end{align}
%With $\g = T_e G$ the Lie algebra of $G$, the exponential map $\exp:\g\rightarrow G$ is a local diffeomorphism.

\begin{definition}[Group action]\label{def:group_action}
    Let $\Phi:G\times\M\rightarrow\M$ be a smooth map such that $\Phi_g(\cdot) := \Phi(g,\cdot)$ is a diffeomorphism:
    \begin{align}
        \Phi:&\, G \rightarrow \Diff(\M) \,;\;  g \mapsto \Phi_g(\cdot) \,.  
    \end{align}
    Then $\Phi$ is a \textbf{left group action} of $G$ on $\M$ if it is a homomorphism~\cite{Holm2009}: 
    \begin{equation}\label{eq:group_homomorphism}
        \Phi_{gh} = \Phi_g \circ \Phi_h \,, % %\Phi(gh,\cdot) = \Phi(g, \cdot ) \circ \Phi(h,\cdot) = \Phi\big(g,\Phi(h,\cdot)\big)
    \end{equation}
    and $\Phi$ is a \textbf{right group action} of $G$ on $\M$ if it is an anti-homomorphism~\cite{Holm2009}:
    \begin{equation}\label{eq:group_antihomomorphism}
        \Phi_{gh} = \Phi_h \circ \Phi_g\,.
    \end{equation}
\end{definition}

An example for a left action of $G = GL(3,\R) := \{g\in\R^{3\times 3} \;|\; \det(g) \neq 0\}$ on $x \in \R^3$ is the matrix-vector product $\Phi(g,x) = gx$, and an example for a right action is $\Phi(g,x) = g^{-1}x$. Both actions represent combined rotation, scaling and sheering of $\R^3$. 

% The action allows to formalize the notion of symmetry: a function $f:\M\rightarrow\R$ is said to be invariant under $\Phi$ if $f = f\circ\Phi_g$, and a vector field $X \in \Gamma(TM)$ is said to be equivariant under $\Phi$ if $X\circ\Phi_g = \Phi_{g,*2}X$.

% With reference to Figure \ref{fig:MorLie_Overview}, a group action is required to express % the reconstruction equation
% $\appstate(t) = \Phi\big(g(t),\state_0\big)$. However a lot of additional depth hides in this definition, and an immediate questions is about the merrit of requiring $\Phi$ to be a homomorphism, which we aim to answer in the sequel. 

The action can have a number of additional properties:

% The topology of $\mathcal{O}(\state_0)$ since it restricts the behaviour of $\appstate$. Given that $\Phi$ is a homomorphism, this topology is fully determined by properties of the action:

\begin{definition}[Properties of an action]\label{def:properties_of_action}
    The group action $\Phi$ can also have the following properties~\cite{Holm2009}:
    \begin{itemize}
        \item \textbf{Faithful}: for all $g \in G\backslash e$ there exists $\state \in \M$ such that $\Phi(g,\state) \neq \state$.
        \item \textbf{Free}: $g \cdot\state = \state$ if and only if $g = e$. %\Phi(g,\state)
        \item \textbf{Proper}: when sequences $\{\state_n\}$ and $\{g_n\cdot\state_n\}$ converge in $\M$, then $\{g_n\}$ converges in $G$. 
        \item \textbf{Transitive}: for all $\state,\stateb \in \M$ there exists $g \in G$ such that $\state = g \cdot \stateb$. %\Phi(g,\stateb)
    \end{itemize}
\end{definition}

Going back to the example of $GL(3,\R)$, the left and right action are neither transitive nor free: the element $x = 0 \in \R^3$ can not be reached from any other $x$, and $g \cdot 0 = 0$ for any $g \in GL(3,\R)$. Instead, the actions of $GL(3,\R)$ on $\R^3 \backslash 0$, i.e., $\R^3$ excluding $0$, are transitive, but again not free. 

% Then, %the reconstructed solution
% $\appstate(t)$ lies within the orbit of $\state_0$, defined as %We relegate a conclusive answer to this question to Section~\ref{sec:MORLie}.

\begin{definition}[Orbit of an action]\label{def:orbit}
    The \textbf{orbit} of a point $\state \in \M$ under the action $\Phi$ is the set
    \begin{equation}
        \mathcal{O}(\state) := \{ g \cdot \state \;|\; g \in G \} \,. % \Phi(g,\state)
    \end{equation}
\end{definition}

In other words, the orbit $\mathcal{O}(\state)$ collects all elements that can be reached from $\state$ by applying a group element $g$. We note that $\mathcal{O}(\state)$ is a submanifold of $\M$. Again going back to the example of $GL(3,\R)$, $\mathcal{O}(0) = \{0\}$ for $0 \in \R^3$.

% Depending on properties of the action, $\mathcal{O}(\state)$ can be further classified as shown in the following.

\begin{theorem}[Properties of the orbit,~\cite{Holm2009}]\label{thm:properties_of_orbit}
    If $\Phi$ is transitive, then $\mathcal{O}(\state) = \M$, and it is identical for all $\state\in\M$. If $\Phi$ is free, then $\mathcal{O}(\state)$ is isomorphic to $G$ as a manifold, for all $\state \in \M$. If $\Phi$ is free and proper, then $\M/G := \{\mathcal{O}(\state) \;|\; \state \in \M\}$ is a uniquely determined smooth manifold of dimension $\dim \M - \dim G$.   
\end{theorem}

\begin{definition}[Infinitesimal generator]\label{def:infinitesimal_generator}
    Given $\widetilde{A} \in \g$ and $\Phi:G\times\M\rightarrow\M$, the \textbf{infinitesimal generator} $X_{\widetilde{A}} \in \Gamma(T\M)$ of $\widetilde{A}$ w.r.t.\ $\Phi$ is defined pointwise as: 
    \begin{equation}\label{eq:inf_generator}
        X_{\widetilde{A}}(\state) := \frac{\ext}{\ext t} \Phi(\exp(\widetilde{A}t),\state)\,.%\Phi_{*_1}(e,\state)\widetilde{A}\,,
    \end{equation}
    % defined by 
    % its action on a function $f \in C^1(\M,\R)$ as
    % \begin{align}
    %     \big(\appX(f)\big)(\state) :=& \frac{d}{dt} f\circ\Phi\big(\exp(\rho(\state)t),\state\big) \\
    %     =& \big(\Phi_{*_1}(e, \state)\rho(\state)\big)(f) \,. \nonumber
    % \end{align}
\end{definition}
Intuitively, the infinitesimal generator $X_{\widetilde{A}}(\state)$ describes the velocity at $\state$ generated by the action of $\exp(\widetilde{A}t)$. %one-parameter subgroup $g(t)\cdot \state$ on $\M$ at $t=0$, if $g(0) = e$ and $\dot{g}(0) = \widetilde{A}$. 

\begin{definition}[Distribution]
    A smooth distribution $\Delta \subseteq T\M$ assigns to any $\state\in\M$ a subspace $\Delta_\state \subseteq T_\state\M$ smoothly, i.e., for any $\state \in \M$ there are smooth vector fields $\{X_1,\cdots, X_k\} \subseteq \Gamma(T\M)$ and $\mathcal{U}_\state \subseteq \M$ a neighborhood of $\state$ such that 
    \begin{equation}
        \forall \stateb \in \mathcal{U}_\state:\; \Delta_\stateb = \textnormal{span}\{X_1(\stateb),\cdots, X_k(\stateb)\}\,.
    \end{equation}
    A distribution is called
    \begin{itemize}
        \item \textbf{Regular} of dimension $k$ if $\Delta_\state \subseteq T_\state \M$ is $k$-dimensional for every $\state \in \M$.
        \item \textbf{Integrable} if for any $\state \in \M$ exists a submanifold $\mathcal{N} \subseteq \M$ such that $\Delta_\state = T_\state \mathcal{N}$.
    \end{itemize}
\end{definition}

In the present work, the relevant distribution will be the span of infinitesimal generators.

% Again with reference to Figure~\ref{fig:MorLie_Overview}, given $\widetilde{A} \in \g$ and $\dot{g} = {R_g}_*\widetilde{A}$, then $\dot{\appstate}$ is determined by the infinitesimal generator $X_{\widetilde{A}}(\state)$. The definition of the reduced vector field $\rho:\M\rightarrow\g$ comes naturally as the ``best fit'' (to be further specified in Section \ref{sec:MORLie}) of the infinitesimal generator $X_{\rho(\state)}$ to the full order $X(\state)$. All possible infinitesimal generators span a distribution $\Delta \subseteq T\M$, defined in the following. 

% Theorem: Lie algebra anti-homomorphism 

\begin{definition}[Distribution induced by $G,\Phi$]\label{def:induced_distribution}
    A Lie group $G$ and an action $\Phi$ induce, at each $\state \in \M$, a subspace $\Delta_\state \subseteq T_\state \M$ by the image of the infinitesimal generator: % $\Phi_{*_1}(e,\state):\g \rightarrow T_\state\M$: %\federico{If $\Delta$ is the image of the linear operator, it is a full subspace, while (7) gives a single vector for each $\g$. $\Delta$ is the collection of all (7) for all possible $\g$.} \yannik{Here, $\mathfrak{g}$ is not a single element of a Lie algebra, but rather the full Lie algebra $\mathfrak{g} = T_eG$. $\Delta_\state$ is its image under the pushforward, which is a full subspace.}
    \begin{align}
        \Delta_\state %&= X_{\g}(x) \\ %\big(\Phi_{*_1}(e,\state)\big)(\g) \\
        &= \{ X_{\widetilde{A}}(x) \,|\, \widetilde{A} \in \g \}\,. %\nonumber 
        %\{ \big(\Phi_{*_1}(e,\state)\big)\widetilde{A} \,|\, \widetilde{A} \in \g \}\,. \nonumber
    \end{align}
    This in turn defines a distribution $\Delta \subseteq T\M$ as
    \begin{equation}
        \Delta = \dot{\bigcup}_{\state\in\M} \Delta_\state\,,
    \end{equation}
    which we call the distribution induced by $\Phi$.
\end{definition}

The induced distribution collects the tangent vectors to $\mathcal{O}(\state) \subseteq \M$. %, i.e., velocities contained in $\Delta$ are tangent to some orbit $\mathcal{O}(\state)$ of the action $\Phi$.

\begin{theorem}[Properties of the induced distribution]\label{thm:properties_of_induced_distribution}
    If the action $\Phi$ is free, then the induced distribution $\Delta$ is regular of dimension $\dim(G)$. % $k = \dim(G)$
    The induced distribution $\Delta$ is always integrable, and tangent to the orbits $\mathcal{O}(\state)$ at every $\state\in\M$. %$\Delta_\state = T_\state \mathcal{O}(\state)$ with orbits $\mathcal{O}(\state)$ at every $\state\in\M$.
\end{theorem}

\begin{proof}
    The first statement is a direct consequence of~\cite[Proposition 7.26]{Lee2012}. Integrability of $\Delta$ follows from~\cite[Proposition 19.2]{Lee2012} and Frobenius' Theorem~\cite[Theorem 19.12]{Lee2012}. Finally, tangency of $\Delta_x$ and $\mathcal{O}(\state)$ is a direct consequence of Defs.~\ref{def:orbit},~\ref{def:infinitesimal_generator} and~\ref{def:induced_distribution}. %Theorem 2.16/Theorem 2.18 Theorem 19.12
\end{proof}

In other words, integrability of $\Delta$ directly corresponds to $\mathcal{O}(\state)$ being submanifolds of $\M$.

% In the context of model order reduction via Lie groups, this induced distribution restricts the best possible reduced vector field $\rho:\M \rightarrow \g$. 

\subsection{The Kolmogorov \textit{N}-width}\label{ssec:Kolmogorov-N-width}
Given a Hilbert space $\M$, and denoting by $W\subseteq\M$ an $N$-dimensional \textit{subspace}, the Kolmogorov $N$-width of a \textit{subset} $S \subseteq \M$ is 
\begin{align}\label{eq:Kolmogorov-N-width}
    d_N(S) = &\, \inf_{\substack{W \subseteq \M  \\ \dim W = N}} \; \sup_{\state \in S} \; \inf_{y\in W} \| \state - y\| \\
    = &\,  \inf_{\substack{W \subseteq \M  \\ \dim W = N}} E(W,S) \nonumber \,,
\end{align}
with $E(W,S) = \sup_{\state \in S} \; \inf_{y\in W} \| \state - y\|$ the maximum approximation error.

In the context of MOR the Kolmogorov $N$-width is analyzed for the set of solution snapshots collected in $S$, and it provides a lower bound on the worst-case error of linear-subspace methods in approximating $S$. Analytic estimates of the Kolmogorov $N$-width are well-known for certain problem classes: for example, the Kolmogorov $N$-width is known to decay exponentially for diffusion problems~\cite{Bachmayr2010}, but only slowly (order $\sim\frac{1}{\sqrt{N}}$) for certain linear transport and wave problems~\cite{Greif2019}. While achieving the theoretical best case given by $d_N(S)$ is difficult, it can often be approached. Yet, the slow decay for linear transport suggests that linear-subspace methods are ill-suited for most transport dominated problems~\cite{Peherstorfer2022}.

Lower bounds for nonlinear MOR methods are given by nonlinear alternatives to the linear-subspace Kolmogorov $N$-width. For these nonlinear widths the linear width~\eqref{eq:Kolmogorov-N-width} generally presents an upper bound, refer also to~\cite{Cohen2023,Buchfink2023b}. 

% Denoting $E_N:\M\rightarrow\R^N$, $D_N:\R^N\rightarrow\M$ continuous encoder and decoder functions, then the manifold $N$-width~\cite{DeVore1989,DeVore1993,Rim2023} of $S \subset \M$ is 
% \begin{align}\label{eq:mfd-N-width}
%     d^{\text{mfd}}_N(S) = &\, \inf_{\substack{E_N, D_N \\ \text{continuous}}} \; \sup_{\state \in S} \| \state - D_N \big(E_N(\state)\big)\| \,.
% \end{align}

Denoting $D_N:\R^N\rightarrow\M$ a continuous, Lipschitz decoder function, constant $\gamma > 0$, then the Lipschitz $N$-width~\cite{Petrova2021} of $S \subset \M$ is 
\begin{align}\label{eq:Lipschitz-N-width}
    d^{\textnormal{Lip}}_N(S) = & \inf_{\substack{D_N,\|\cdot\| \\ \gamma\textnormal{-Lipschitz}}} \sup_{\state \in S} \inf_{\|\bar{g}\| \leq B} \| \state - D_N (\bar{g})\| \,,
\end{align}
where the final infimum is also over a suitably restricted class of norms $\|\cdot\|$ on $\M$.

%More generally~\cite{Rim2023} present a benchmark width that may be used to compare various widths in a unified framework. With the parameter manifold $\mathcal{P}$, a Banach space $L$, they let $C:\mathcal{P}\rightarrow L$ represent a quantity of interest $C(\mu) = \mathcal{l}(u_\mu)$ to be computed from a solution $u_\mu$ to the FOM. For $\R^k$ that fully encodes the trajectory $u_\mu$ at a given resolution, they abstract online and offline stages of different MOR methods to a solution operator $\text{On}_N:\mathcal{P}\rightarrow \R^k$  and an evaluation operator $\text{Off}:\R^k\rightarrow L$, respectively. %represent online and offline stages of the MOR process, valued in  To this end, 
% Then they present the benchmark nonlinear width as 
% \begin{align}\label{eq:benchmark-width-1}
%      d^\text{bench}_N(\mathcal{P}) = \inf_{\substack{\text{On}_N \\ \text{Off}_N}}\sup_{\mu \in \mathcal{P}} \|Q(\mu) - \text{Off}_N\big(\text{On}_N(\mu)\big) \|\,.
% \end{align}
% \begin{align}\label{eq:benchmark-width-2}
%      d^\text{bench}_N(\mathcal{P}) = \inf_{\substack{\text{On}_N \\ \text{Off}_N}}\sup_{\mu \in \mathcal{P}} \|Q(\mu) - \text{Off}_N\big(\text{On}_N(\mu)\big) \|\,.
% \end{align}
% -> Goes to discussion:$ this width encompasses ours/ allows benchmarking it against others.
% We relegate a brief treatment to Section~\ref{ssec:group-N-width}. 

In~\cite{Engwer2025} the Sectional Komogorov $N$-width is presented. Denote by $\sigma:\mathcal{P}\rightarrow\M$, $\sigma \in \Gamma$ a section of a fiber bundle with the parameter manifold $\mathcal{P}$ as the base space, and the state manifold $\mathcal{M}$ as the fiber. Then the linear Sectional $N$-width measures distance to an optimally chosen linear sub-bundle $\Gamma_N$: 
\begin{align}\label{eq:sectional-N-width}
    d^{\textnormal{Sec}}_N(S) = \inf_{\substack{\Gamma_N \subset \Gamma\;\\ \dim \Gamma_N \leq N \\ \textnormal{linear}}} \sup_{\mu \in \mathcal{P}} \inf_{\sigma_N \in \Gamma_N} \| \sigma(\mu) - \sigma_N(\mu)\| \,.
\end{align}
Non-linear extensions in the same work~\cite{Engwer2025} emphasize that $\M$ can be a manifold, in which case $\Gamma_N$ is a sub-bundle for which the fiber $\sigma_N(\mu)$ is a submanifold of $\sigma(\mu)$. %The fiber bundle over $\mathcal{P}$ is also allowed to be non-trivial, i.e., with the total space does not have to be $\mathcal{P} \times \M$ but only looks like that locally. 

%More generally~\cite{Rim2023} present a benchmark for various widths, that may be used to compare them. To this end, they present the form 
% \begin{align}\label{eq:benchmark-width}
%     d^\text{b}_N(\mathcal{P}) = 
% \end{align}
% -> Goes to discussion: this width encompasses ours/ allows benchmarking it against others.

% Adaptive basis approaches~\cite{Peherstorfer2022,Peherstorfer2020} are subject to the Kolmogorov $(N,M)$ and $(N,M,L)$-widths in~\cite{Rim2020}. For $u(x)\in \R$ a function on a 1D spatial domain $x \in \R$, and $T^l:\R\rightarrow\R$ 

%Adaptive basis approaches~\cite{Peherstorfer2022,Peherstorfer2020} have also been shown to be subject to a different error bound, presented as the Kolmogorov $(N,M)$-width in~\cite{Rim2020}

\section{Problem formulation}\label{sec:problem_formulation} % possibly with little overview figure highlighting the main concept, $\state(t) = \Phi(g(t),\state_0)$
%Describe goal, which is to generalize MOR from submanifolds to distributions, to describe a dynamical system restricted to a suitable distribution by means of group theory, and to describe non-intrusive methods for MOR using group theory.
Assume the FOM to be the dynamic system on $\M$:
\begin{equation}\label{eq:FOM}
    \dot{\state} = X_\mu(\state)\,,\; \state(0) = \state_{\mu,0}\,,
\end{equation}
with parameters $\mu \in \mathcal{P}$ parameterizing both the vector field $X_\mu\in\Gamma(T\M)$ and initial conditions $\state_{\mu,0}$. Solutions of~\eqref{eq:FOM} will be denoted $\state(t)$, or $\state_\mu(t)$ to emphasize the parameter-dependence. Let a set of solution snapshots be given as
\begin{equation}\label{eq:solution-snapshots}
    S = \{\state_\mu(t) \,|\, t \in [0,T], \mu \in \mathcal{P} \}\,.
\end{equation}
We will mention it explicitly when finite or countably infinite snapshot sets $S$ are considered, or when $S \subset \M \times \mathcal{P} \times [0,T]$ by abuse of notation.

%with $\state_\mu(t) \in S_\state := \pi_\M(S)$. 
The goal of this paper is to find a Lie group $G$ with dimension $\dim G \ll \dim \M$, action $\Phi:G\times\M\rightarrow\M$, and dynamics of $g_\mu(t)$ such that
\begin{equation}\label{eq:MorLie-trajectory-approximation}
    \state_\mu(t) \approx \Phi(g_\mu(t), \state_{\mu,0}) %g_\mu(t) \cdot \state_{\mu,0} % \Phi_{g_\mu(t)}(\state_{\mu,0})
\end{equation}
for $\state_\mu(t) \in S$. In particular, we aim to answer the questions:
\begin{enumerate}
    \item How can a dynamics for $g_\mu(t)$ be chosen and optimized? In Section~\ref{ssec:MORLie_Theory}.
    \item What is a theoretical error bound for the resulting methods? In Section~\ref{ssec:group-N-width}.
    \item What are choices for $G$, $\Phi$ such that dynamics on $G$ can be evaluated more efficiently than the FOM dynamics? In Section~\ref{sec:optimization}.
    \item How does this approach relate to submanifold methods? In Appendix~\ref{sec:manimor}.
\end{enumerate}

\section{MOR via Lie groups}\label{sec:morlie}

% Describe basic theory of MorLie. %, including

% \subsection{Basic theory}
% Describe how choice of $(G,\Phi,\rho)$ induces a reduced order model (ROM) on a Lie group. Include also product reductions, i.e., combining $(G,\Phi_G,\rho_G)$ and $(H,\Phi_H,\rho_H)$ to $(G\times H,\Phi_{G\times H},\rho_{G\times H})$ when $\Phi_G$ and $\Phi_H$ commute.

% \subsection{The group Kolmogorov $N$-width}
% Definition of the group Kolmogorov $N$-width, showing that it is upper-bounded by the Kolmogorov $N$-width.

% \subsection{Non-intrusive methods}
% Non-intrusive methods for optimizing $(G,\Phi,\rho)$: velocity-based and velocity-free cases.

We describe MORLie and introduce the group Kolmogorov $N$-width. 

\subsection{ROM dynamics}\label{ssec:MORLie_Theory}

% We begin by defining the class of full order systems under consideration. 

% \begin{definition}[Full order system]
%     We call the $n$-dimensional manifold $\M$ the \textbf{full order manifold} and $X \in \Gamma(T\M)$ the \textbf{full order model} (FOM). Then $X$ defines a dynamical system as 
%     \begin{equation} \label{eq:FOM}
%         \dot{\state} = X(\state)\,,\; \state(0) = \state_0\,, 
%     \end{equation}
%     and solutions are denoted by $\state(t)$.
% \end{definition}

% This class includes ordinary differential equations for the case of finite $n$, and parabolic and hyperbolic partial differential equations for infinite $n$. We restrict our attention to smooth dynamics, i.e., without discontinuities. % remark how this is meant to include PDEs like the nonlinear Burger's equation?

%Given the FOM in equation~\eqref{eq:FOM}, we first describe how it is projected to an approximate vector field.
%We define MOR via Lie groups as follows:

We investigate the choice of dynamics for $g_\mu(t)$ in~\eqref{eq:MorLie-trajectory-approximation} to arrive at a description of the ROM dynamics in MORLie. 

We want to choose the Lie group dynamics such that the evolution of $\Phi(g_\mu(t), \state_{\mu,0})$ closely follows (w.r.t.\ a to-be-defined metric) the full-order state $\state_\mu(t)$. Since the FOM dynamics $X_\mu(\state_\mu)$ vary with $\state_\mu$, the optimal Lie group dynamics should also vary with $\state_\mu$. We encode this by a map $\rho_\mu:\M\rightarrow\g$, and immediately state the main technical result of this subsection: the induced dynamics on $\M$ generated by dynamics on $G$. The key idea is that a vector in the Lie algebra produces a vector field on $\M$ via the infinitesimal generator.  %infinitesimal generator $X_{\widetilde{A}_\mu}$ should generally also vary with $\state_\mu$. 
%We make this explicit by defining $\widetilde{A}_\mu(g_\mu,t) = \rho_\mu(\appstate_\mu)$. 
%
%We immediately state the main technical result:
\begin{theorem}[MorLie reduction, reconstruction and induced dynamics]\label{thm:MorLie reconstruction}
    Given a Lie group $G$, a left action $\Phi:G\times\M\rightarrow\M$ and a map $\rho_\mu:\M\rightarrow\g$. Define $\appstate_\mu(t) \in \M$ by 
    \begin{equation}\label{eq:MorLie-Reconstruction} %\label{eq:appstate}
        \appstate_\mu(t) := \Phi(g_\mu(t),\state_{\mu,0})\,,
    \end{equation} 
    and define dynamics on $G$ to follow $\rho_\mu(\appstate_\mu) \in\g$: 
    \begin{equation}\label{eq:MorLie-group-dynamics} %\label{eq:dynamics-prior-G}
        \dot{g}_\mu = {R_{g_\mu}}_* \rho_\mu\big(\appstate_\mu\big)\,,\; g(0) = e \,.
    \end{equation}
    Then% = g_\mu(t)\cdot\state_{\mu,0}$. Then
    \begin{equation}\label{eq:appstate-dynamics}
        \frac{\ext}{\ext t} \appstate_\mu(t) = 
        X_{\rho_\mu\big(\appstate_\mu(t)\big)}\big(\appstate_\mu(t)\big)\,,
        %X_{\widetilde{A}_\mu(g_\mu,t)}\big(\appstate_\mu(t)\big)\,,
    \end{equation} 
    where $X_{\rho_\mu(\appstate_\mu)} \in \Gamma(\Delta)$ is the infinitesimal generator of $\rho_\mu(\appstate_\mu)$ w.r.t.\ $\Phi$ (cf.\ Def.~\ref{def:infinitesimal_generator}), and $\Delta\subseteq T\M$ is the induced distribution (cf.\ Def.~\ref{def:induced_distribution}). Equation~\eqref{eq:appstate-dynamics} also holds when $\Phi$ is a right action, if the left-translation were used in~\eqref{eq:MorLie-group-dynamics}. 
\end{theorem}
\begin{proof}
    %Denote $\widetilde{A}_\mu(g_\mu) = \rho_\mu(\appstate_\mu) \in\g$~--~this is well-defined: dependence on $\appstate_\mu$ on the LHS is fully encoded by $\mu$. 
    For a left action $\Phi$, differentiation of $\appstate_{\mu}(t)$ yields:
    \begin{align}
        \frac{\ext}{\ext t} \appstate_\mu(t) 
        %& = \frac{\ext}{\ext t} \big(g_\mu(t) \cdot \state_{\mu,0}\big) 
        &= \frac{\ext}{\ext t} \Phi(g_\mu(t), \state_{\mu,0}) \\ 
        & = \frac{\ext}{\ext s} \Phi(e^{\rho(\appstate_\mu(t)) s} g_\mu(t), \state_{\mu,0}) \\
        & = \frac{\ext}{\ext s} \Phi(e^{\rho(\appstate_\mu(t)) s}, \Phi_{g_\mu(t)} \state_{\mu,0}) \\
        %& = X_{\widetilde{A}_\mu(g_\mu)}(\Phi_{g_\mu}(\state_{\mu,0})) \\ 
        & = X_{\rho_\mu(\appstate_\mu(t))}\big(\appstate_\mu(t)\big) 
    \end{align}
    The second equality uses that $g_\mu(s)$ in~\eqref{eq:MorLie-group-dynamics} is tangent to $e^{\rho(\appstate_\mu(t)) s}g_\mu(t)$ at $t = s$. The third equality uses~\eqref{eq:group_homomorphism}, using that $\Phi$ is a left action. The fourth equality uses~\eqref{eq:inf_generator}.
    For a right action $\Phi$, and if ${R_{g_\mu}}_*$ in~\eqref{eq:MorLie-group-dynamics} were replaced by ${L_{g_\mu}}_*$: then $g_\mu(s)$ would be tangent to $g_\mu(t)e^{\rho(\appstate_\mu(t))s}$ at $t = s$, in the second equality, and the third equality would use~\eqref{eq:group_antihomomorphism}. 
\end{proof}

We formally define further terms inspired by Theorem~\ref{thm:MorLie reconstruction}:
\begin{definition}[MorLie reduction \& approximated dynamics]\label{def:AFOM}
    Given the FOM $(\M,X_\mu)$, and a tuple $(G,\Phi,\rho_\mu)$ of a Lie group $G$, a group action $\Phi:G\times \M \rightarrow \M$ and $\rho_\mu:\M\rightarrow \g$. Then we call $\rho_\mu$ the \textbf{reduced vector field}, and $\appX \in \Gamma(\Delta) \subseteq \Gamma(T\M)$ 
    \begin{equation}
        \appX_\mu(\state):= X_{\rho_\mu(\state)}(\state) \,,
    \end{equation}
    the \textbf{approximated dynamics}. This defines the dynamical system %here $X_{\rho_\mu(\state)}(\state)$ is the infinitesimal generator of $\rho_\mu(\state) \in \g$ w.r.t.\ $\Phi$ evaluated at $\state \in \M$. The approximated dynamics define the dynamical system
    \begin{equation}\label{eq:AFOM}
        \dot{\appstate}(t) = \appX_\mu(\appstate)\,,\; \appstate(0) = \state_{\mu,0} \,.
    \end{equation}
    We call $\appstate(t)$ (cf.~\eqref{eq:MorLie-Reconstruction}) the \textbf{reconstructed solution}.
\end{definition}
Thus, $\rho$ defines dynamics on $G$, which induce approximated dynamics $\appX$ on $\M$, constrained to the orbits $\mathcal{O}(\state)$ and tangent to the distribution $\Delta$.

A reduced order model is obtained when $\dim G < \dim \M$: then Theorem~\eqref{thm:MorLie reconstruction} lets us solve a low dimensional dynamical system on the Lie group $G$ to compute solutions $\appstate_\mu$ of the high-dimensional approximated dynamics $\appX$ on $\M$. We call~\eqref{eq:MorLie-group-dynamics} a family of \textbf{reduced order models on the Lie group} $G$, since they vary with initial state $\state_{\mu,0}$ and parameter $\mu$. This presents a formulation for MOR on manifolds that allows to approximate solutions globally, where different $\state_{\mu,0}$ can lead to different dynamical systems on $G$. Distinct ROMs can in principle also be qualitatively different, i.e., having different stability properties, number of equilibria or periodic orbits. 

Theorem~\ref{thm:MorLie reconstruction} shows that the evolution of $\appstate_\mu(t)$ in~\eqref{eq:MorLie-group-dynamics} follows an approximated dynamics $\appX_\mu \in \Gamma(\Delta)$, which is restricted to lie within the distribution $\Delta \subseteq T\M$ induced by $G,\Phi$ (see Definition~\ref{def:induced_distribution}). Theorem~\ref{thm:MorLie reconstruction} can alternatively be interpreted as a choice of a family of approximate dynamics $\appX_\mu \in \Gamma(\Delta)$ that admit a family of ROMs of the form~\eqref{eq:MorLie-group-dynamics}.

We emphasize that~\eqref{eq:MorLie-group-dynamics} represents any autonomous first-order dynamics $\dot{g}_\mu = f(g_\mu)$ without loss of generality: the choice of dynamics $\rho_\mu(\appstate_\mu) \in\g$ can be written as $\widetilde{A}_\mu(g_\mu) := \rho_\mu(\appstate_\mu) \in\g$, hiding dependence on $\state_{\mu,0}$ on the RHS in dependence on $\mu$ on the LHS.\@ Then $\widetilde{A}_\mu(g_\mu) = {R_{g^{-1}}}_*f(g_\mu)$ implements arbitrary autonomous first-order dynamics.

Theorem~\ref{thm:MorLie reconstruction} also shows that the approximated solution $\appstate(t) \in \mathcal{O}(\state_0)$ is restricted to lie in the orbit $\mathcal{O}(\state_0) \subseteq \M\,$ (see Definition~\ref{def:orbit}). The following Section uses this fact to define the group Kolmogorov $N$-width.

% An immediate question is about the computational efficiency of such a model order reduction, since the reduced vector field $\rho$ and the action $\Phi$ still work with the full order manifold $\M$. For large $n$, this will be tackled by approximating states on the manifold $\M$ with a map 
% \begin{equation}
%         \pi: \M \rightarrow \bar{\M} \; \textnormal{s.t.} \,.
% \end{equation}
% that is separate from $(G,\Phi,\rho)$, defining an action $\bar{\Phi}:G\times\bar{\M}\rightarrow\bar{\M}$ such that 
% \begin{equation}
%     \bar{\Phi}(g,\cdot)\circ \pi := \pi \circ \Phi(g,\cdot)\,.
% \end{equation}
% and defining $\bar{rho}:\bar{\M}\rightarrow\g$ such that 
% \begin{equation}
%     \rho = \bar{\rho}\circ\pi \,.
% \end{equation}
%I.e., the approximation of $\M$ is treated separately from the approximation of the dynamics on $\M$.

\subsection{The group Kolmogorov \textit{N}-width}\label{ssec:group-N-width}
Suppose that $\M$ is a metric space equipped with a distance $\textnormal{dist}:\M\times\M\rightarrow\R$, and given the set $S \subseteq \M$ representing solution snapshots (cf.~\eqref{eq:solution-snapshots}). Define $\Xi_N$ as a set of pairs $(G,\Phi)\in\Xi_N$ of $N$-dimensional Lie groups $G$ and actions $\Phi:G\times\M\rightarrow\M$, henceforth called admissible pairs. Then we define:
\begin{definition}[Kolmogorov $\Xi_N$-width]
    The Kolmogorov $\Xi_N$-width of a set $S\subseteq\M$ relative to an arbitrary point $\state_0\in S$ is
    \begin{equation}\label{eq:group-width}
        d_{\Xi_N}(S,\state_0) := \inf_{\substack{(G,\Phi) \in \Xi_N}}\; \sup_{\state \in S}\; \inf_{y \in \mathcal{O}(\state_0)} \textnormal{dist}(\state,y) \,.
    \end{equation} 
    For $S$ a set of solution snapshots~\eqref{eq:solution-snapshots}, and denoting $\state_\mu([a,b]):=\{\state_\mu(t)\;|\; t\in[a,b]\}$, we define the Kolmogorov $\Xi_N$-width of $S$ over the initial conditions $\state_{0,\mu} \in S$ as
    \begin{equation}\label{eq:group-width-over-IC}
        d^T_{\Xi_N}(S) := \sup_{\mu \in \mathcal{P}} d_{\Xi_N}\big(\state_{\mu}([0,T]),\state_{0,\mu}\big)\,,
    \end{equation}
    and the Kolmogorov $\Xi_N$-width of $S$ over a time-horizon $\tau \leq T$ as
    \begin{equation}\label{eq:group-width-over-tau}
        d^{T,\tau}_{\Xi_N}(S) := \sup_{
            \substack{\mu \in \mathcal{P},\\ 
                        t \in [0,T-\tau]}}
         d_{\Xi_N}\big(\state_{\mu}([t,t+\tau]),\state_{\mu}(t)\big)\,.
    \end{equation}
\end{definition}
Intuitively, the group width measures how well solution trajectories can be approximated by a single group orbit through the initial state.

In particular, $d_{\Xi_N}(S,\state_0)$ is the worst case error over the orbit $\mathcal{O}(\state_0)$, for the best choice of admissible pair $(G,\Phi) \in \Xi_N$. In turn, $d^T_{\Xi_N}(S)$ formalizes the error between FOM-trajectories $\state_\mu(t)$ and ROM-trajectories $\appstate_\mu(t)$ (cf.~\eqref{eq:MorLie-trajectory-approximation}), returning the worst case over all initial conditions. Finally, $d^{T,\tau}_{\Xi_N}(S)$ measures a worst-case error between $\state_\mu(t)$ and $\appstate_\mu(t)$ over a time-horizon $\tau$, and without restricting itself to ROM-trajectories starting at $x_{\mu,0}$. Alternatively, we can interpret $\frac{1}{\tau}d^{T,\tau}_{\Xi_N}(S)$ as a bound on the increase of $\textnormal{dist}\big(\bar{\state}(t),\state(t)\big)$ per $\tau$. We note that Definitions~\eqref{eq:group-width-over-IC} and~\eqref{eq:group-width-over-tau} coincide for $\tau = T$ ($d^{T}_{\Xi_N}(S) = d^{T,T}_{\Xi_N}(S)$), in which case the supremum over $t \in [0, T - \tau]$ in~\eqref{eq:group-width-over-tau} is over a single point $t = 0$. 
% \yannik{This can be extended further: taking the max of $d_{\Xi_N}(\state_{\mu}(\R),\state_{0,\mu})$ over $\mu\in\mathcal{P}$ gives a notion of error for all orbits $\state_{\mu}(\R)\in S$, that takes into account the free initial condition $\state_{0,\mu}$ in the reconstruction~\eqref{eq:MorLie-trajectory-approximation}.}

If $\Xi_N$ is not restricted, e.g., $\Xi_1$ contains all possible pairs of one-dimensional Lie groups and compatible actions, then $d^{T,\tau}_{\Xi_1}(S) = 0$ irrespective of $S$: this can be seen for $G = (\R^1,+)$ and $\Phi(a,\state) = \Psi^a_X(\state)$ with $\Psi^t_X:\M\rightarrow\M$ the flow of $X \in \Gamma(TM)$. This exactly recovers the FOM trajectories, but does not result in a useful ROM, as this action is computationally intense to evaluate.

Similarly, if $y$ were not restricted to $y \in \mathcal{O}(\state_0)$ then $d_{\Xi_N}(S,S) = 0$, i.e., it would be possible to pick $y = \state$ for any $\state\in S$. 

%The Kolmogorov $N$-width~\eqref{eq:Kolmogorov-N-width} can be recovered from~\eqref{eq:group-width}:
\begin{proposition}
    Assume that $\M$ is a Hilbert space with $\textnormal{dist}(x,y) = \|x - y\|$, define  
    \begin{align}
        \Xi^{\textnormal{Vec}}_N := \{&(\R^N,\Phi)\;|\; \Phi(g,x) := x + g^i e_i\,, \\ 
        & \{e_1,\cdots,e_N\} \subseteq \M \textnormal{ lin. indep.} \}\,. \nonumber
    \end{align} 
    Then we recover the Kolmogorov $N$-width~\eqref{eq:Kolmogorov-N-width} from~\eqref{eq:group-width}
    \begin{equation}
        d_N(S) = d_{\Xi^\textnormal{Vec}_N}(S,0)\,.
    \end{equation}
\end{proposition}
\begin{proof}
    Different orbits $\mathcal{O}(\state_0)$ (cf. Def.~\ref{def:orbit}) generated by $(\R^N,\Phi) \in \Xi^{\textnormal{Vec}}_N$ correspond to $N$-dimensional affine subspaces through $\state_0$:
    \begin{equation}
        \mathcal{O}_{(\R^N,\Phi)}(\state_0) = \{\state_0 + g^i e_i \;|\; g \in \R^N \}\,.
    \end{equation}
    Thus, the orbits $W = \mathcal{O}_{(\R^N,\Phi)}(0)$ correspond to $N$-dimensional subspaces. Further, given any $N$-dimensional subspace $W$ spanned by $\{e_1,\cdots,e_N\}\subseteq W$, there is $(\R^N,\Phi)\in\Xi^{\textnormal{Vec}}_N$ such that $W = \mathcal{O}_{(\R^N,\Phi)}(0)$. Hence, the first infimum in $d_{\Xi^\textnormal{Vec}_N}(S,0)$ %Eq.~\eqref{eq:group-width} 
    is equivalent to an infimum over subspaces $W \subseteq \M$ such that $\dim W = N$, and the second infimum is over elements $v \in W = \mathcal{O}(0)$:
    \begin{align}
        d_{\Xi^\textnormal{Vec}_N}(S,0) 
        & = \inf_{\substack{(G,\Phi) \in \Xi^{\textnormal{Vec}}_N}}\; \sup_{\state \in S}\; \inf_{y \in \mathcal{O}(0)} \textnormal{dist}(\state,y) \\
        & = \inf_{\substack{W \subseteq \M  \\ \dim W = N}}\; \sup_{\state \in S}\; \inf_{y \in W} \|\state - y\| \,.
    \end{align}
\end{proof}

For $\M$ a vector space, including linear, affine or even nonlinear transformations is possible by considering different choices of $\Xi_N$, such as 
\begin{align}
    \Xi^\textnormal{Lin}_N &:= \{(G,\Phi) \; | \;  \dim G = N\,,\; \Phi \textnormal{ linear} \} \,,\\
    \Xi^\textnormal{Aff}_N &:= \{(G,\Phi) \; | \;  \dim G = N\,,\; \Phi \textnormal{ affine} \} \,,\\
    \Xi^\textnormal{Pr}_N &:= \{(G,\Phi) \; | \;  \dim G = N\,,\; \Phi \textnormal{ proper} \} \,.
\end{align}
In particular, linear $\Phi$ are linear representations % $\Phi:G\rightarrow\textnormal{Aut}(\M)$, 
and affine $\Phi$ are affine representations %$\Phi:G\rightarrow\textnormal{Aff}(\M)$, 
such that the $N$-dimensional Lie groups in $\Xi^\textnormal{Lin}_N$ and $\Xi^\textnormal{Aff}_N$ can always be associated with $N$-dimensional subsets of general linear group $GL(n,\R)$, general affine group $\textnormal{Aff}(n,\R)$ respectively, with $n = \dim \M$.
It can also be shown that $\Xi^\textnormal{Pr}_N$ includes $\Xi^{\textnormal{Vec}}_N\,,\,\Xi^\textnormal{Lin}_N\,,\,\Xi^\textnormal{Aff}_N$ as subsets, so it follows that 
\begin{equation}
    d_{\Xi^\textnormal{Pr}_N}(S,0) \leq d_N(S)\,.
\end{equation}
Therefore, group methods can break the Kolmogorov barrier. The set $\Xi^\textnormal{Aff}_N$ includes $\Xi^\textnormal{Lin}_N$ and $\Xi^{\textnormal{Vec}}_N$ as (non-intersecting) subsets, so it also follows that
\begin{align}
    %& d_{\Xi^\textnormal{Pr}_N}(S,0) \leq 
    d_{\Xi^\textnormal{Aff}_N}(S,0) \leq d_{\Xi^\textnormal{Lin}_N}(S,0)\,,\\
    %& d_{\Xi^\textnormal{Pr}_N}(S,0) \leq 
    d_{\Xi^\textnormal{Aff}_N}(S,0) \leq d_{\Xi^\textnormal{Vec}_N}(S,0)\,.
\end{align}
Yet, affine MOR  only provides a marginal advantage over linear MOR since $\Xi^\textnormal{Lin}_{N+1}$ includes $\textnormal{Aff}_{N}$, such that $d_{\Xi^\textnormal{Lin}_{N+1}}(S,0) \leq d_{\Xi^\textnormal{Aff}_{N}}(S,0)$. Sections~\ref{sec:analytic_examples},~\ref{sec:numeric_examples} will show more significant examples of breaking the Kolmogorov barrier by means of group methods, and Section~\ref{sec:analytic_examples} provides a specific group Kolmogorov $N$-width of linear transport.
%As such, we investigate  as a means of breaking the Kolmogorov barrier.
% Unless further specified, we consider
% \begin{equation}
%     \Xi_G := \{\Phi:G\times\M\rightarrow\M \; | \; \Phi \textnormal{ is free, proper} \}
% \end{equation}
\begin{remark}
    It is interesting to restrict $\Xi_N$ to pairs $(G,\Phi)$ such that the $\Phi$ is of a low computational complexity. Possible choices are the aforementioned linear and affine transformations, but also nonlinear actions or flows of differential equations could be parameterized by shallow neural networks to provide interesting classes. %We also investigate this briefly in Section~\ref{sec:examples}.
\end{remark} 
\begin{remark}
    When the action of a pair $(G,\Phi) \in \Xi_N$ is transitive, i.e., $\mathcal{O}(\state_0) = \M$, we likewise get that $d_{\Xi_N}(S,\state_0) = 0$. However, for proper and transitive actions (ignoring improper transitive actions resulting from topologically transitive flows) we have $\dim G > \dim \M$, which makes the case uninteresting for MOR, but leads to the interesting class of Lie group methods for integration of lifted (rather than reduced) dynamics on $G$. 
\end{remark}

% Note on computational complexity of particular $\Xi_N$, and in particular the action $\Phi$?

% For Banach spaces, we will highlight in Section~\ref{sec:PDEs} that $d_N(S,0) \leq d_N(S)$ independent of $S$ for an appropriate choice of $\Xi_N,\Xi_G$. This is because the Kolmogorov $N$-width is a special case of the group Kolmogorov $N$-width, obtained when $\mathcal{O}(0)$ are $N$-dimensional subspace.

% For the linear transport equation, we will highlight that $d_1(S,\state_0) = 0$, for simple choices of $\Xi_N$ and $\Xi_G$.

The Lipschitz $N$-width~\eqref{eq:Lipschitz-N-width} can also be recovered from~\eqref{eq:group-width}:

\begin{proposition}
    Assume that $\M$ is a Hilbert space with $\textnormal{dist}(x,y) = \|x - y\|$, $T^N = S^1 \times \cdots \times S^1$ is the $N$-torus. Let $D_N:\R^N \rightarrow \M$ denote the decoder function in~\eqref{eq:Lipschitz-N-width}, with constant $\gamma > 0$ and $\|\cdot\|$ from a suitably restricted class of norms on $\M$. Also construct a discontinuous embedding $\textnormal{Emb}:T^N \rightarrow \{x \in \R^N \;|\; \|x\|\leq B\}$ mapping $N$ copies of $S^1$ to $N$ copies of the half-open interval $(-B,B] \in \R$. Define  
    \begin{align}\label{eq:Xi-Lip}
        \Xi^{\textnormal{Lip}}_N := \{&(T^N,\Phi)\;|\; \Phi(g,x) := x+D_N(\textnormal{Emb}(g))\,, \\ 
        & D_N \textnormal{is } \gamma\textnormal{-Lipschitz} \}\,. \nonumber
    \end{align}
    Then we recover the Lipschitz $N$-width~\eqref{eq:Lipschitz-N-width} from~\eqref{eq:group-width}
    \begin{equation}
        d^{\textnormal{Lip}}_N(S) = \inf_{\|\cdot\|_{\M}} d_{\Xi^\textnormal{Lip}_N}(S,0)\,.
    \end{equation}
\end{proposition}
\begin{proof}
    The orbit $\mathcal{O}(0)$ is 
    \begin{equation}
        \mathcal{O}_{(T^N,\Phi)}(0) = D_N(\textnormal{Emb}(T^N)) \subset \M \,.
    \end{equation}
    Hence, the first infimum in $d_{\Xi^\textnormal{Lip}_N}(S,0)$ %Eq.~\eqref{eq:group-width}
    is 
    \begin{align}
        \inf_{y \in \mathcal{O}(0)}\| x - y\| &= \inf_{g \in T^N}\| x - D_N(\textnormal{Emb}(g))\| \\ &= \inf_{\|\bar{g}\|\leq B}\| x - D_N(\bar{g})\| \,,
    \end{align} 
    where the final step holds since the image of $\textnormal{Emb}$ is dense in the closed ball $\{x\;|\; \|x\| \leq B \}$, and we recover
    \begin{align}
        \inf_{\|\cdot\|} d_{\Xi^\textnormal{Lip}_N}(S,0)
         & = \inf_{\substack{(G,\Phi) \in \Xi^{\textnormal{Lip}}_N \\ \|\cdot\|}}\; \sup_{\state \in S}\; \inf_{y \in \mathcal{O}(0)} \textnormal{dist}(\state,y) \\
         & = \inf_{\substack{D_N, \|\cdot\| \\
         \gamma\textnormal{-Lipschitz}}}\; \sup_{\state \in S} \inf_{\|\bar{g}\| \leq B} \| x - D_N(\bar{g}) \| \,,
    \end{align}
    which is equivalent to~\eqref{eq:Lipschitz-N-width}.
\end{proof}
Here, the Lie group formalism is taken very loosely: the map $\Phi(x,g) = x+D_N(\textnormal{Emb}(g))$ in~\eqref{eq:Xi-Lip} is not explicitly restricted to be a homomorphism (cf. Def.~\ref{def:group_action}), with the result that properties of the orbits (cf. Def.~\ref{thm:properties_of_orbit}) and induced distribution (cf. Def.~\ref{thm:properties_of_induced_distribution}) are no longer guaranteed. The resulting orbits may self-intersect, thus overparameterizing $\M$~--~this problem would be avoided if the homomorphism property can be enforced.

A ROM in MORLie can also be evaluated by means of the Sectional $N$-width~\eqref{eq:sectional-N-width}. 
\begin{proposition}
    Assume that $\M$ is a Hilbert space with $\textnormal{dist}(x,y) = \|x - y\|$. Let $\Phi_{\mu}(g,x)$ be a parameter-dependent action. Define  
    \begin{align}
        \Xi^{\textnormal{Sec}}_N := \{&(\R^N,\Phi_{\mu})\;|\; \Phi_{\mu}(g,x) := x + g^i e_{\mu,i}\,, \\ 
        & \{e_{\mu,1},\cdots,e_{\mu,N}\} \subseteq \M \textnormal{ lin. indep.} \}\,. \nonumber
    \end{align}
    Any such action $\Phi_{\mu}$ induces a linear section 
    \begin{equation}
        \sigma_N(\mu) = \mathcal{O}_{(\R^N, \Phi_\mu)}(x_\mu) \,.
    \end{equation} 
    \begin{equation}
        d^{\textnormal{Sec}}_N(S) = d_{\Xi^\textnormal{Lip}_N}(S,0)\,.
    \end{equation}
\end{proposition}
I.e., a parameterized linear action gives rise to a linear section $\sigma_N(\mu)$ with varying linear fiber. Similarly, nonlinear sections can be induced by general parameterized actions.

% \begin{proposition}
%     Assume that $\M$ is a Hilbert space with $\textnormal{dist}(x,y) = \|x - y\|$, let $E_N:\M\rightarrow\R^N$, $D_N:\R^N \rightarrow \M$ denote the encoder and decoder functions in~\eqref{eq:mfd-N-width}, and assume that $\inf_{g \in \R^N}\| x - D_N(g)\| = \| x - D_N(E_N(x)) \|$. Define  
%     \begin{align}
%         \Xi^{\textnormal{mfd}}_N := \{&(\R^N,\Phi)\;|\; \Phi(g,x) := x+D_N(g)\,, \\ 
%         & \}\,. \nonumber
%     \end{align}
%     Then we recover the manifold-width~\eqref{eq:mfd-N-width} from~\eqref{eq:group-width}
%     \begin{equation}
%         d^{\text{mfd}}_N(S) = d_{\Xi^\textnormal{mfd}_N}(S,0)\,.
%     \end{equation}
% \end{proposition}
% \begin{proof}
%     The orbit $\mathcal{O}(0)$ is 
%     \begin{equation}
%         \mathcal{O}_{(\R^N,\Phi)}(0) = D_N(\R^N) \subset \M \,.
%     \end{equation}
%     Hence, the first infimum in $d_{\Xi^\textnormal{mfd}_N}(S,0)$ %Eq.~\eqref{eq:group-width}
%     is 
%     \begin{equation}
%         \inf_{y \in D_N(\R^N)}\| x - y\| = \inf_{g \in \R^N}\| x - D_N(g)\| = \| x - D_N(E_N(x)) \| \,.
%     \end{equation} 
%     and we recover
%     \begin{align}
%         d_{\Xi^\textnormal{mfd}_N}(S,0) 
%         & = \inf_{\substack{(G,\Phi) \in \Xi^{\textnormal{mfd}}_N}}\; \sup_{\state \in S}\; \inf_{y \in \mathcal{O}(0)} \textnormal{dist}(\state,y) \\
%         & = \inf_{\substack{D_N}}\; \sup_{\state \in S} \| x - D_N(E_N(x)) \| \,,
%     \end{align}
%     which is equivalent to~\eqref{eq:mfd-N-width} since the condition $\inf_{g \in \R^N}\| x - D_N(g)\| = \| x - D_N(E_N(x)) \|$ removes dependence on $E_N$.
% \end{proof}

\section{Optimization in MORLie}\label{sec:optimization}\label{ssec:MORLie_constructive}
We want to choose $(G,\Phi,\rho_\mu)$ to arrive at ROM dynamics that can be evaluated more efficiently than the FOM dynamics, and enable reconstructed solutions to approximate FOM solutions (cf.~\eqref{eq:MorLie-trajectory-approximation}).
To this end, intrusive and non-intrusive methods for MORLie are presented, also preparing for the examples in Sections~\ref{sec:analytic_examples},~\ref{sec:numeric_examples}. In both cases we cast the process of determining $(G,\Phi,\rho_\mu)$ for a given $(\M,X_\mu,S)$ into the form of an optimization problem. %We briefly describe how to optimize $(G,\Phi,\rho_\mu)$ for a given $(\M,X_\mu,S)$. 
% determining rho by optimization, data based learning for phi, t, determining G

Let $\theta \in \R^{n_\theta}$ denote coefficients of a parameterization $\rho:\R^{n_\theta}\times\mathcal{P}\rightarrow C^\infty(\M,\g)$, i.e., $\rho_{\theta,\mu}:\M\rightarrow\g$ describes a subset $\Upsilon\subseteq C^\infty(\M,\g)$ of reduced vector fields. Similar to Section~\ref{ssec:group-N-width}, let $\Xi$ be a to-be-determined set of pairs $(G,\Phi)$, called admissible pairs, and let $J$ be a to-be-determined real-valued cost-function. The general form of the optimization problem will be:
\begin{equation}\label{eq:optimization-problem}
    (G^*,\Phi^*,\theta^*) = \arg\min_{\substack{(G,\Phi)\in\Xi\,, \\ \theta \in \R^{n_\theta}}} J\,. %\theta \in \R^{n_\theta} % \rho_{\theta,\mu}\in \Upsilon
\end{equation}
%Here, $J$ can be a non-intrusive cost-function $J(G,\Phi,\rho, S) \in \R_+$ or an intrusive cost-function $J(G,\Phi,\rho, S, X_\mu) \in \R_+$, where the latter has access to the FOM dynamics $X_\mu$. In this section we restrict our interest to a few special cases of a sufficiently general form. 
In the following subsections we elaborate on~\eqref{eq:optimization-problem}, treating in more detail choices of $\Xi$ and $\Upsilon$, a choice of $J$ for intrusive and non-intrusive optimization of the ROM. Finally, we will present a strategy for solving~\eqref{eq:optimization-problem} for $(G^*,\Phi^*,\theta^*)$, and further dimensionanility reduction going from $(G^*,\Phi^*,\rho_{\theta,\mu}^*)$ to $(H^*,\Phi^*,\rho_{\theta,\mu}^*)$ with $H^*\subset G^*$. %by solving an algebraic problem on the Lie algebra $\g$ of the optimal admissible pair $(G,\Phi)$. % a decoupled optimization strategy that performs a discrete search over $\Xi$ and optimizes over $\Upsilon$ by gradient descent.  

\subsection{Search spaces for $(G,\Phi)$} %{Induced sets of admissible pairs}
We show how a search space $\Xi$ for $(G,\Phi)$ may be determined, for a given manifold $\M$.

A full classification of Lie group actions with particular properties is an open research issue~\cite{Hirsch2011}, with non-trivial dependence on the manifold $\M$.
For example, not every Lie group $G$ has a meaningful action on a given manifold $\M$: there is a free, faithful and proper action of $SO(3)$ on $\R^3$ but not on $T^3$. There is a free action of $SU(2)$ on $S^7$, but not on $S^5$. Also the types of possible actions $\Phi$ depend on $\M$: for example, isometric actions are only defined on Riemannian manifolds, and symplectic actions are only defined on symplectic manifolds. Linear and affine actions can only be defined when $\M$ is a vector-space, in which case actions are referred to as representations, leading to a rich and ongoing field of representation theory~\cite{Hall2015}.
We make no attempt to summarize the vast field of possible pairings of manifolds, Lie groups and actions, and relegate the interested reader to~\cite{Hirsch2011,Hall2015,Knapp1996}. 

The possible search spaces $\Xi$ for $(G,\Phi)$ non-trivially depend on $\M$, since both $G,\Phi$ non-trivially depend on $\M$.%and %$\Upsilon$ for $\rho:\mathcal{P}\rightarrow C^\infty(\M,\g)$ have to be determined on a case-by-case basis, and 
%requires knowldge about the particular manifold $\M$. 

In the following, we make minor remarks on possible choices of $\Xi$, assuming prior information of a Lie group $G$ and action $\Phi:G\times\M\rightarrow\M$ that act on $\M$. Then a set $\Xi$ of admissible Lie group and action pairs can be induced:
%A homogenous manifold is a manifold $\M$ together with a transitive group action $\widetilde{\Phi}:\widetilde{G}\times\M\rightarrow\M$ (cf. Def.~\ref{def:properties_of_action}). Subgroups of $\widetilde{G}$ induce a set of admissible pairs $\Xi$:
\begin{definition}[Induced admissible pairs]\label{def:induced_admissible_pairs}
    Given a manifold $\M$, Lie group $G$ and group action $\Phi:G\times\M\rightarrow\M$, define the set of admissible pairs induced by $(G,\Phi)$:
    \begin{equation}
        \Xi^{G,\Phi} := \{ (H,\Phi) \;|\; H \subseteq G \}\,,
    \end{equation} 
    where $H\subseteq G$ denotes a Lie subgroup. Then induced sets of admissible pairs of a fixed dimension $N$, or with additional properties (cf. Def.~\ref{def:properties_of_action}) can be defined as 
    \begin{align}
        \Xi^{G,\Phi}_N &:= \{ (H,\Phi) \;|\; H \subseteq G\,, \dim H = N \}\,. \\
        \Xi^{G,\Phi}_{\textnormal{Pr}} &:= \{ (H,\Phi) \;|\; H \subseteq G  \textnormal{ \normalfont proper subgroup} \}\,.
    \end{align} 
\end{definition}
Similar definitions allow to preserve any properties of the action, further pruning the set of subgroups and action pairs $\Xi^{G,\Phi}$. 

We consider a few special cases to apply Definition~\ref{def:induced_admissible_pairs}. For example, given $\M = \R^3$, the affine group $Aff(3)$ (see Sec.~\ref{ssec:rigid-pointcloud}) and its action on $\R^3$ induce $\Xi^{G,\Phi}$ containing also $SE(n)$, $SO(n)$, $(\R^n,+)$ (for $n \leq 3$) as proper subgroups.  %, i.e., even if $\Phi:G\times\M\rightarrow\M$ were proper, there is not guarantee that $\Phi:H\times\M\rightarrow\M$ is proper. However, the proper subgroups and group action pairs do form a subset of the $\Xi$, and further pruning is in principle possible.

For arbitrary manifolds $\M$, the diffeomorphism group $\Diff(\M)$ of smooth invertible maps $\varphi:\M\rightarrow\M$ (a topological Lie group) induces a versatile search space $\Xi^{G,\Phi}$, whose subgroups were described in terms of neural nets parameterizing vector fields, in~\cite{Kleikamp2022}.

Homogenous manifolds $\M$ are a class of manifolds for which a transitive action $\Phi:G\times\M\rightarrow\M$ is already known. Then the induced distribution $\Delta = T\M$ (cf. Def.~\ref{def:induced_distribution}) already contains the FOM dynamics $X\in \Gamma(T\M)$, and it is guaranteed that $\Xi^{G,\Phi}$ contains a group and action pair such that the ROM exactly describes the FOM. We will reuse this fact in Section~\ref{ssec:subalgebra-search}.

\subsection{Search space for reduced vector fields}
We consider two cases for parameterizing $\rho_{\theta,\mu}:\M\rightarrow\g$, corresponding to the search space $\Upsilon \subseteq C^\infty(\M,\g)$. % $$ containing :% have to be determined on a case-by-case basis
% \begin{itemize}
%     \item Maps $\rho_{\theta,\mu}:\M\rightarrow\g$ is directly parameterized in terms of component functions $f_\theta:\R^{\dim \M}\times \R^{\dim \mathcal{P}} \rightarrow \R^{\dim \g}$ in local charts on $\M$
%     \item The family of ROMs ${R_g}_*\rho_{\theta,\mu}(\Phi(g,\state_{\mu,0}))$ on $G$ is parameterized without direct reference to $\M$, in terms of component functions $f_\theta:\R^{\dim G}\times \R^{\dim \mathcal{P}} \rightarrow \R^{\dim \g}$ in local charts on $G$
% \end{itemize}

First, it is possible to directly parameterize the map $\rho_{\theta,\mu}:\M\rightarrow\g$ in local charts $(U_i, \Coord_i)$ with $U_i \subseteq \M$ and $\Coord_i:U_i\rightarrow\R^n$. To this end, let there be a finite collection of charts $(U_i,\Coord_i)$ that cover $\M$ and have smooth transition maps, and given a partition of unity $\sigma_i:\M\rightarrow\R$ with respect to the charts, i.e., functions such that $\sigma_i(\state) > 0$ for $\state \in U_i$, $\sigma_i(\state) = 0$ for $x \notin U_i$ and $\sum_i \sigma_i(\state) = 1$ for all $\state \in \M$. Further, let $\Lambda:\R^{\dim \g}\rightarrow \g$ be an invertible, linear map implementing a basis of the Lie algebra. Then $\rho_{\theta,\mu}$ can be parameterized in terms of component functions $f^i_\theta:\R^{\dim \M}\times \R^{\dim \mathcal{P}} \rightarrow \R^{\dim \g}$ as 
\begin{equation}
    \rho_{\theta,\mu}(\state) = \sum_i \Lambda\big((\sigma_i f^i_\theta)(\Coord_i^{-1}(\state),\mu)\big)\,.
\end{equation}
Component functions $f^i_\theta$ can be implemented e.g., by autoencoders, series expansions, or other parameterizations. However, direct parameterization of $\rho_{\theta,\mu}$ will require explicit computation of $\appstate = \Phi(g_\mu(t),\state_{\mu,0})$ at every time-instance, possibly incuring a large computational overhead by working with the FOM state directly when the ROM dynamics are solved. 

As an alternative, we parameterize $\rho_{\theta,\mu}$ such that the reduced-order state $g_\mu$ (rather than the full-order state $\appstate$) can be used to evaluate the ROM dynamics. Removing reference to the full-order state reduces computational overhead, and is typically referred to as hyperreduction. %We present hyperreduction as an alternative choice for parameterizing $\rho_{\theta,\mu}$, in the following. %Here, this means evaluating a $\rho_{\mu,\theta}$  working with a family of functions that is potentially lighter to evaluate.
% The dynamics on $G$, as a function of the reduced-order state $g_\mu$ are given by 
% \begin{equation}\label{eq:search-space-dynamics-G}
%     \dot{g}_\mu = {R_{g_\mu}}_*\rho_{\theta,\mu}(\Phi(g_\mu,\state_{\mu,0}))\,.
% \end{equation}
Let there be a finite collection of charts $(\widetilde{U}_i,\widetilde{\Coord}_i)$ that cover $G$ and have smooth transition maps, and given a partition of unity $\widetilde{\sigma}_i:G\rightarrow\R$ with respect to the charts. Then $\rho_{\theta,\mu}(\Phi(g_\mu,\state_{\mu,0}))$ can be parameterized in terms of component functions $f^i_\theta:\R^{\dim G}\times \R^{\dim \mathcal{P}} \rightarrow \R^{\dim \g}$ as
\begin{equation}
    \rho_{\theta,\mu}(\Phi(g_\mu,\state_{\mu,0})) = \sum_i \Lambda\big((\widetilde{\sigma}_i \widetilde{f}^i_\theta)(\widetilde{\Coord}_i^{-1}(g_\mu),\mu)\big)\,.
\end{equation}
The parameterizations are otherwise equivalent in terms of approximation-power: no information is lost by not considering $\appstate$ directly, since all relevant complexity w.r.t.\ different initial conditions $\state_{\mu,0}$ is implicit in the fixed parameter $\mu$, and the full-order state is fully determined by $g_\mu, \state_{\mu,0}$. The component functions $\widetilde{f}^i_\theta$ can be implemented similarly to $f^i_\theta$. For high-dimensional parameters $\mu$ it may also be interesting to further reduce them to their most important features, e.g., by use of autoencoders.

\subsection{Intrusive and non-intrusive MORLie}
For a first intrusive cost-function, we consider a smooth Riemannian manifold $\M$ with metric-induced norm $\| \cdot\|$ on local tangent-spaces $T_\state\M$, and metric-induced volume form. With a finite parameter set $\mathcal{P}_f\subset \mathcal{P}$ we define
\begin{align}\label{eq:cost-intrusive-integral}
    J&(G,\Phi,\rho_{\theta,\mu},X_\mu) \\
    & = \sum_{\mu_i\in\mathcal{P}_f}\int_\M \|X_{\mu_i}(\state) - X_{\rho_{\theta,\mu_i}(\state)}(\state) \| \,. \nonumber
\end{align}
In~\eqref{eq:cost-intrusive-integral}, it is the explicit dependence on $X_{\mu_i}$ that marks the cost-function as intrusive.
\begin{theorem}\label{thm:explicit-rho-intrusive}
    For fixed $G,\Phi$ the optimization problem~\eqref{eq:optimization-problem} with cost~\eqref{eq:cost-intrusive-integral} has an explicit solution $\rho^*_{\mu} \in C^\infty(\M,\g)$ given by 
    \begin{equation}\label{eq:solution-rho-cost-intrusive}
        \rho^*_{\mu}(\state) = {X_\g}^{\dagger} \Pi_\Delta X_{\mu_i}(\state)\,,
    \end{equation}
    where $\Pi_\Delta:T\M\rightarrow\Delta$ is the metric-projection onto the induced distribution $\Delta$ (cf. Def.~\ref{def:induced_distribution}) and ${X_\g}^{\dagger}:\Delta\rightarrow\g$ inverts the infinitesimal generator ($\forall{\widetilde{A}\in\g}:{X_\g}^{\dagger}X_{\widetilde{A}} = \widetilde{A}$)\footnote{The operator ${X_\g}^{\dagger} \Pi_\Delta$ is a coordinate-free version of a Moore-Penrose pseudo-inverse of $X_\g$, i.e., corresponding to algebraic inversion.}. Further, if $X_\mu$ is smooth, then so is $\rho^*_\mu$.
\end{theorem}
\begin{proof}
    The proposed explicit solution~\eqref{eq:solution-rho-cost-intrusive} minimizes the integrand $\|X_{\mu_i}(\state) - X_{\rho_{\theta,\mu_i}(\state)}(\state) \|$ for all $\state\in\M,\mu_i \in \mathcal{P}$. Thus, also~\eqref{eq:cost-intrusive-integral} is minimized. Finally, $\rho^*_\mu$ is the composition of smooth maps $X_\g^\dagger$ and $\Pi_\Delta$, and the vector field $X_{\mu}$ which is smooth by assumption, so $\rho^*_\mu$ is also smooth.
\end{proof}
The result of Theorem~\ref{thm:explicit-rho-intrusive} extends to $\theta$ dependent $\rho_{\theta,\mu}$ if there are $\theta^*$ such that $\rho_{\theta^*,\mu}= \rho_\mu^*$. 

For a discrete implementation consider a finite set of state-parameter snapshots $S$ collecting discrete trajectories $\state_{i,k} := \state_{\mu_i}(t_{i,k})$ and parameters $\mu_i$:
\begin{equation}\label{eq:snapshots-velocity-based}
    S:= \{\big(\state_{i,k}, \mu_i\big) \;|\; k \in \mathcal{I}_k, i \in \mathcal{I}_i \}\,.
\end{equation}
We define
\begin{align}\label{eq:cost-non-intrusive-velocity-based}
    J&(G,\Phi,\rho_{\theta,\mu},S,X_\mu) \\
    &= \sum_{(\state_{i,k}, \mu_i)\in S} \|X_{\mu_i}(\state_{i,k}) - X_{\rho_{\theta,\mu_i}(\state_{i,k})}(\state_{i,k}) \|\,. \nonumber
\end{align}
Contrary to~\eqref{eq:cost-intrusive-integral}, the cost~\eqref{eq:cost-non-intrusive-velocity-based} can also be implemented as a non-intrusive cost-function, by replacing $X_{\mu_i}(\state_{i,k})$ with velocity-measurements. 

\begin{theorem}\label{thm:explicit-rho-non-intrusive-velocity-based}
    For fixed $(G,\Phi)$, the optimization problem~\eqref{eq:optimization-problem} with cost~\eqref{eq:cost-non-intrusive-velocity-based} has the non-unique solution $\rho^*_{\mu} \in C^\infty(\M,\g)$ given by~\eqref{eq:solution-rho-cost-intrusive}.   
\end{theorem}
\begin{proof}
    The proposed explicit solution~\eqref{eq:solution-rho-cost-intrusive} minimizes the summand $\|X_{\mu_i}(\state_{i,k}) - X_{\rho_{\theta,\mu_i}(\state_{i,k})}(\state_{i,k}) \|$ for all $\mu_i, \state_{i,k}$, and thus also the sum. However, values of $\rho_\mu(\state)$ for $(\state,\mu) \notin S$ do not influence the value of the cost~\eqref{eq:cost-non-intrusive-velocity-based}, so they are not uniquely determined by the optimization problem.
\end{proof}
The result of Theorem~\ref{thm:explicit-rho-non-intrusive-velocity-based} extends to $\theta$-dependent $\rho_{\theta,\mu}$ if there are $\theta^*$ such that $\rho_{\theta^*,\mu}= \rho_\mu^*$. Such a parametrization is generally easier to find than for Theorem~\ref{thm:explicit-rho-intrusive}, since the equality only has to hold at a finite number of points.

For a velocity-free non-intrusive cost-function we consider a smooth space-time manifold $\M\times\R$ with spatial metric $\textnormal{dist}:\M\times\M\rightarrow\R_+$, and finite set of state-time-parameter snapshots $S$ collecting discrete trajectories $\state_{i,k} := \state_{\mu_i}(t_{i,k})$, time-instances $t_{i,k}$ and parameters $\mu_i$:
\begin{equation}\label{eq:snapshots-velocity-free}
    S:= \{\big(\state_{i,k}, t_{i,k}, \mu_i\big) \;|\; k \in \mathcal{I}_k, i \in \mathcal{I}_i \}\,.
\end{equation}
We define
\begin{align}\label{eq:cost-non-intrusive-velocity-free}
    J&(G,\Phi,\rho_{\theta,\mu},S) \\
     &= \sum_{(\state_{i,k}, t_{i,k}, \mu_i)\in S} \textnormal{dist}\big(\state_{i,k+1},\Phi(e^{\rho_{\theta,\mu_i}(\state_{i,k})\Delta t_{i,k}},\state_{i,k})\big)\,. \nonumber
\end{align}
We note that the cost~\eqref{eq:cost-non-intrusive-velocity-free} corresponds to a discrete-time flow-matching, which is strongly related to the worst case orbit approximation~\eqref{eq:group-width-over-tau} for small $\tau = \Delta t_{i,k}$. The velocity based cost~\eqref{eq:cost-non-intrusive-velocity-based} is an instantaneous projection error, and may be seen as a limiting case of~\eqref{eq:group-width-over-tau} as $\lim_{\tau\rightarrow 0}$. 
\begin{theorem}\label{thm:explicit-rho-non-intrusive-velocity-free}
    If $\rho^*_{\theta,\mu}$ in~\eqref{eq:cost-intrusive-integral} satisfies 
    \begin{align}\label{eq:solution-rho-cost-non-intrusive-velocity-free}
        \rho^*_{\mu_i}&(\state_{i,k}) \\ &= \arg\min_{\widetilde{A}\in\g} \textnormal{dist}\big(\state_{i,k+1},\Phi(e^{\widetilde{A}\Delta t_{i,k}},\state_{i,k})\big)\,, \nonumber
    \end{align}
    for all $(\state_{i,k},t_{i,k}, \mu_i)\in S$ then it is a non-unique solution to the optimization problem~\eqref{eq:optimization-problem} with cost~\eqref{eq:cost-intrusive-integral}, for fixed $G,\Phi$.
\end{theorem}
\begin{proof}
    Analogous to the proof of Theorem~\ref{thm:explicit-rho-non-intrusive-velocity-based}.
\end{proof}
We separately optimize over $\Xi$ and $\theta$, in the following section. %consider a decoupled optimization strategy, to

\subsection{Optimization Strategy}
We investigate a decoupled optimization strategy, in which a first optimization identifies $(G^*,\Phi^*)\in \Xi$, and a second optimization identifies the minimizer $\rho^*_{\theta,\mu}$ of~\eqref{eq:cost-intrusive-integral},~\eqref{eq:cost-non-intrusive-velocity-based} or~\eqref{eq:cost-non-intrusive-velocity-free} for fixed $(G^*,\Phi^*)$. The following theorem states that if the solution trajectories already lie on a group orbit, then the optimal reduced vector field recovers the FOM exactly.

%So far, we assumed that $(G^*,\Phi^*)$ were already known. 
\begin{theorem}[Decoupled Optimization]\label{thm:decoupled_optimization}
    With $S$ the set of solution snapshots~\eqref{eq:solution-snapshots}. If there exist $(G^*,\Phi^*)$ such that either 
    \begin{enumerate}
        \item $\state_\mu(t) \in \mathcal{O}(\state_{\mu,0})$ for all $\state_{\mu,0} \in S$.
        \item $X_\mu \in \Gamma(\Delta)$.
    \end{enumerate} 
    Then there is $\rho^*_{\mu}:\M\rightarrow\g$ such that $(G^*,\Phi^*,\rho^*_\mu)$ minimizes $J$ in Eqs.~\eqref{eq:cost-intrusive-integral},~\eqref{eq:cost-non-intrusive-velocity-based} and~\eqref{eq:cost-non-intrusive-velocity-free}.
\end{theorem}
\begin{proof}
    If there is $(G^*,\Phi^*)$ such that $\state_\mu(t) \in \mathcal{O}(\state_{\mu,0})$ for all $\state_{\mu,0} \in S$, then it holds by definition for any $\state_1, \state_2 \in \mathcal{O}(\state_{\mu,0})$, that there exists $g \in G$ such that 
    \begin{equation}
        \state_2 = \Phi(g,\state_1)\,.
    \end{equation}
    With reference to Theorem~\ref{thm:explicit-rho-non-intrusive-velocity-free}, we then construct the explicit minimizer $\rho^*_{\mu_i}(\state_{i,k}) = \log(g_{i,k} )/ \Delta t_{i,k}$ of~\eqref{eq:cost-non-intrusive-velocity-free}. Instead, using that $X_\mu \in \Gamma(\Delta)$ and with reference to Theorem~\ref{thm:explicit-rho-intrusive}, the explicit solution can be constructed as $\rho_\mu^*(\state) = X^\dagger_\g X_\mu(\state)$. This yields the solution to~\eqref{eq:cost-intrusive-integral} and~\eqref{eq:cost-non-intrusive-velocity-based}.
    %See Appendix~\ref{sec:appendix}.
\end{proof}
The assumptions of Theorem~\ref{thm:decoupled_optimization} directly encode that the FOM dynamics can be expressed as dynamics on $G^*$. The assumptions are equivalent~--~however, Assumption 1 can be checked more easily for discrete velocity-free measurements, and Assumption 2 can be checked more easily when the FOM dynamics are known.

If we can identify $(G^*,\Phi^*)\in \Xi$ for which either of the assumptions hold, the optimal $\rho^*_{\theta,\mu}$ is explicitly determined by Theorems~\ref{thm:explicit-rho-intrusive},\ref{thm:explicit-rho-non-intrusive-velocity-based} and~\ref{thm:explicit-rho-non-intrusive-velocity-free}. To this end we present a concrete algorithm for optimization in Section~\ref{sec:numeric_examples}, where we perform a discrete search over select $\Xi$.

\subsection{Subalgebra search}\label{ssec:subalgebra-search}
So far we did not discuss the dimension of the group $G^*$, which may be high-dimensional. Here, we present a means to further reduce the dimension of the ROM, assuming that a first optimizer $(G^*,\Phi^*,\rho_\mu^*)$ was found. 

We investigate the induced admissible pairs $\Xi^{G^*,\Phi^*}$ (cf. Def.~\ref{def:induced_admissible_pairs}), which form a search space for yet lower dimensional approximations on subgroups $H^*\subseteq G^*$. To this end, we present a theorem:
\begin{theorem}[Subalgebra search]\label{thm:subalgebra-search}
    Given $\M$ and $(G^*,\Phi^*,\rho_\mu^*)$ that minimizes the optimization problem~\eqref{eq:optimization-problem} with cost~\eqref{eq:cost-intrusive-integral},~\eqref{eq:cost-non-intrusive-velocity-based} or~\eqref{eq:cost-non-intrusive-velocity-free}, respectively. If $\rho_{\mu}^*(\state) \in \mathfrak{h} \subseteq \g$ is fully contained in the Lie subalgebra $\mathfrak{h} \subseteq \g$ corresponding to the Lie subgroup $H^* \subseteq G^*$, then $(H^*,\Phi^*,\rho_\mu^*)$ also solves~\eqref{eq:optimization-problem} for the respective cost. 
\end{theorem}
By Theorem~\ref{thm:subalgebra-search}, %The dynamics $\rho_{\mu}(\appstate)$ are uniquely defined if $\Phi$ is free (prove this).
the problem of finding a subgroup $H^*\subset G^*$ reduces to finding a subalgebra $\mathfrak{h}^*\subset\g^*$ that contains $\rho_\mu^*$. 

% We also describe an analogous Theorem for the non-intrusive cases introduced earlier:
% \begin{theorem}[Non-intrusive Hyperreduction]
%     Given $(G^*,\Phi^*,\rho^*_\mu)$ minimizing~\eqref{eq:cost-non-intrusive-velocity-based} or~\eqref{eq:cost-non-intrusive-velocity-free} for snapshots $S$, respectively. Consider the reduced snapshot matrix
%     \begin{equation}
%         S_\g := \{\rho^*_\mu(\state_\mu) \;|\; (\state_\mu,\mu) \in S\}\,.
%     \end{equation}
%     If $\textnormal{span} (S_\g) = \mathfrak{h}\subseteq \g$, then $(H^*,\Phi^*,\rho^*_\mu)$ also minimizes~\eqref{eq:cost-non-intrusive-velocity-based} or~\eqref{eq:cost-non-intrusive-velocity-free}, respectively.
% \end{theorem}
In a non-intrusive context we make a further definition: 
\begin{definition}\label{def:reduced-snapshot-matrix}
    Given $(G^*,\Phi^*,\rho^*_\mu)$ that solves~\eqref{eq:cost-non-intrusive-velocity-based} or~\eqref{eq:cost-non-intrusive-velocity-free}, and given a set of snapshots $S$ as in~\eqref{eq:snapshots-velocity-based} or~\eqref{eq:snapshots-velocity-free}, respectively. Then we define the reduced snapshot matrix $S_\g$ as
    \begin{equation}\label{eq:reduced-snapshot-matrix}
        S_\g := \{\rho^*_\mu(\state_\mu) \;|\; (\state_\mu,\mu) \in S\}\,.
    \end{equation}
\end{definition} 
It holds that %We note the following Corollary of Theorem~\ref{thm:subalgebra-search}:
% \begin{corollary}\label{corollary:reduced-snapshot-matrix}
%     If it holds that $\rho_{\mu}^*(\state) \in \mathfrak{h} \subseteq \g$, then this also holds for the columns of $S_\g$.
% \end{corollary}
\begin{equation}\label{eq:property-reduced-snapshot-matrix}
    \rho_{\mu}^*(\state) \in \mathfrak{h} \Leftrightarrow S_\g \subseteq \mathfrak{h}\,.
\end{equation}
Theorem~\ref{thm:subalgebra-search} and~\eqref{eq:property-reduced-snapshot-matrix} allow us to identify a lower dimensional Lie group $H^* \subseteq G^*$, but this requires that $S_\g \subseteq \mathfrak{h}$ holds for some subalgebra $\mathfrak{h} \subseteq \g$.

In practice, however, it rarely holds that $S_\g$ is exactly contained in some subalgebra of $\g$. We therefore relax the condition and approximate the dominant Lie algebra directions via PCA. Concretely, we introduce an inner product on $\g$, and aim to identify a subalgebra $\mathfrak{h}\subseteq \g$ by principal component analysis (PCA) of $S_\g$ and subsequent completion of a Lie algebra $\mathfrak{h}$ by repeated application of the Lie bracket, such that ``most'' of $S_\g$ is contained within $\mathfrak{h}$. %and a corresponding subgroup $H\subseteq G$ for hyperreduction

To this end, we will explore select special cases in Section~\ref{sec:numeric_examples}. % $\in \Xi_N$ and $\in C^\infty(\M,\g)$, 

% and an optimization method.

% Multiple section topics: Examples (Radial Oscillations, PDEs), Theory (Application to PDEs, relation to the method of freezing), and Literature (ManiMOR, initial value dependent embeddings, method of freezing). - Which section should application to PDEs belong to? Or should it be a topic of its own.
% For now:
% Simple example: radial oscillator (as subsection)
% Extensive example: PDEs and relation to method of freezing (as its own section)
% Relation to the literature (remaining sections)

\section{Analytic Examples}\label{sec:analytic_examples}

Sections~\ref{ssec:linear-transport} and~\ref{ssec:method_of_freezing} highlight that MORLie is applicable to distributed systems, although the presented theory in the main article focused on the finite dimensional case. 

\subsection{Kolmogorov $\Xi$ width of linear transport}\label{ssec:linear-transport}
Consider the linear transport equation for a scalar-valued function $u \in C^\infty(\R)$, and scalar $\mu_1,\mu_2 \in \R$:
\begin{equation}\label{eq:linear-transport}
    \der[u]{t} + \mu_1\der[u]{x} = 0\,,\; u(x,0) = u_{\mu_2,0}(x) = \sin(\mu_2 x) \,.
\end{equation}
The solutions of~\eqref{eq:linear-transport} are $u_\mu(x,t) = u_{\mu_2,0}(x-\mu_1 t)$, providing a well understood toy-example for MOR that challenges linear subspace methods. Further consider a set of solution snapshots
\begin{equation}\label{eq:snapshots-linear-transport}
    S = \{ u_\mu(x,t) \;|\; t \in [0,T],\, (\mu_1,\mu_2) \in \mathcal{P} \subseteq \R^2\}\,.
\end{equation}
The Kolmogorov $N$-width of $S$ is well-known (cf. Sec.~\ref{ssec:Kolmogorov-N-width}) to decay slowly. Instead, consider the following result:
 %We derive the Kolmogorov $\Xi$ width of $S$ over the initial conditions $u_{\mu_2,0}$, i.e.,~\eqref{eq:group-width-over-IC}, and show that $d^T_\Xi(S) = 0$. 
\begin{theorem}
    Let $S$ be as in~\eqref{eq:snapshots-linear-transport}. Further consider the Lie group $(\R,+)$ with action $\Phi:\R\times C^\infty(\R)\rightarrow C^\infty(\R)$ given by $\Phi\big(g,u\big)(x) = f(x+g)$, and let $\Xi = \{(\R,\Phi)\}$. Then the Kolmogorov $\Xi$-width of $S$ over the initial conditions $u_{\mu_2,0}$ is 
    \begin{equation}
        d^T_\Xi(S) = 0\,.
    \end{equation}
    irrespective of the choice of metric on $C^\infty(\R)$.
\end{theorem}
\begin{proof}
    We compute
    \begin{equation}
        d^T_{\Xi}(S) := \sup_{\mu_1,\mu_2 \in \R} d_{\Xi}\big(u_{\mu}(\cdot,[0,T]),u_{\mu_2,0}(\cdot)\big)\,.
    \end{equation}
    Expand $d_{\Xi}\big(u_{\mu}(\cdot,[0,T]),u_{\mu_2,0}(\cdot)\big)$ to find
    \begin{align}
        &d_{\Xi}\big(u_{\mu}(\cdot,[0,T]),u_{\mu_2,0}(\cdot)\big) \\
        &=\inf_{\substack{(G,\Phi) \in \Xi}}\; \sup_{u(\cdot,t) \in u_{\mu}(\cdot,[0,T])}\; \inf_{v(\cdot) \in \mathcal{O}(u_{0,\mu_2}(\cdot))} \textnormal{dist}(u,v)\nonumber\\ 
        &= \sup_{t\in[0,T]}\; \inf_{g\in \R} \textnormal{dist}(u_{\mu_2,0}(x-\mu_1 t),u_{\mu_2,0}(x+g))\nonumber\\
        &= \sup_{t\in[0,T]} \textnormal{dist}(u_{\mu_2,0}(x-\mu_1 t),u_{\mu_2,0}(x-\mu_1 t)) \nonumber\\
        &= 0\,.\nonumber
    \end{align}
    Where the third equality holds for $g = -\mu_1 t$, and the fourth equality holds since $\textnormal{dist}(u,u) = 0$ holds for every metric. Since no specific $\mu$ was assumed, also $d^T_{\Xi}(S) = 0$.
\end{proof}
Thus, the Kolmogorov $\Xi$-width of $S$ over the initial conditions is identically zero for a computationally tractable group action. This provides an example where $d^T_{\Xi_N^{\textnormal{Prop}}}(S) < d_N(S)$. 

For completeness, a family of ROMs on the Lie group $\R$ can be identified by Theorem~\eqref{thm:explicit-rho-intrusive} with constant $\rho_\mu(u) = -\mu_1$, resulting in
\begin{equation}
    \dot{g}_\mu = {R_g}_*\rho_\mu = -\mu_1\,,\; g_\mu(0) = 0\,,    
\end{equation}
and reconstructed solutions
\begin{align}
    \bar{u}_\mu(x,t) &= \Phi(g_\mu(t),u_{\mu_2,0}(x)) \\ 
    &= u_{\mu_2,0}(x + g_\mu(t))  \nonumber \\
    &= u_{\mu_2,0}(x - \mu_1 t)  \nonumber \\
    & = u_\mu(x,t) \,. \nonumber
\end{align}

\subsection{The method of freezing in MORLie}\label{ssec:method_of_freezing}
We describe the method of freezing~\cite{Ohlberger2013} in MORLie, recovering it in terms of a choice $(\widetilde{G},\widetilde{\Phi},\widetilde{\rho})$. 

We begin by describing the method of freezing as presented in~\cite{Ohlberger2013}.
Consider the function space $U =  H^\infty(\R^n,\R)$, an elliptic operator $\mathcal{L}_{\mu}:U\rightarrow U$, and take $u_\mu \in U$ subject to the following nonlinear Cauchy problem 
\begin{equation}\label{eq:mof-Nonlinear_Cauchy_Problem}
    \dot{u}_{\mu} + \mathcal{L}_{\mu}(u_{\mu}) = 0 \,, \; u_{\mu}(0) = u_0\,.
\end{equation}  
Let there be an action $\Phi:G\times U \rightarrow U$, such that the operator $\mathcal{L}_{\mu}$ is equivariant w.r.t. $\Phi$:
\begin{equation}\label{eq:mof_equivariance}
    g_\mu^{-1}\cdot \mathcal{L}_{\mu} \big(g_\mu\cdot v_\mu\big) = \mathcal{L}_{\mu}(v_\mu)\,,
\end{equation}
and expand $u_\mu \in U$ as $g_\mu \in G$ acting on $v_\mu \in U$:
\begin{equation}\label{eq:mof-general-action}
    u_\mu(t) = \Phi\big(g_\mu(t), v_\mu(t)\big) =: g_\mu(t) \cdot v_\mu(t) \,.
\end{equation}  
Then the PDE~\eqref{eq:mof-Nonlinear_Cauchy_Problem} can be rewritten as a partial differential algebratic equation (PDAE)~\eqref{eq:mof-horizontal-dynamics}~-~\eqref{eq:mof-phase-condition}:
\begin{align}\label{eq:mof-horizontal-dynamics} %eq:separated_dynamics
    \dot{v}_\mu &= - \mathcal{L}^G_{\mu,\widetilde{A}} \big(v_\mu\big) \,, \\
    \dot{g}_\mu &= g_\mu \widetilde{A}\,, \label{eq:mof-vertical-dynamics} \\ %eq:reconstruction_equation
    \Psi&\big(\widetilde{A},\dot{v}_\mu\big) =  0\,, \label{eq:mof-phase-condition} 
\end{align} 
where $\mathcal{L}^G_{\mu,\widetilde{A}} \big(v_\mu\big) := \mathcal{L}_{\mu} \big(v_\mu\big) + X_{\widetilde{A}} (v_\mu)$, and $\Psi\big(\widetilde{A},\dot{v}_\mu\big)$ %Eq.~\eqref{eq:mof-phase-condition}
is an invertible phase condition that uniquely determines $\widetilde{A}(t)\in\g$.

The method of freezing proceeds as follows:  high-fidelity solution snapshots of $v_\mu(t)$ are computed, a POD-greedy approach is used to determine a reduced basis 
\begin{equation}\label{eq:mof-reduced-basis}
    \bar{U} := \textnormal{span}(\{v_1,\cdots,v_N\}) \subseteq U\,,
\end{equation}
and $v_\mu$ is approximated as %in practice, this is encoded in a projection operator $P:U \rightarrow \bar{U}$.
\begin{equation}
    v_\mu(t) \approx \sum_{i=1}^N c_\mu^i(t) v_i \,.
\end{equation} 
Next, the (already discretized, high-dimensional) dynamics and phase-condition are projected to this reduced basis, using empirical operator interpolation~\cite{Drohmann2012} to derive $\mathbb{L}^G_{\mu,\widetilde{A}}:\R^N\rightarrow\R^N$ and $\mathbb{F}:\R^N\rightarrow\g$ that result in a ROM of the form: 
\begin{align}
    \dot{c}_\mu &= - \mathbb{L}^G_{\mu,\widetilde{A}}(c_\mu(t))\,,\; c_\mu(0) = c_{\mu,0}\,, \label{eq:mof-ROM-dynamics-c}\\
    \dot{g}_\mu &= g_\mu \widetilde{A}\ \,,\; g_\mu(0) =e \label{eq:mof-ROM-dynamics-g}\\
    \widetilde{A}(t) &= \mathbb{F}(c_\mu(t))\,.
\end{align} 
Finally, solutions are reconstructed as:
\begin{align}\label{eq:mof-reconstructed-solution}
    u_\mu(t) &= g_\mu(t)\cdot v_\mu(t) \\
    & \approx \sum_{i=1}^N c_\mu^i(t)\big(g_\mu(t)\cdot v_i) \,. \nonumber
\end{align}

To describe the method of freezing in MORLie, the goal will be find $(\widetilde{G},\widetilde{\Phi},\widetilde{\rho})$ such that the reconstructed solution $\bar{u}_\mu(t) = \Phi(g_\mu(t),u_{\mu,0})$ agrees with that in~\eqref{eq:mof-reconstructed-solution}. For completeness, we also mention that the full order manifold is $\M := U =  H^\infty(\R^n,\R)$ and the full order dynamics are $X_\mu(u) := -\mathcal{L}_\mu(u)$.

We begin by identifying $(\widetilde{G},\widetilde{\Phi})$. First, note that the full solutions to the method of freezing lie in the set $\mathcal{O}(\R^k,G) \subseteq H^\infty(\R^n,\R)$ that we define as 
\begin{equation}\label{eq:mof-reachable-subset}
    \mathcal{O}(\R^k,G) := \bigg\{ \sum_{i=1}^N c^i\big(g\cdot v_i)\,|\, c \in \R^k, g \in G \bigg\} \,.
\end{equation}  

\begin{lemma}[Group \& Action]\label{lemma:mof-group-and-action-MoRLie}
    Let $\widetilde{G} = G \times \R^k$ be the product of the Abelian Lie group $(\R^k,+)$ and the Lie group $G$, with elements denoted by $(c_\mu,g_u) \in \widetilde{G}$.
    Further, define the augmented manifold $\M_G = G\times\M$ (cf. Lemma~\ref{lemma:commuting_actions_by_augmentation}) and actions $\widetilde{\Phi}_1:G \times \M_G \rightarrow \M_G$, $\widetilde{\Phi}_2: \R^k \times \M_G \rightarrow \M_G$ as 
    \begin{align}
        \widetilde{\Phi}_1\big(g,(h,u)\big) &= (hg, g\cdot u) \,, \\
        \widetilde{\Phi}_2\big(c,(h,u)\big) &= \bigg(h, u + \sum_{i=1}^N c^i (h\cdot v_i)\bigg)\,.
    \end{align}  
    Then
    \begin{enumerate}
        \item The map $\widetilde{\Phi}: \widetilde{G} \times \M_G \rightarrow \M_G$ given by
                \begin{equation}\label{eq:mof-action-MORLie}
                    \widetilde{\Phi}\big( (g,c), (h,u)\big) = \bigg(hg, g\cdot u +\sum_{i=1}^N c^i (h\cdot g\cdot v_i)\bigg)\,.
                \end{equation}
                is an action of $\widetilde{G}$ on $\M_G$.
        \item With $\pi_2:\M_G\rightarrow\M$ the projection on the second factor, and $\mathcal{O}(\R^k,G)$ as in~\eqref{eq:mof-reachable-subset}, the orbit $\mathcal{O}\big((e,0)\big)$ under $\widetilde{\Phi}$ is such that
                \begin{equation}\label{eq:mof-orbit-equivalence}
                    \pi_2 \mathcal{O}\big((e,0)\big) = \mathcal{O}(\R^k,G) \,.%:= \{ (g, \sum_{i=1}^N c^i (g\cdot v_i)\big) \,|\,  (g,c) \in G\times\R^k \}\,.
                \end{equation}  
    \end{enumerate} 
     
\end{lemma}
\begin{proof}
    First, consider the actions $\Phi_1:G\times\M \rightarrow\M$, $\Phi_2:\R^k\times\M \rightarrow\M$ given by 
    \begin{align}
        \Phi_1(g,u) &= g\cdot u\,,\\
        \Phi_2(c,u) &= u + \sum_{i=1}^N c^i v_i \,.
    \end{align}
    Following Lemma~\ref{lemma:commuting_actions_by_augmentation}, commuting actions $\widetilde{\Phi}_1:G \times \M_G \rightarrow \M_G$, $\widetilde{\Phi}_2: \R^k \times \M_G \rightarrow \M_G$ are constructed, and their product $\widetilde{\Phi}: \widetilde{G} \times \M_G \rightarrow \M_G$ is guaranteed to be an action.
    Second, the orbit $\mathcal{O}\big((h,v)\big)$ is 
    \begin{equation}
        \mathcal{O}\big((h,v)\big) := \big\{ \Phi\big((g,c),(h,v)\big) \,|\,  (g,c) \in \widetilde{G} \big\} \,.
    \end{equation}
    For $(g,\gamma_0) = (e,0)$ this becomes
    \begin{equation}
        \mathcal{O}\big((e,0)\big) = \bigg\{ (g, \sum_{i=1}^N c^i (g\cdot v_i)\big) \,|\,  (g,c) \in \widetilde{G} \bigg\}\,,
    \end{equation}
    and the result~\eqref{eq:mof-orbit-equivalence} follows by inspection.
\end{proof}
% \begin{remark}
%     Actions $\Phi_1:G\times\M \rightarrow\M$, $\Phi_2:\R^k\times\M \rightarrow\M$ that avoid state-augmentation, e.g., given by $\Phi_1(g,u) = g\cdot u$ and $\Phi_2(c,u) = u + \sum_{i=1}^N c^i v_i$ do not commute.
% \end{remark}
\begin{remark}
    An alternate construction of an action of $\widetilde{G}$ on the augmented manifold $M_{\R^k}$ is also possible (cf. Lemma~\ref{lemma:commuting_actions_by_augmentation}) and similarly satisfies $\pi_2 \mathcal{O}\big((e,0)\big) = \mathcal{O}(\R^k,G)$, but is not further investigated.
\end{remark}
We are now ready to describe the method of freezing in MORLie:

\begin{theorem}[Method of freezing in MORLie]\label{thm:mof-morlie}
Given the Lie group $\widetilde{G} = G \times \R^k$, and action $\widetilde{\Phi}: \widetilde{G} \times \M_G \rightarrow \M_G$ as in~\eqref{eq:mof-action-MORLie}. 
Denote by $P:U\rightarrow\R^k$ the projection onto components $c^i$ of the basis $\{v_1,\cdots,v_n\}$ spanning $\bar{U}\subseteq U$ in~\eqref{eq:mof-reduced-basis}. 
Define $\rho:U_G \rightarrow \widetilde{\g}$ in terms of components $\rho_1:U_G \rightarrow \R^k$ and $\rho_2:U_G \rightarrow \g$:
\begin{align}
    \rho_1\big( g,u ) & = - \mathbb{L}^G_{\mu,\mathbb{F}(c)}(P(g^{-1}\cdot u))\\
    \rho_2\big( g,u ) & = \mathbb{F}(P(g^{-1}\cdot u)) \,.
\end{align}
Further identify $(g,u) \in \mathcal{O}\big((e,0)\big)$ with $(g,c) = (g, P(g^{-1}\cdot u))\in \widetilde{G}$. Then the MORLie ROM on $\widetilde{G}$ starting at $(g_0,c_0) \in \mathcal{O}\big((e,0)\big)$ reads, with $\bar{c}(t) = c_0 + c(t)$:
\begin{align} 
        \dot{c} &= - \mathbb{L}^G_{\mu,\mathbb{F}(\bar{c})}\big(\bar{c}(t)\big)\,,\; c(0) = 0\,, \\
        \dot{h} &= h \mathbb{F}\big(\bar{c}(t)\big)\,,\; h(0) = e\,.
\end{align} 
Let $(g_0,u_0) = \big( e, \, \sum_{i=1}^N c_0^i(v_i )\big)$, then the reconstructed solution starting at $(g_0,u_0)$ is
\begin{equation}\label{eq:thm:mof-morlie}
    \widetilde{\Phi}\big((h(t),c(t)),(e,u_0)\big) = \bigg(h(t), \, \sum_{i=1}^N \bar{c}^i(t)(h(t)\cdot v_i)\bigg) \,.
\end{equation}
This solution agrees with~\eqref{eq:mof-reconstructed-solution} for $c_\mu(0) = c_0$ in~\eqref{eq:mof-ROM-dynamics-c} and $g_\mu(0) = e$ in~\eqref{eq:mof-ROM-dynamics-g}, in which case $u_\mu(t) =\pi_2 \widetilde{\Phi}\big((h(t),c(t)),(e,u_0)\big)$ is the solution in the method of freezing.
% Let $(g_0,u_0) = \big( \bar{g}_0, \, \sum_{i=1}^N \bar{c}_0^i(\bar{g}_0 \cdot v_i )\big)$, then the reconstructed solution starting at $(g_0,u_0)$ is
% \begin{equation}
%     \widetilde{\Phi}\big((h,c),(g_0,u_0)\big) = \big(\bar{g}(t), \, \sum_{i=1}^N \bar{c}^i(t)(\bar{g}(t)\cdot v_i)\big) \,.
% \end{equation}
% with $\bar{c}(t) = \bar{c}_0 + c(t)$ and $\bar{g}(t) = \bar{g}_0 g(t)$. This solution agrees with~\eqref{eq:mof-reconstructed-solution} for $c_\mu(0) = \bar{c}_0$ and $g_\mu(0) = \bar{g}_0$, in which case $u_\mu(t) =\pi_2 \widetilde{\Phi}\big((g,c),(g_0,u_0)\big)$ is the solution in the method of freezing.
\end{theorem}
\begin{proof}
    The solutions~\eqref{eq:mof-reconstructed-solution} and~\eqref{eq:thm:mof-morlie} are equal if their initial conditions are equal and their dynamics are equal. This follows from uniqueness of solutions for Lipschitz dynamics. The initial conditions are $u_0$ in both cases. The dynamics are equal since $\dot{h} = \dot{g}_\mu$ and $\dot{c} = \dot{c}_\mu$.  
\end{proof}
% In the method of freezing, the reconstructed solution is
% \begin{equation}
%     \sum_{i=1}^N c_\mu^i(t)\big(g_\mu(t)\cdot v_i\big)\,.
% \end{equation}
% Choosing $\bar{c}_0 = c_\mu(0)$ and $\bar{g}_0 = g_\mu(0)$, we also find that $\dot{\bar{c}}(t) = \dot{c}_\mu$ and $\bar{g}(t) = \bar{g}_0 g(t)$.

% The method of freezing for the PDE \eqref{eq:Nonlinear_Cauchy_Problem} with equivariant $\Phi:G\times H^\infty(\R^k,\R) \rightarrow H^\infty(\R^n,\R)$, reduced basis $(v_1,\cdots,v_k) \subseteq H^\infty(\R^k,\R)$ and interpolated operators $\big(\mathbb{F}\big(\bar{c}(t)\big), - \mathbb{L}^G_{\mu,\,\mathbb{F}(\bar{c})}\big(\bar{c}(t)\big)\big)$, is equivalent to a model order reduction via a Lie group $(\widetilde{G},\widetilde{\Phi},\rho)$, with
% \begin{itemize}
%     \item $\M = G \times H^\infty(\R^n,\R)$
%     \item $X(g,u) = \big(0, -\mathcal{L}(u)\big)$
%     \item $\widetilde{G} = G\times\R^k$
%     \item $\widetilde{\Phi}\big((h,c),(g,u)\big) = \big(gh, u + \sum_{i=1}^N c^i (gh\cdot v_i)\big)$
%     \item $\rho(g,u) = \big(\mathbb{F}\big(\bar{c}(t)\big), - \mathbb{L}^G_{\mu,\,\mathbb{F}(\bar{c})}\big(\bar{c}(t)\big)\big)$
%     \item $u$ of the form $u = \sum_{i=1}^N \bar{c}^i (\bar{g}\cdot v_i)\big)$
% \end{itemize}
\begin{remark}
    Here, the key-insight is that fixed basis elements can be realized as actions of $\R^k$, and moving basis elements by the action of a product Lie group $G\times\R^k$. At this point, the formalism becomes similar to that of transformation-based methods~\cite{Cagniart2019,Cagniart2020_thesis,Black2020,Schulze2023} and the shifted POD~\cite{Reiss2018,Krah2025}, for which MORLie provides a geometric picture that generalizes to manifolds.
\end{remark}

\section{Numerical Examples}\label{sec:numeric_examples}
Here we provide basic and advanced numerical examples of the Theory in Sections~\ref{sec:morlie} and~\ref{sec:optimization}. %(excluding in-depth descriptions of the accompanying code.) of noisy point-clouds undergoing group-motions. Highlight how the data fails to lie on a submanifold, and how group motions naturally describe the problem, despite the data not exactly following a group motion. 
For a summary see Table~\ref{tab:num_example_summary}. The code is available at \href{https://github.com/YPWotte/MORLie}{\texttt{github.com/YPWotte/MORLie}}.
\begin{table}[H]
    \centering
    \begin{tabularx}{\columnwidth}{bss}
        Name & Section & Variables \\ 
        \hline
        Radial Oscillator & \ref{ssec:radial_oscillator} & $\rho_\mu$ \\%$\rho_\mu:\R^2 \rightarrow so(2)$ \\
        Rigid pointcloud & \ref{ssec:rigid-pointcloud} & $G, \rho_\mu$ \\ %$G\subseteq \textnormal{Aff}(3)\,, \rho_\mu:\R^{3N_i} \rightarrow \g$, \\
        Sheering pointclouds & \ref{ssec:sheering-pointclouds} & $G,\Phi,\rho_\mu$ \\ %$G\subseteq \textnormal{Aff}^n(3)\,, $\Phi$ $\rho_\mu:\R^{3N_i} \rightarrow \g$\\
        Liver respiration & \ref{ssec:liver} & $\rho_\mu$
    \end{tabularx}
    \caption{Summary of degrees of freedom for optimization in numerical examples Sec.~\ref{ssec:radial_oscillator} to~\ref{ssec:liver}.\label{tab:num_example_summary}}
\end{table}

 \subsection{Radial oscillations}\label{ssec:radial_oscillator}
% Basic example showing an initial condition following the radial oscillator.

% \subsection{Example: Radial Oscillations}\label{ssec:MORLie_radial_oscillator}
We begin with a basic example highlighting the idea of projecting a vector field to a distribution induced by a known Lie group.
We consider a Riemannian manifold $\R^2$ equipped with the standard Euclidean metric. With constants $a \in \R$, $b \in \mathbb{N}$, $a,b \gg 1$, and parameter $\mu \in \R$, we define a dynamical system describing radial oscillations in polar coordinates $x_1 = q_1\cos(q_2), x_2 = q_1 \sin(q_2)$:
\begin{equation} 
    \begin{cases} 
        \dot{q}_1 = \frac{q_1}{a}\sin(b q_2)\,, \\ 
        \dot{q_2} = \mu \,.
    \end{cases}
\end{equation}
We rewrite this as the FOM vector field
\begin{equation}
    X_\mu(q_1,q_2) = \frac{q_1}{a}\sin(bq_2) \der[]{q_1} + \mu\der[]{q_2}\,.
\end{equation}
We further construct an action on $\R^2$ by the 1-dimensional Lie group $G = SO(2)$, expressed in polar coordinates as 
\begin{equation}
    \Phi\big(\alpha,(q_1,q_2)) = (q_1,q_2+\alpha)\,,
\end{equation}
where the corresponding infinitesimal generator is
\begin{equation}
    X_{\dot{\alpha}}(q_1,q_2) = \dot{\alpha}\der[]{q_2} \,,
\end{equation}
which spans a distribution $\Delta = \{a(q_1,q_2)\der[]{q_2}\;|\; a \in C^\infty(\R^2)  \}$.
% transitive action on $\R^2$ by the Lie group $\widetilde{G} = SIM(2) \cong \R_+ \times S^1$ of scalings and rotations, whose action $\widetilde{\Phi}$ is expressed as
% \begin{equation}
%     \widetilde{\Phi}\big((s,\alpha),(q_1,q_2)) = (sr,q_2+\alpha)\,,
% \end{equation}
% where the corresponding infinitesimal generator is expressed in polar coordinates as
% \begin{equation}
%     X_{\dot{s},\dot{\alpha}}(q_1,q_2) = (q_1\dot{s},\dot{\alpha}) \,.
% \end{equation}
%
Finally, the Euclidean metric in polar coordinates reads $M = \ext q_1 \otimes \ext q_1 + q_1^2 \ext q_2 \otimes \ext q_2$, and the explicit intrusive solution in Theorem~\ref{thm:explicit-rho-intrusive} becomes 
\begin{equation}
    \rho_\mu^*(q_1,q_2) = X^\dagger\Pi_\Delta X_\mu(q_1,q_2) =  \big(0,\mu\big)\,,
\end{equation}
This determines the reduced dynamics $\appX_\mu \in \Gamma(\Delta)$ as
\begin{equation}
    \appX_\mu(q_1,q_2) = X_{\rho_\mu^*(q_1,q_2)}(q_1,q_2) =  \mu\der[]{q_2}\,.
\end{equation}
By an application of Theorem~\ref{thm:MorLie reconstruction}, this reduced vector field can be integrated on $SO(2)$ by solving the system
\begin{align}
    %(\bar{q_1},\bar{q_2}) &= \Phi(\alpha, (q_{1,0},q_{2,0}))\,, \\
    \dot{\alpha} &= \rho_\mu^*\big(\Phi(\alpha, (q_{1,0},q_{2,0}))\big) = \mu \,,\; \alpha(0) = 0\,. 
\end{align}
The resulting approximate solutions are 
\begin{equation}
    \big(\bar{q_1}(t), \bar{q_2}(t)\big) = \Phi\big(\alpha(t),(q_{1,0},q_{2,0})\big) =  (q_{1,0},q_{2,0}+\mu t)\,,
\end{equation}
which %global on $\R^2$, 
are compared to FOM trajectories in Figure~\ref{fig:radial_oscillator}.
\begin{figure}[H]
    \centering
    \def\svgwidth{0.9\columnwidth}\scriptsize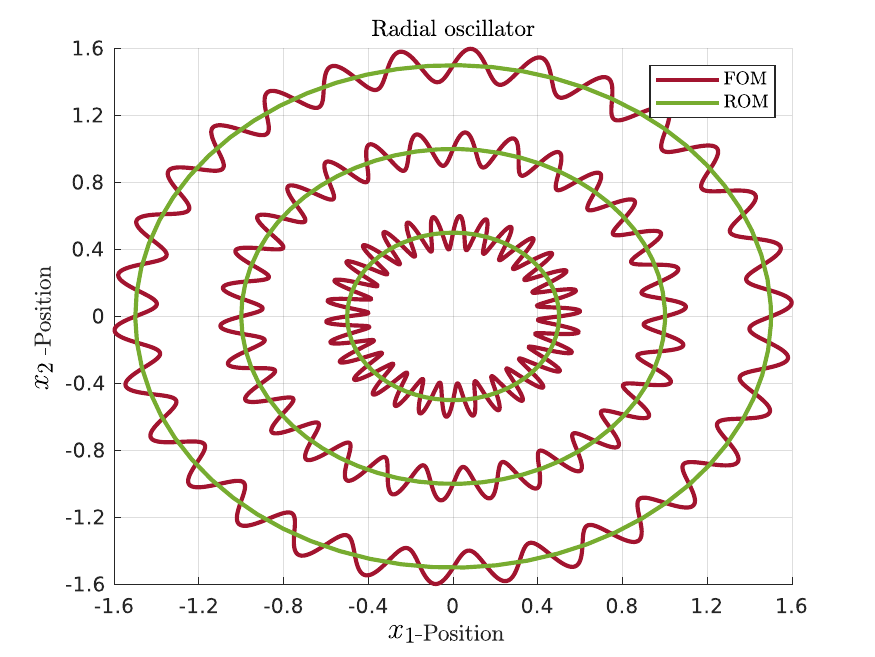
    %\includesvg[width=0.9\columnwidth]{Figures/Radial_Oscillator.svg}
    \caption{Example~\ref{ssec:radial_oscillator}, MOR of the radial oscillator, showing trajectories of the full order model (FOM) and the reduced order model (ROM).}\label{fig:radial_oscillator}
\end{figure}

\subsection{Rigid pointcloud}\label{ssec:rigid-pointcloud}
The following example features optimization over a set of admissible groups, and compares MORLie with a naive POD method. 
We will recover a rigid body motion %$H(t)\in SE(3) \subseteq GL(4,\R)$ 
from noisy measurements of a rigidly moving point cloud. %, and identify the Lie group $SE(3)$ in absence of prior knowledge. 
%This is a proxy for, e.g., measurements of position of landmarks on a rigid body extracted from video data, or of particles in a wind-tunnel following a rigid motion. It also serves as a toy problem for more complex applications, such as recovering a ROM from wind-tunnel experiments, splitting particles into various clusters undergoing different group motions, or recovering the dynamics of a flexible bodies given the position of landmarks. 
%

We consider $H(t) \in SE(3)$ with dynamics 
\begin{equation}\label{eq:rigid-group-dynamics}
    \dot{H}(t) = H(t)\widetilde{T}(t)\,,\; H(0) = H_0\,.
\end{equation}
for a time-dependent $\widetilde{T}(t) \in se(3)$. 

With $j \in \{0,\cdots,N_j\}$ indexing trajectories, $i \in \{0,\cdots,N_i\}$ indexing particles in a given trajectory, $k \in \{0,\cdots,N_k\}$ indexing time-instances $t_k \in [0,T]$ with $t_0 = 0$, we denote the $i$-th particle on the $j$-th trajectory at time $t_k$ by $p^i_{j,k}$. Given measurement noise $\eta_{j,k}^{i} \sim \mathcal{N}(0_3, \sigma I_{3\times 3})$, we model the position measurements by
\begin{equation}\label{eq:rigid-pointcloud-dynamics}
    \pmat{p^i_{j,k} \\ 1} 
    = H(t_k) \pmat{p^i_{j,0} \\ 1} 
    + \pmat{\eta^i_{j,k} \\ 0}\,.
\end{equation}  
This represents e.g., position measurements of markers on a rigid body or landmarks on a rigid body identified from video data.

Finally, we define the system state $P_{j,k} \in \R^{3N}$ of $N$ particles on the $j$-th trajectory and at time $t_k$, by   
\begin{equation}\label{eq:rigid-pointcloud}
    P_{j,k} 
    = \pmat{p^1_{j,k} \\ \vdots \\ p^N_{j,k}}\,,
\end{equation}  
and denote the finite set of state-time solution snapshots as $S \in \R^{(3N_i+1)\times N_j N_k}$ given by
\begin{equation}\label{eq:snapshots-rigid-pointcloud}
    S := \bigg[\pmat{P_{0,0} \\ t_0}, \cdots, \pmat{P_{N_j,N_k}\\ t_{N_k}}\bigg] \,.
\end{equation}  

\subsubsection{Application of a naive POD}\label{sssec:naive-POD-rigid-pointcloud}
For a naive POD, we directly investigate the singular values of $S$ without further pre-processing. Leaving noise aside for an initial analysis:

\begin{theorem}\label{thm:singular-values-rigid-snapshots}
    Assume the following:
    \begin{enumerate}
        \item The number of columns $N_jN_k$ and the number of rows $3N_i+1$ of $S$ are sufficiently large: $N_jN_k \geq 9(N_j+1)+3$ and $3N_i+1 \geq 9(N_j+1)+3$
        %\item The noise is $\eta^i_{j,k} = 0$
        \item Each initial point cloud $P_{j,0} \in \R^{N_j}$ spans $\R^3$: $$\dim\textnormal{span}\{p^i_{j,0}\;|\;i \in \{0,\cdots,N_i\}\} = 3$$
        \item The matrices $R(t_k)$ span $\R^{3\times 3}$: $$\dim\textnormal{span}\{R(t_k)\;|\;k \in \{0,\cdots,N_k\}\} = 9$$
        \item The translations $b(t_k)$ span $\R^3$: $$\dim\textnormal{span}\{b(t_k)\;|\;k \in \{0,\cdots,N_k\}\} = 3$$
        \item The subspaces $$W_j := \textnormal{span}\{P_{j,k}\;|\; k \in \{0,\cdots,N_k\} \}$$ of $\R^{3N_i}$ associated with different point cloud trajectories $P_{j,k}$ are independent. 
    \end{enumerate}  
    Then the number of non-zero singular values of the snapshot set $S$ in~\eqref{eq:snapshots-rigid-pointcloud} is $9(N_j+1)+3$.
\end{theorem}

\begin{proof}
    For the proof, see Appendix~\ref{ssec:appendix-rigid-svd}.
\end{proof}

Hence, the singular values of $S$ scale with the number of point-cloud trajectories. The assumptions for Theorem~\ref{thm:singular-values-rigid-snapshots} are not strong: given randomly sampled initial pointclouds $P_{j,0}$ with sufficiently many points, they hold generically. Considering the simplicity of the underlying (6D) group dynamics~\eqref{eq:rigid-group-dynamics}, this is not yet a problem. If the group dynamics were more complicated, and required many different point-cloud trajectories to capture their complexity, the application of a POD without further processing of $S$ would become unfeasible. 

\begin{remark}\label{rem:noisefloor-rigid-snapshots}
    The noise differs per particle in the point cloud, as is common in practical scenarios, and it is effectively a random vector in $\mathbb{R}^{3N_i}$. Expressing the noise accurately requires a POD basis that spans $\mathbb{R}^{3N_i}$, leading to a noise floor in the singular values.
\end{remark}

\subsubsection{Application of MORLie}
As an example of MORLie, we solve~\eqref{eq:optimization-problem} for the non-intrusive, velocity-free cost~\eqref{eq:cost-non-intrusive-velocity-free}. Note that the space $\R^{3N_i}$ here constructed comes with a natural metric given by the average distance between individual particles: 
\begin{equation}\label{eq:dist_particle_cloud}
    \textnormal{dist}(P_{j_1},P_{j_2}) = \frac{1}{N_i} \sum_i \| p^i_{j_1} - p^i_{j_2} \|\,. 
\end{equation}

Further, let $\textnormal{Aff}(3) = GL(3,\R)\ltimes \R^3 \subseteq GL(4,\R)$ denote the affine group in three dimensions, i.e., $g \in \textnormal{Aff}(3)$ is identified with $A \in GL(3,\R),\,b \in \R^3$ and the semi-direct product $\ltimes$ denotes that the group operation is $(A_1,b_1)\cdot(A_2,b_2) = (A_1A_2, b_1 + A_1b_2)$. The affine action on a single particle $p^i_{j,k} \in \R^3$ is a left action given by
\begin{equation}
    g \cdot p^i_{j,k} = A p^i_{j,k}+b \,.
\end{equation} 
We define the left action $\Phi:\textnormal{Aff}(3)\times\R^{3N_i} \rightarrow \R^{3N_i}$ by letting $g \in \textnormal{Aff}(3)$ act affinely on each particle. For the set of admissible groups we consider $\Xi^{\textnormal{Aff}(3),\Phi}$ (cf. Def.~\ref{def:induced_admissible_pairs}). 

For fixed $(\textnormal{Aff}(3),\Phi)$, we solve for $\rho^*:\R^{3N_i}\rightarrow \textnormal{aff}(3)$ in two ways: first by directly applying Theorem~\ref{thm:explicit-rho-non-intrusive-velocity-free} (in code we use a Lie-group version of the Levenberg-Marquardt algorithm for optimization, see e.g.~\cite[Chapter 8.4.2]{Absil2008}), and second by approximating the FOM vector field as
\begin{equation}\label{eq:approximate-velocity}
    X_\mu(P_{j,k},t) \approx (P_{j,k} - P_{j+1,k})/\Delta t\,,
\end{equation}
and applying Theorem~\ref{thm:explicit-rho-non-intrusive-velocity-based} (in code the closed form solution is computed via a pseudo-inverse of $X_\g$). These are used to construct reduced snapshot matrices $S_{\textnormal{aff}(3)}$ (cf. Def.~\ref{def:reduced-snapshot-matrix}). 

The reduced snapshot matrices are then processed by Algorithm~\ref{app:alg_subalgebra_search} (cf. Appendix~\ref{ssec:appendix-code}). Algorithm~\ref{app:alg_subalgebra_search} proceeds in three steps: first, it performs a POD of $S_{\textnormal{aff}(3)}$. Singular vectors $\{\widetilde{A}_1,\cdots,\widetilde{A}_k\}\subseteq \textnormal{aff}(3)$ are identified, corresponding to singular values $\sigma_1,\cdots,\sigma_k$ such that $\sum^k_{i=1}\sigma_i > a \sum \sigma_i$, with $0<a<1$ a hyperparameter that we pick as $a = 0.99$. Finally, the singular vectors are bracketed until we arrive a set $\g \subseteq \textnormal{aff}(3)$ that is closed under the bracket, and is thus a Lie algebra.

For both cases of $S_{\textnormal{aff}(3)}$, the algorithm identified the 6-dimensional subalgebra $se(3) \subseteq \textnormal{aff}(3)$. Finally, we fit parameterized $\rho_{\theta}:\R\rightarrow se(3)$ to the identified $S_{\textnormal{aff}(3)}$, by solving
\begin{eqnarray}
    \theta^* = \arg\min_\theta \sum_{j,k} \|\rho_{\theta}(t_{j,k}) - \rho^*_{j,k}(P_{j,k})\|\,.  
\end{eqnarray}
%We similarly fit a parameterized $\rho_{se(3),\theta}:\R\rightarrow se(3)$ corresponding to the prior knowledge of a rigid motion.
This results in a ROM on the Lie group $SE(3)$: %$G$ and
\begin{equation}
    \dot{H}(t) = \rho_{\theta^*}(t) H(t)\,,\; H(0) = I_4\,,
\end{equation}
with reconstructed solutions
\begin{equation}
    P_{i}(t) = \Phi(H(t),P_{i,0})\,.
\end{equation}

\subsubsection{Results}
We perform the computation with $N_j = 9, N_i  = 99, N_k = 999$. Executing the provided code on a Lenovo P15v, identification of $S_{\textnormal{aff}(3)}$ via Theorem~\ref{thm:explicit-rho-non-intrusive-velocity-free} takes 16 seconds, and 3 seconds using the closed form Theorem~\ref{thm:explicit-rho-non-intrusive-velocity-based}. The optimization step takes 15 seconds, for an initial guess of $\rho^*$ being a Hermite polynomial with 100 segments, fit to every 10th entry of an average of $S_{\textnormal{aff}(3)}$ over the 10 trajectories.

Snapshots of a sample trajectory of a point-cloud are shown in Figures~\ref{sfig:rigidcloud_trajectories_a}~-~\ref{sfig:rigidcloud_trajectories_c}, along with the reconstructed trajectory with $S_{\textnormal{aff}(3)}$ based on Theorem~\ref{thm:explicit-rho-non-intrusive-velocity-free}. Figures~\ref{sfig:rigidcloud_error_and_svd_a} and~\ref{sfig:rigidcloud_error_and_svd_b} shows the error of the reconstructions via MORLie, and Figure~\ref{sfig:rigidcloud_error_and_svd_e} shows the error for various projections to a POD basis. Figure~\ref{sfig:rigidcloud_error_and_svd_c} shows the singular values of $S$ in~\eqref{eq:snapshots-rigid-pointcloud}. The singular values sharply drop at $9 (N_j+1)+3 = 93$, c.f. Theorem~\ref{thm:singular-values-rigid-snapshots}. The singular values do not drop to zero since noise is added to each node, c.f. Remark~\ref{rem:noisefloor-rigid-snapshots}. Figure~\ref{sfig:rigidcloud_error_and_svd_d} shows the singular values of $S_{\textnormal{aff}(3)}$ that relate to the Kolmogorov $\Xi^{\textnormal{Aff}(3)}_N$ width of the problem. 

%Table [...] summarizes the Kolmogorov $(\Xi^{\textnormal{Aff}(3),\Phi})$ width $d^T_{(\Xi^{\textnormal{Aff}(3),\Phi})}$ w.r.t.\ the initial conditions (c.f.~\eqref{eq:group-width-over-IC}), the Kolmogorov $(\Xi^{G,\Phi})$ width $d^T_{(\Xi^{G,\Phi})}$ w.r.t.\ the initial conditions, the Kolmogorov $(\Xi^{SE(3),\Phi})$ width $d^T_{(\Xi^{SE(3),\Phi})}$ w.r.t.\ the initial conditions, and the results of a naive POD for an estimate of various Kolmogorov $N$-widths $d_N$ (cf.~\eqref{eq:Kolmogorov-N-width}).

\begin{figure*}[t]
    \centering
    \begin{tabular}{c c c}
    \begin{subfigure}[t]{.28\textwidth}
        \def\svgwidth{1.15\textwidth}\scriptsize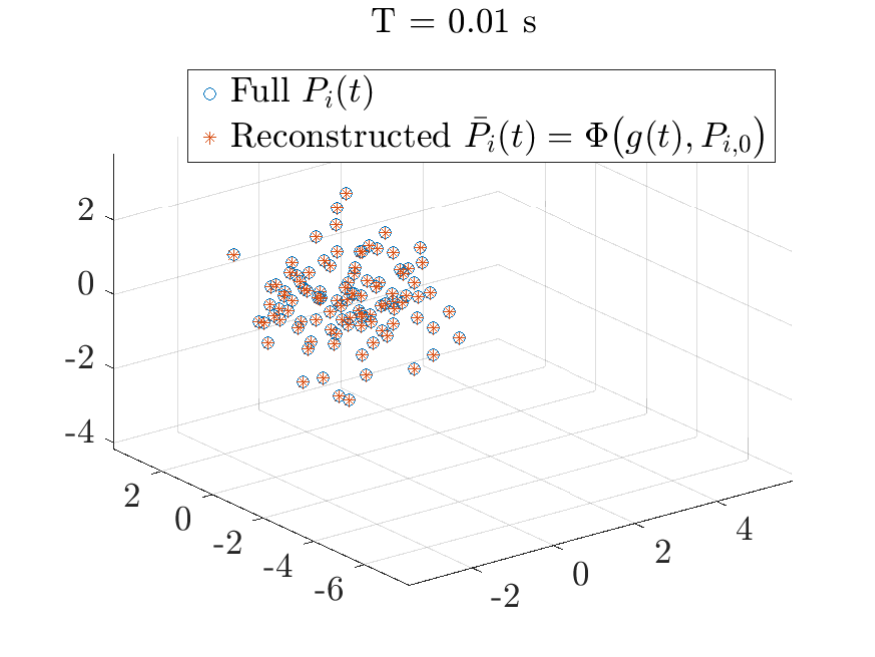
        %\includesvg[height=12em]{Figures/RigidBody/ReconstructionSnapshot_1R.svg}
        \subcaption{T = 0.01s.}\label{sfig:rigidcloud_trajectories_a}
    \end{subfigure}
    &
    \begin{subfigure}[t]{.28\textwidth}
        \def\svgwidth{1.15\textwidth}\scriptsize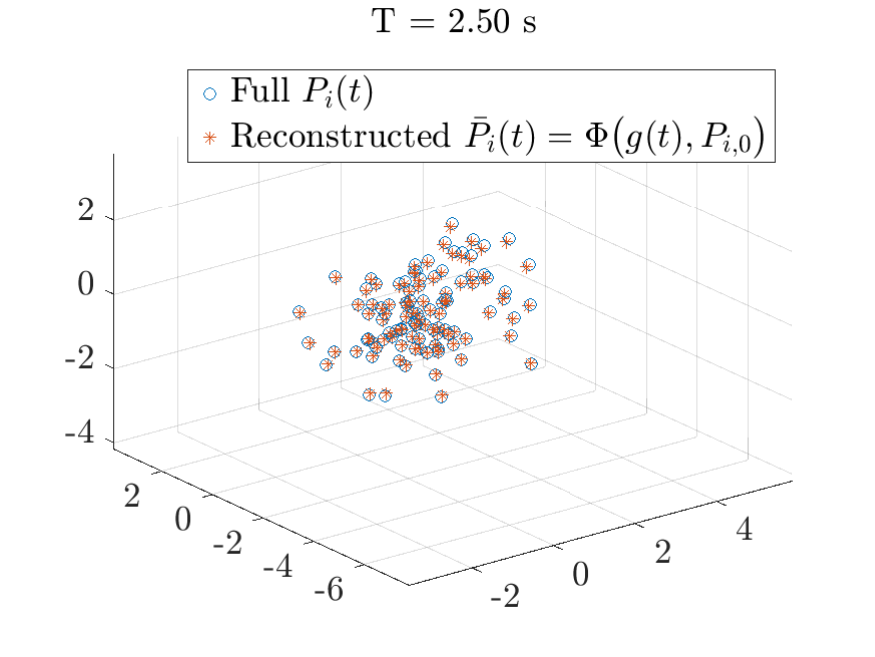
        %\includesvg[height=12em]{Figures/RigidBody/ReconstructionSnapshot_3R.svg}
        \subcaption{T = 2.5s.}\label{sfig:rigidcloud_trajectories_b}
    \end{subfigure}
    &
    \begin{subfigure}[t]{.28\textwidth}
        \def\svgwidth{1.15\textwidth}\scriptsize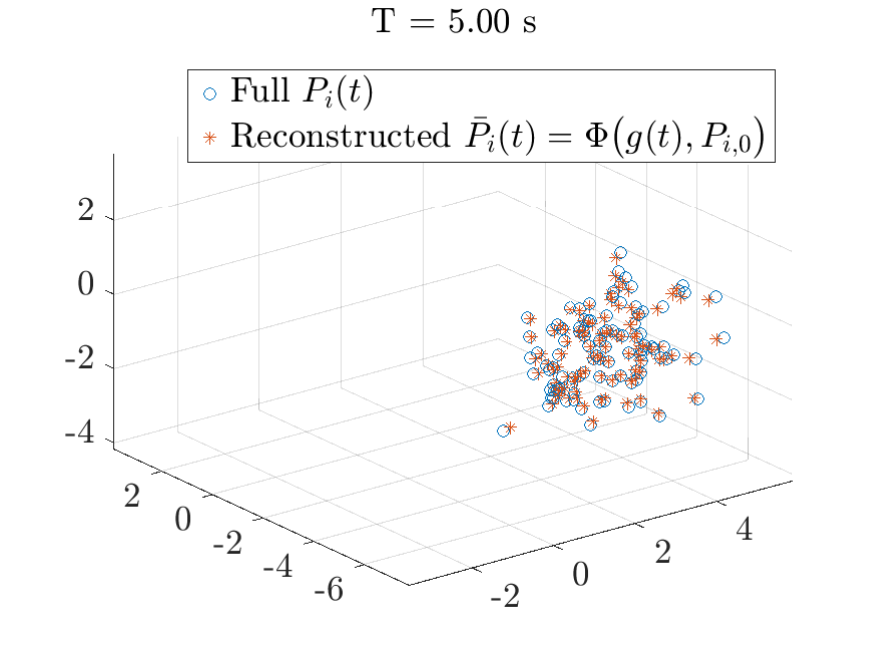
        %\includesvg[height=12em]{Figures/RigidBody/ReconstructionSnapshot_5R.svg}
        \subcaption{T = 5s.}\label{sfig:rigidcloud_trajectories_c}
    \end{subfigure}
    \end{tabular}
    \caption{\small Example~\ref{ssec:rigid-pointcloud},  rigidly evolving pointlcloud with noise (blue circles) and reconstructed solution (red stars). }\label{fig:rigidcloud_trajectories}
\end{figure*}

\begin{figure*}[t]
    \centering
    \begin{tabular}{c c c}
    \begin{subfigure}[t]{.24\textwidth}
        \includegraphics[height=9.5em]{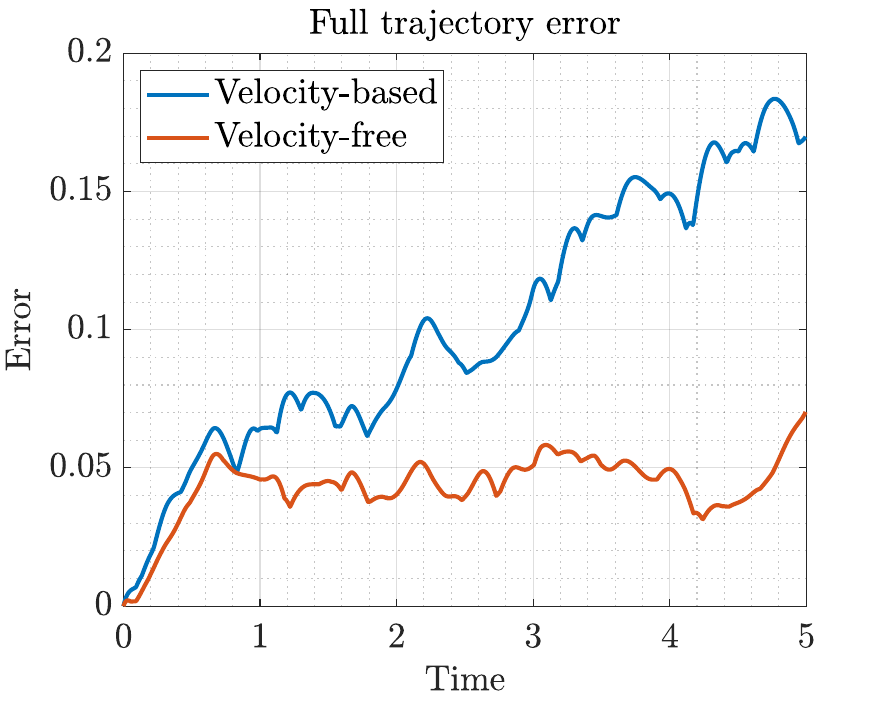}
        %\scriptsize\includesvg[inkscapelatex=false, height=12em]{Figures/RigidBody/TT_Error_2025-9-18_at_13-52-33.svg}
        \subcaption{Full trajectory reconstruction errors $\textnormal{dist}\big(\bar{P}(t),P(t)\big)$ for different fits. }\label{sfig:rigidcloud_error_and_svd_a}
    \end{subfigure}
    &
    \begin{subfigure}[t]{.24\textwidth}
        \includegraphics[height=9.5em]{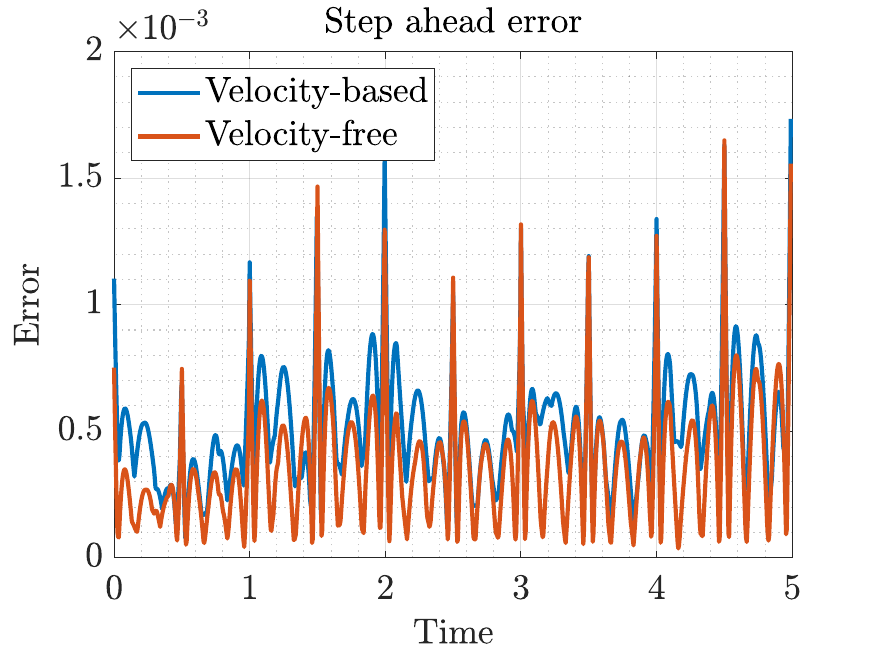}
        %\scriptsize\includesvg[inkscapelatex=false, height=12em]{Figures/RigidBody/SP_Error_2025-9-18_at_13-52-33.svg}
        \subcaption{Step ahead reconstruction error $\textnormal{dist}\big(\bar{P}_{P(t)}(\Delta t),P(t+\Delta t)\big)$ for different fits, where $\bar{P}_{P(t)}(s)$ is an approximate solution starting at $P(t)$.}\label{sfig:rigidcloud_error_and_svd_b}
    \end{subfigure}
    &
    \begin{subfigure}[t]{.24\textwidth}
        \includegraphics[height=9.5em]{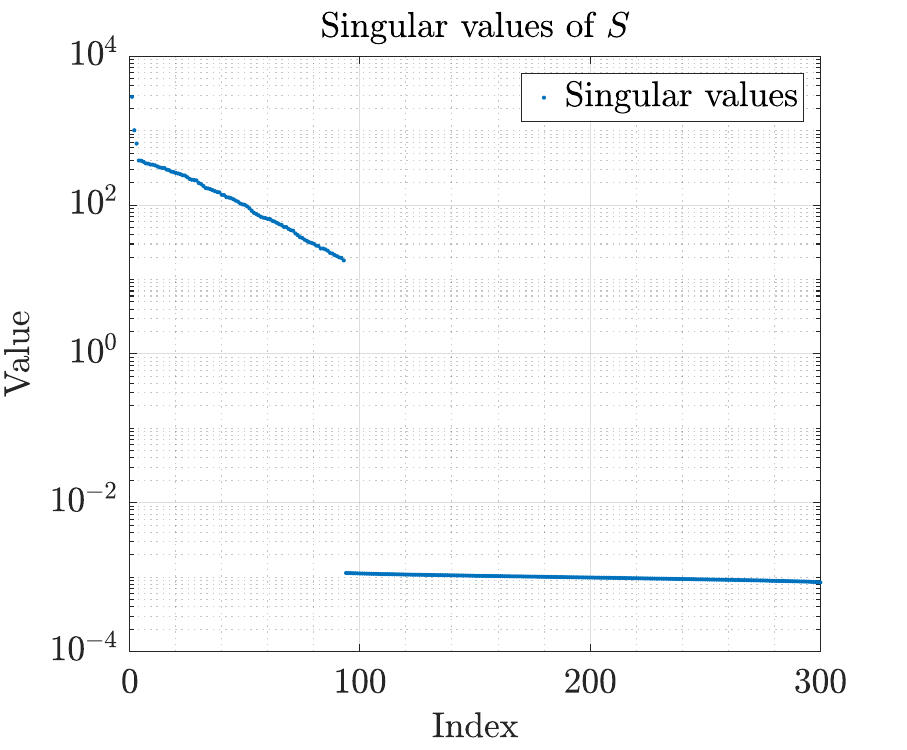}
        %\tiny\includesvg[inkscapelatex=false, height=14em]{Figures/RigidBody/SVD_2025-9-18_at_13-52-33.svg}
        \subcaption{Singular values of $S$.% for $N_j = 9$ (this example).
        }\label{sfig:rigidcloud_error_and_svd_c}
    \end{subfigure}
    \\
    \begin{subfigure}[t]{.24\textwidth}
        \includegraphics[height=9.5em]{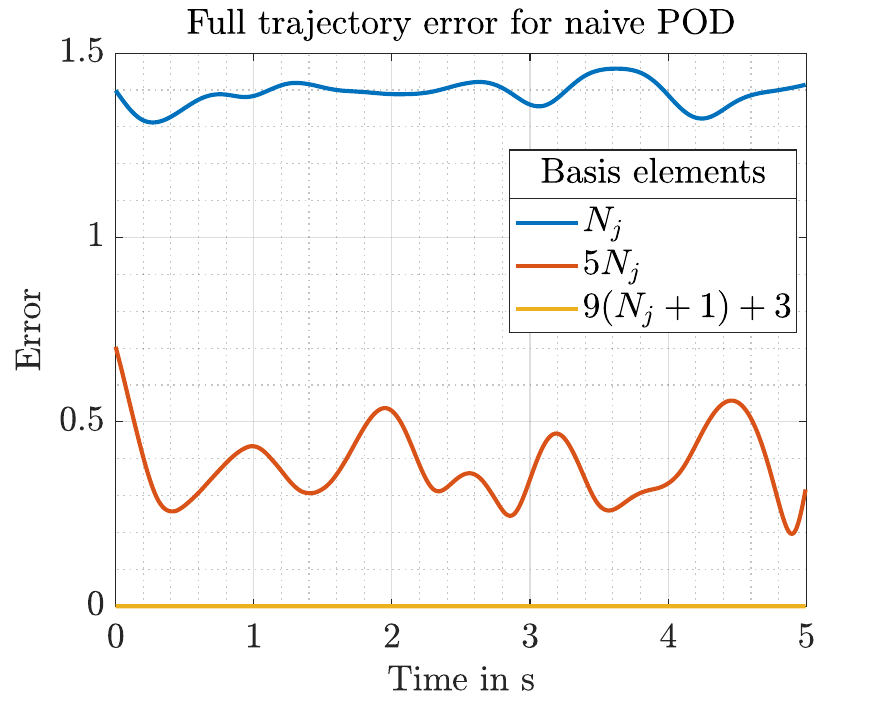}
        %\scriptsize\includesvg[inkscapelatex=false, height=12em]{Figures/RigidBody/POD_Error.svg}
        \subcaption{Full trajectory errors $\textnormal{dist}\big(P_{POD}(t),P(t)\big)$, where $P_{POD}$ is a projection of $P$ to a POD basis.}\label{sfig:rigidcloud_error_and_svd_e}
    \end{subfigure}
    &
    &
    \begin{subfigure}[t]{.24\textwidth}
        \includegraphics[height=9.5em]{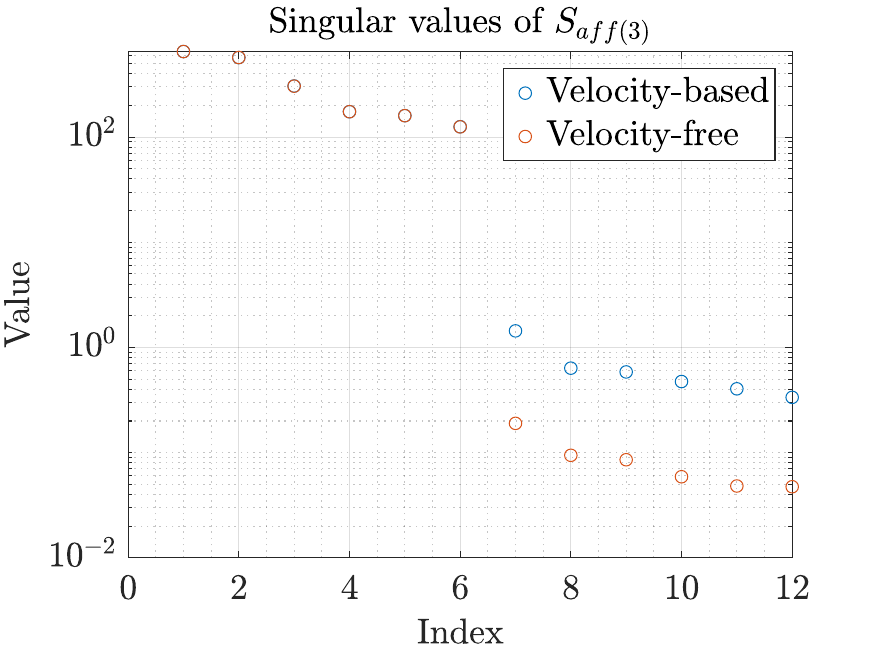}
        %\tiny\includesvg[inkscapelatex=false, height=14em]{Figures/RigidBody/SVD_VB_2025-9-18_at_13-52-33.svg}
        \subcaption{Singular values of $S_{\textnormal{aff}(3)}$.}\label{sfig:rigidcloud_error_and_svd_d}
    \end{subfigure}
    % &
    % \begin{subfigure}[t]{.24\textwidth}
    %     \tiny\includesvg[inkscapelatex=false, height=14em]{Figures/RigidBody/SVD_POD_L20.svg}
    %     \subcaption{Singular values of $S$ for $N_j = 19$ (a variation of this example).}\label{sfig:rigidcloud_error_and_svd_f}
    % \end{subfigure}
    \end{tabular}
    \caption{\small Example~\ref{ssec:rigid-pointcloud}, reconstruction errors and singular values for rigidly evolving pointcloud.}\label{fig:rigidcloud_error_and_svd}
\end{figure*}

% We make two simplifying assumptions: 
% \begin{itemize}
%     \item The candidate groups are of the form $G = \textnormal{Aff}(3) \times \ldots \times \textnormal{Aff}(3)$, i.e., $G$ consists of a to-be-determined number %$n_G$ 
%     of copies of $\textnormal{Aff}(3) = GL(3,\R)\ltimes \R^3$ 
%     \item  For $g \in \textnormal{Aff}(3)$ identified with $A \in GL(3,\R),\,b \in \R^3$, the affine action for one particle $p \in \R^3$ is given by
%     \begin{equation}
%         g \cdot p = A p+b \,.
%     \end{equation} 
%     When $G = \textnormal{Aff}(3) \times \ldots \times \textnormal{Aff}(3)$, candidate actions $\Phi:G \times \R^{3N} \rightarrow \R^{3N}$ act affinely on each particle, with a to-be-determined group acting on each particle.

\subsection{Sheering pointclouds}\label{ssec:sheering-pointclouds}
% Compare naive POD approach and MorLie approach, showing both velocity-free and velocity-based approach.
The following example features optimization over a set of admissible groups as well as actions. 
We will recover group motions of multiple particle clusters %$H(t)\in SE(3) \subseteq GL(4,\R)$ 
from noisy measurements of point clouds consisting of multiple affinely deforming subsets. The example is similar to the previous section, except that we consider trajectories of individual particles given by 
\begin{equation}
    p^i_{j,k} 
    = A_{f(j)}(t_k) p^i_{j,0} 
    + b_{f(j)}(t_k) + \eta^i_{j,k}\,.
\end{equation}  
where $f(j) \in \{0,1,\cdots, n_G\}$ assigns to distinct affine transformations following dynamics 
\begin{equation}
    \dot{g}_{f(j)} = \widetilde{A}_{f(j)}(t) g_{f(j)}\,,\; g_{f(j)}(0) = g_{f(j),0}\,.
\end{equation}
As a search space for groups, we consider any number of copies of the affine group, i.e., groups of the form $G = \textnormal{Aff}^{n_G}(3) := \textnormal{Aff}(3) \times \ldots \times \textnormal{Aff}(3)$, of a to-be-determined number $n_G$ of copies of $\textnormal{Aff}(3)$. For group actions, we allow any fixed assignment of particles to group clusters, which transform them by the respective affine action. This forms a large set of admissible actions $\Xi^{G,\Phi}$.

For an initial optimization inspired by Theorem~\ref{thm:decoupled_optimization}, we present Algorithm~\ref{app:cluster_search} in Appendix~\ref{ssec:appendix-code}. This algorithm uses a nearest-neighbor filtering approach to split finite $S$ into clusters $S_1,\cdots, S_{n_G}$, without further prior knowledge, such that points $p^i_{j,k}$ in $S_{f(j)}$ satisfy 
\begin{equation}
    p^i_{j,k} 
    \approx A_{f(j)}(t_k) p^i_{j,0} 
    + b_{f(j)}(t_k) \,.
\end{equation}
Thus, we arrive at $(G^*,\Phi^*)$, with $G^*$ consisting of $n_G$ copies of $\textnormal{Aff}(3)$, and $\Phi$ assigning affine actions to unique clusters. Then $S_\g$ is identified as previously detailed. Again, a subalgebra is identified by Algorithm~\ref{app:alg_subalgebra_search} resulting in a dynamics on a subgroup $(H^*,\Phi^*)$. %

\subsubsection{Results}
Snapshots of a sample trajectory of two point-clouds with 100 points each are shown in Figures~\ref{sfig:sheeringclouds_trajectories_a}~-~\ref{sfig:sheeringclouds_trajectories_c}. In terms of the previous example, we have $N_j = 9, N_i  = 199, N_k = 999$. Executing the provided code on a Lenovo P15v, identification of $n_G = 2$ and identification of $S_{\textnormal{aff}^2(3)}$ via Theorem~\ref{thm:explicit-rho-non-intrusive-velocity-free} take 34 seconds, and 7 seconds using the closed form in Theorem~\ref{thm:explicit-rho-non-intrusive-velocity-based}. The optimization step again takes on the order of 15 seconds.

Figures~\ref{sfig:sheeringclouds_error_and_svd_a} and~\ref{sfig:sheeringclouds_error_and_svd_b} shows the error of the reconstructions via MORLie, Figure~\ref{sfig:sheeringclouds_error_and_svd_c} shows the singular values of $S$ in~\eqref{eq:snapshots-rigid-pointcloud} that relate to the Kolmogorov $N$-width of the problem, and Figure~\ref{sfig:sheeringclouds_error_and_svd_d} shows the singular values of $S_{\textnormal{aff}^2(3)}$ that relate to the Kolmogorov $\Xi^{\textnormal{Aff}^2(3)}_N$ width of the problem. 
%We consider an example with two clusters of 100 points each, and in terms of the previous example we choose $N_j = 9, N_i  = 199, N_k = 999$. Figures~\ref{sfig:sheeringclouds_trajectories_a} to~\ref{sub@sfig:sheeringclouds_trajectories_c}%Table [...] summarizes the results for various numbers of group clusters. 

\begin{figure*}[t]
    \centering
    \begin{tabular}{c c c}
    \begin{subfigure}[t]{.28\textwidth}
        \def\svgwidth{1.15\textwidth}\scriptsize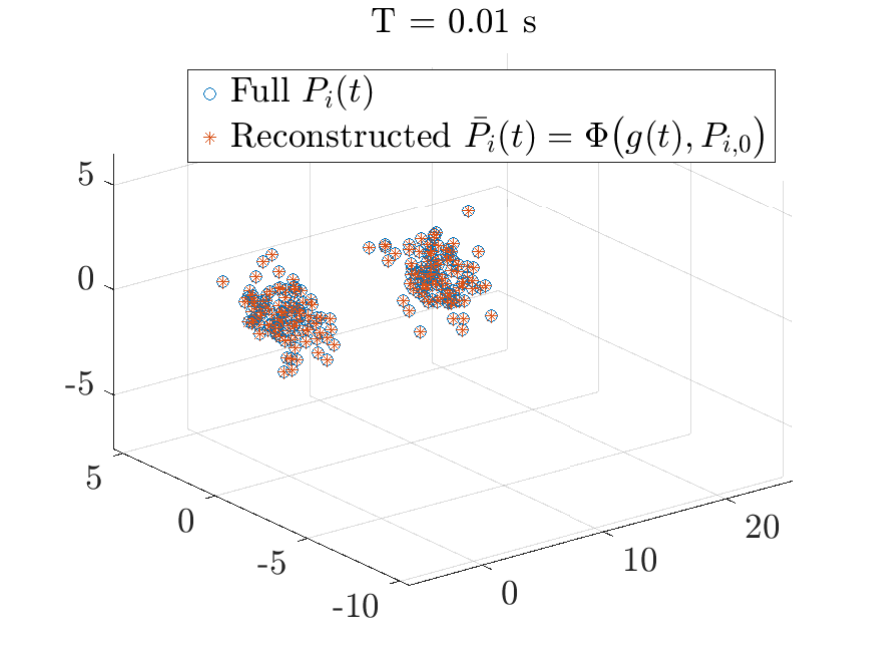
        %\includesvg[height=12em]{Figures/SheeringBodies/ReconstructionSnapshot_1S.svg}
        \subcaption{T = 0.01s.}\label{sfig:sheeringclouds_trajectories_a}
    \end{subfigure}
    &
    \begin{subfigure}[t]{.28\textwidth}
        \def\svgwidth{1.15\textwidth}\scriptsize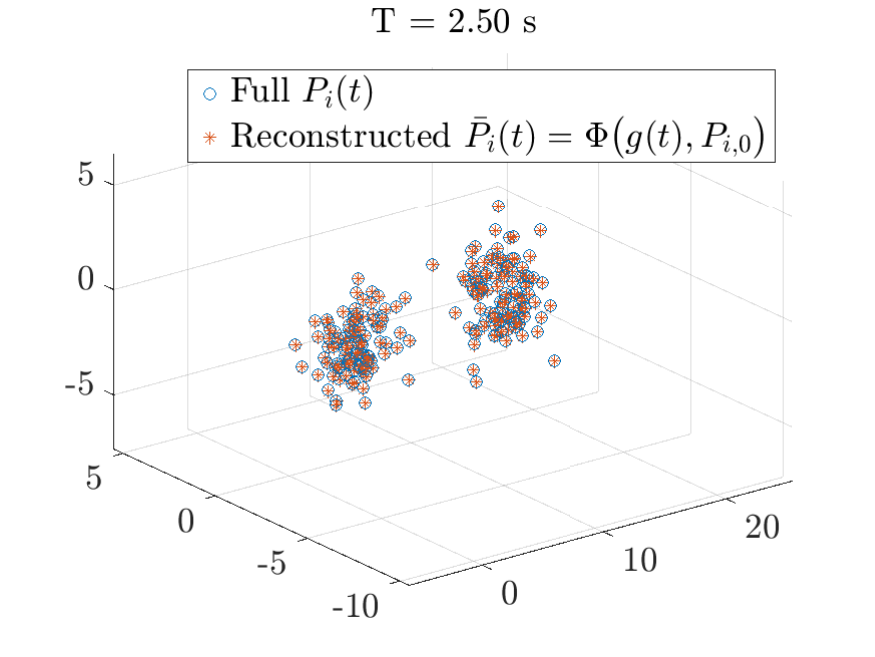
        %\includesvg[height=12em]{Figures/SheeringBodies/ReconstructionSnapshot_3S.svg}
        \subcaption{T = 2.5s.}\label{sfig:sheeringclouds_trajectories_b}
    \end{subfigure}
    &
    \begin{subfigure}[t]{.28\textwidth}
        \def\svgwidth{1.15\textwidth}\scriptsize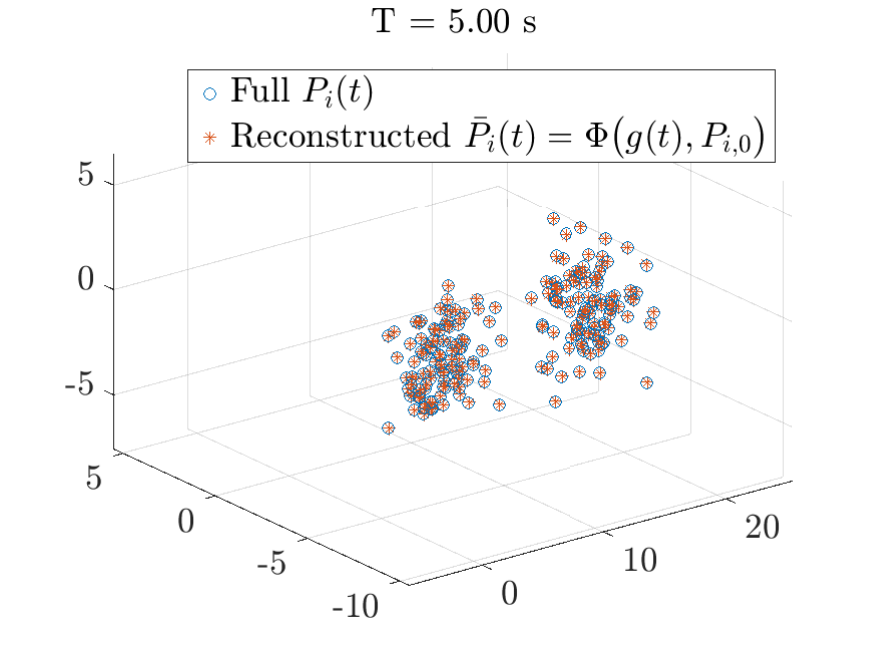
        %\includesvg[height=12em]{Figures/SheeringBodies/ReconstructionSnapshot_5S.svg}
        \subcaption{T = 5s.}\label{sfig:sheeringclouds_trajectories_c}
    \end{subfigure}
    \end{tabular}
    \caption{\small Example~\ref{ssec:sheering-pointclouds}, two sheering pointclouds with noise (blue circles) and reconstructed solution (red stars). }\label{fig:sheeringclouds_trajectories}
\end{figure*}

\begin{figure*}[t]
    \centering
    \begin{tabular}{c c c c}
    \begin{subfigure}[t]{.24\textwidth}
        \includegraphics[height=9em]{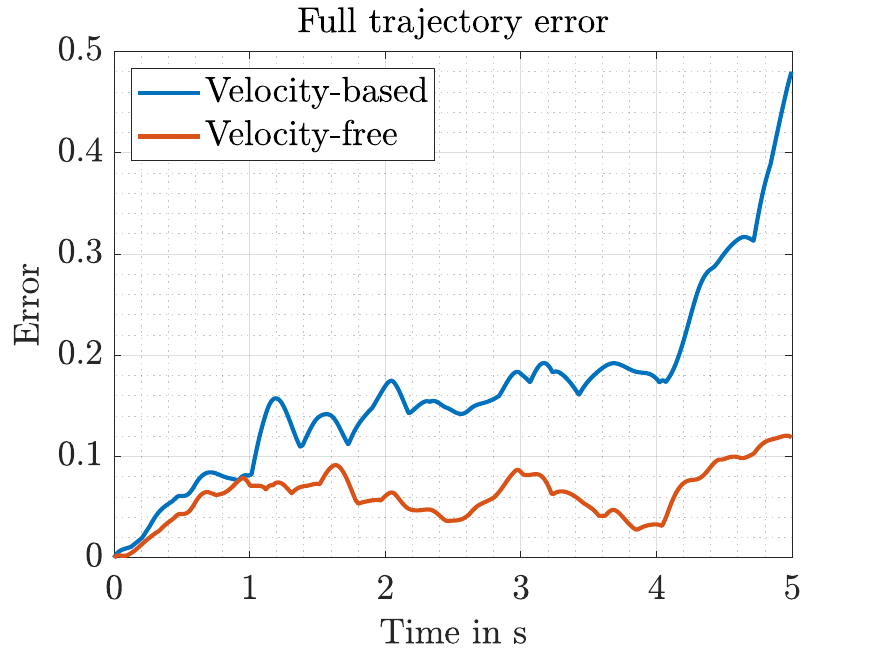}
        %\scriptsize\includesvg[inkscapelatex=false, height=11em]{Figures/SheeringBodies/TT_Error_2025-9-18_at_16-39-46.svg}
        \subcaption{Full trajectory reconstruction errors $\textnormal{dist}\big(\bar{P}(t),P(t)\big)$ for different fits. }\label{sfig:sheeringclouds_error_and_svd_a}
    \end{subfigure}
    &
    \begin{subfigure}[t]{.24\textwidth}
        \includegraphics[height=9em]{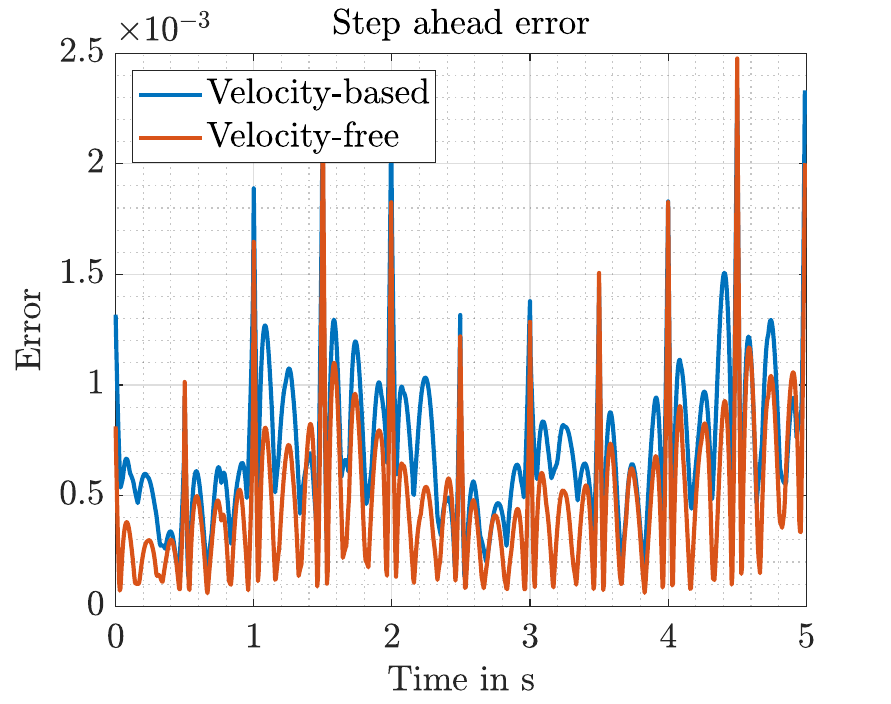}
        %\scriptsize\includesvg[inkscapelatex=false, height=11em]{Figures/SheeringBodies/SP_Error_2025-9-18_at_16-39-46.svg}
        \subcaption{Step ahead reconstruction error $\textnormal{dist}\big(\bar{P}_{P(t)}(\Delta t),P(t+\Delta t)\big)$ for different fits, where $\bar{P}_{P(t)}(s)$ is an approximate solution starting at $P(t)$.}\label{sfig:sheeringclouds_error_and_svd_b} %$\textnormal{dist}\big(\Phi(\Delta g(t),P(t)),P(t+\Delta t)\big)$, $\Delta g(t):= g(t+\Delta t) {g}^{-1}(t)$
    \end{subfigure}
    &
    \begin{subfigure}[t]{.24\textwidth}
        \includegraphics[height=9em]{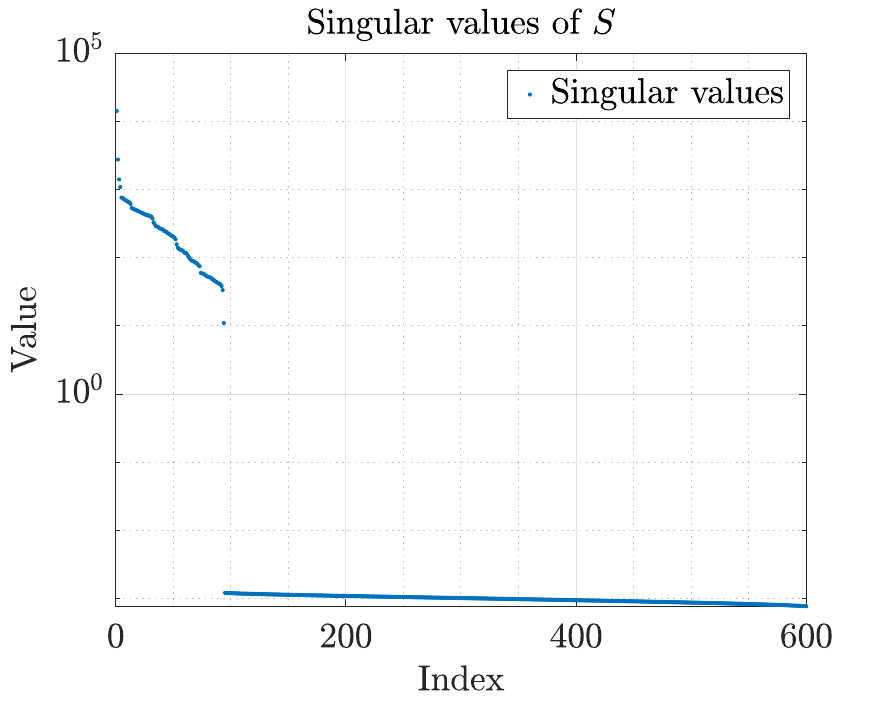}
        %\tiny\includesvg[inkscapelatex=false, height=13em]{Figures/SheeringBodies/SVD_2025-9-18_at_16-39-46.svg}
        \subcaption{Singular values of $S$.}\label{sfig:sheeringclouds_error_and_svd_c}
    \end{subfigure}
    &
    \begin{subfigure}[t]{.24\textwidth}
        \includegraphics[height=9em]{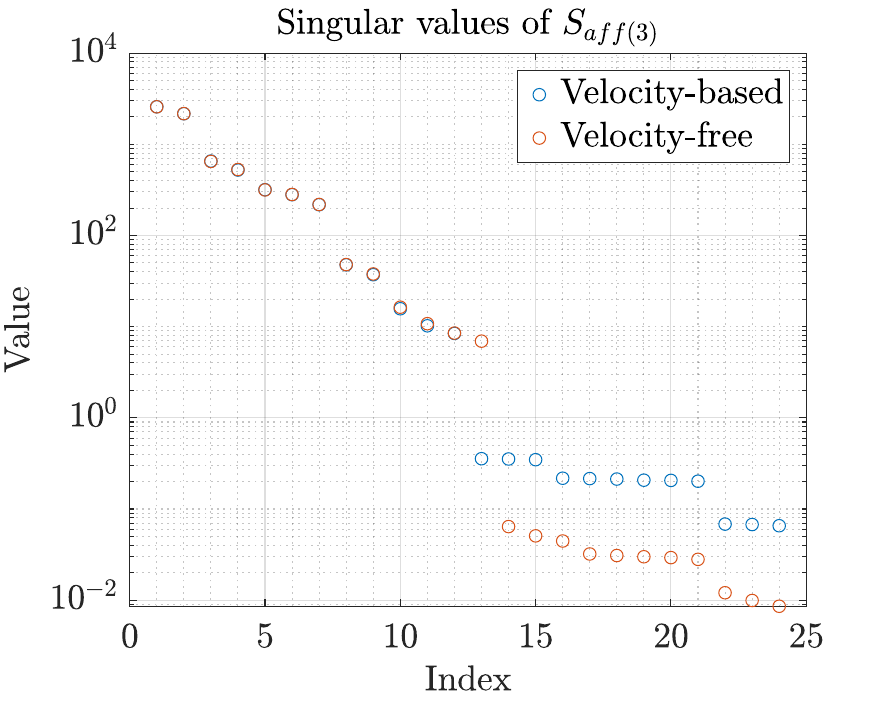}
        %\tiny\includesvg[inkscapelatex=false, height=13em]{Figures/SheeringBodies/SVD_VB_2025-9-18_at_16-39-46.svg}
        \subcaption{Singular values of $S_{\textnormal{aff}^2(3)}$.}\label{sfig:sheeringclouds_error_and_svd_d}
    \end{subfigure}
    \end{tabular}
    \caption{\small Example~\ref{ssec:sheering-pointclouds}, reconstruction errors and singular values for sheering pointclouds.}\label{fig:sheeringclouds_error_and_svd}
\end{figure*}

\subsection{Tracking a liver during respiration}\label{ssec:liver}
We present a medical application using the presented theory, highlighting practical applicability. The example concerns liver-tracking during respiration, a challenging problem that affects surgical procedures due to internal deformation and motion of the flexible liver as a patient breathes during surgery. In our example, we consider a set of points in the form of~\eqref{eq:snapshots-rigid-pointcloud}, collected from edge-tracking of a patient's liver during respiration. A time-dependent trajectory on $SE(3)$ is reconstructed following the steps in Sec.~\ref{ssec:rigid-pointcloud}~--~ this is shown in Figures~\ref{sfig:liver_trajectories_a} to~\ref{sfig:liver_trajectories_c}. The errors of the reconstruction are shown in Figures~\ref{sfig:liver_error_and_svd_a} and~\ref{sfig:liver_error_and_svd_b}. %The results are accurate within required clinical bounds (2mm) for large parts of the trajectory, increasing up to 7mm when internal deformations play a larger role, as shown in Figure~\ref{sfig:liver_error_and_svd_a}. % Figures~\ref{sfig:liver_error_and_svd_c} and~\ref{sub@sfig:liver_error_and_svd_d} highlight that a naive POD of the liver data 

\begin{figure*}[t]
    \centering
    \begin{tabular}{c c c}
    \begin{subfigure}[t]{.28\textwidth}
        \def\svgwidth{1.15\textwidth}\scriptsize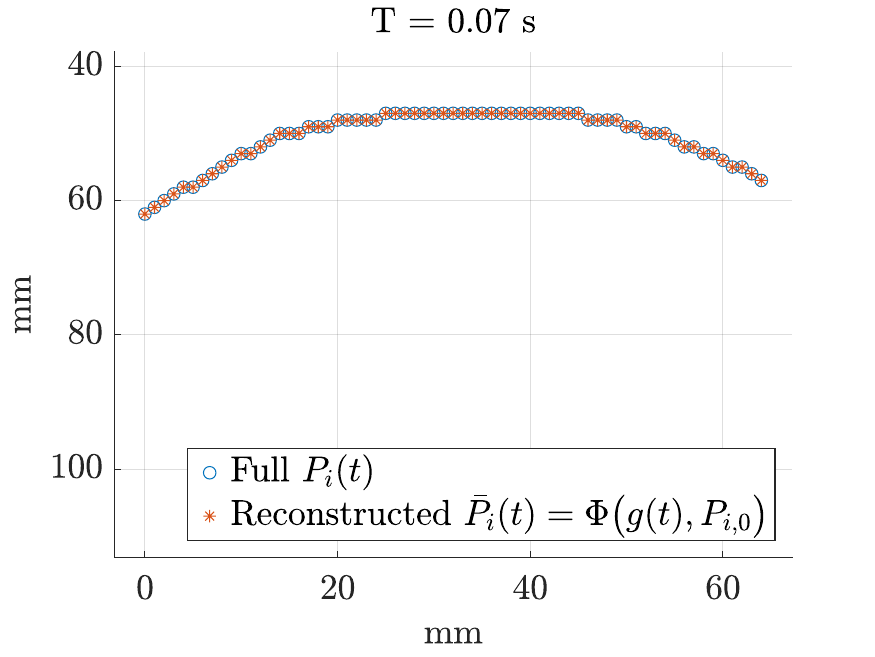
        %\includesvg[height=12em]{Figures/Liver/ReconstructionSnapshot_1L.svg}
        \subcaption{T = 0.07s.}\label{sfig:liver_trajectories_a}
    \end{subfigure}
    &
    \begin{subfigure}[t]{.28\textwidth}
        \def\svgwidth{1.15\textwidth}\scriptsize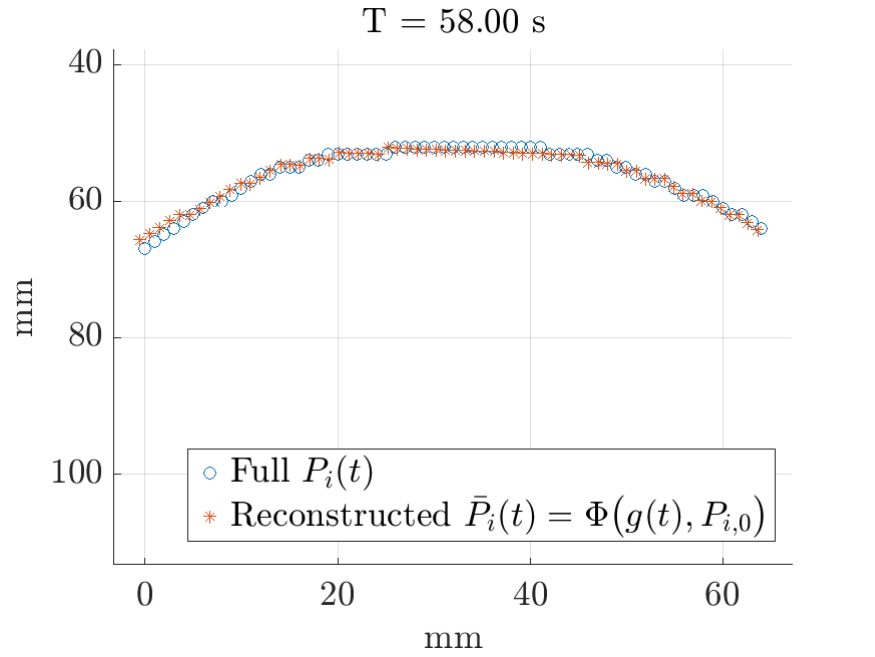
        %\includesvg[height=12em]{Figures/Liver/ReconstructionSnapshot_3L.svg}
        \subcaption{T = 58s.}\label{sfig:liver_trajectories_b}
    \end{subfigure}
    &
    \begin{subfigure}[t]{.28\textwidth}
        \def\svgwidth{1.15\textwidth}\scriptsize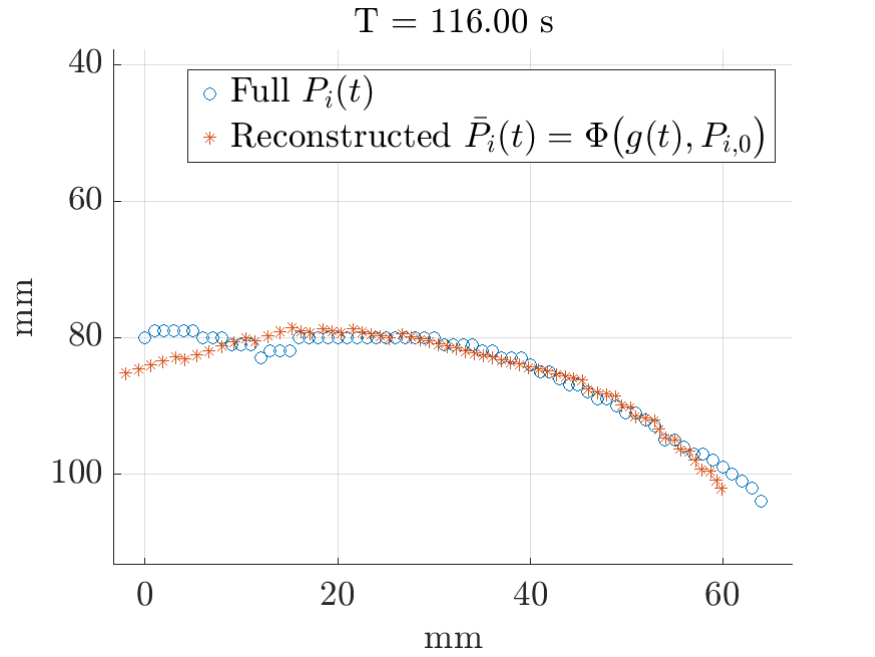
        %\includesvg[height=12em]{Figures/Liver/ReconstructionSnapshot_5L.svg}
        \subcaption{T = 116s.}\label{sfig:liver_trajectories_c}
    \end{subfigure}
    \end{tabular}
    \caption{\small Example~\ref{ssec:liver}, edge-tracking of points on a liver deforming during respiration (blue circles), and reconstructed solution assuming a motion on $SE(3)$ (red stars). }\label{fig:liver_trajectories}
\end{figure*}

\begin{figure*}[t]
    \centering
    \begin{tabular}{c c}
    \begin{subfigure}[t]{.24\textwidth}
        \def\svgwidth{1.1\textwidth}\scriptsize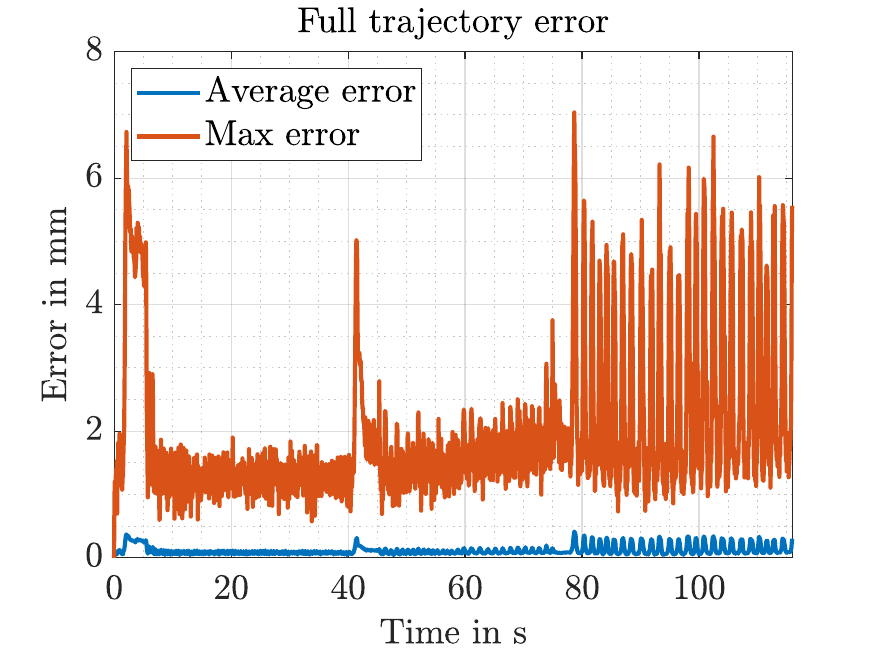
        %\scriptsize\includesvg[height=11em]{Figures/Liver/TT_Error_Best.svg}
        \subcaption{Full trajectory reconstruction errors $\textnormal{dist}\big(\bar{P}(t),P(t)\big)$ for different fits. }\label{sfig:liver_error_and_svd_a}
    \end{subfigure}
    &
    \begin{subfigure}[t]{.24\textwidth}
        \def\svgwidth{1.05\textwidth}\scriptsize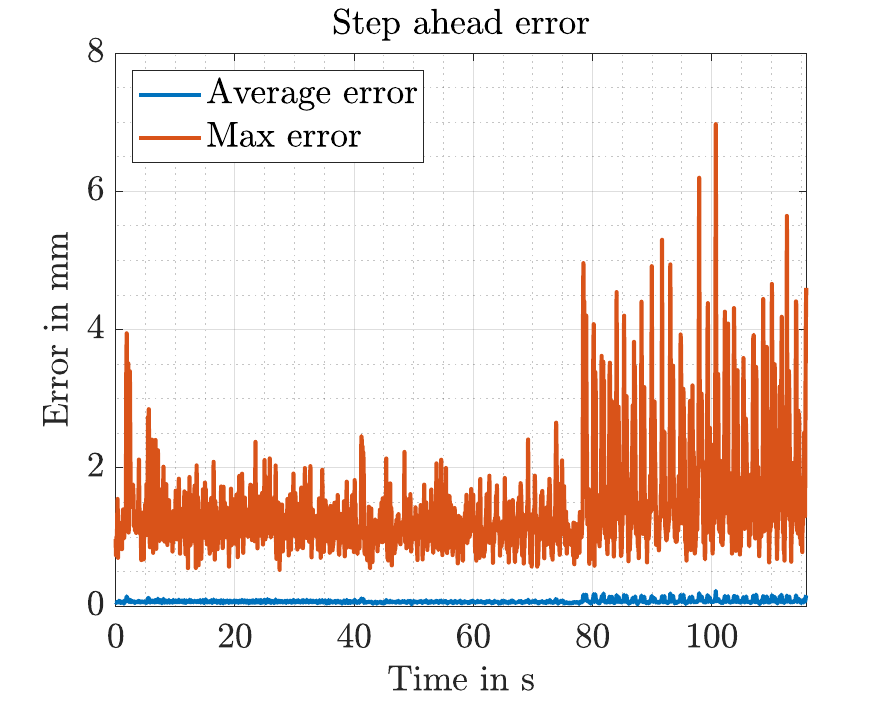
        %\scriptsize\includesvg[height=11em]{Figures/Liver/SP_Error_Best.svg}
        \subcaption{Step ahead reconstruction error $\textnormal{dist}\big(\bar{P}_{P(t)}(\Delta t),P(t+\Delta t)\big)$ for different fits, where $\bar{P}_{P(t)}(s)$ is an approximate solution starting at $P(t)$.}\label{sfig:liver_error_and_svd_b}
    \end{subfigure}
    % &
    % \begin{subfigure}[t]{.24\textwidth}
    %     \tiny\includesvg[inkscapelatex=false, height=13em]{Figures/Liver/SVD.svg}
    %     \subcaption{Singular values of $S$.}\label{sfig:liver_error_and_svd_c}
    % \end{subfigure}
    % &
    % \begin{subfigure}[t]{.24\textwidth}
    %     \tiny\includesvg[inkscapelatex=false, height=13em]{Figures/Liver/SVD_VF.svg}
    %     \subcaption{Singular values of $S_{\textnormal{aff}^2(3)}$.}\label{sfig:liver_error_and_svd_d}
    % \end{subfigure}
    \end{tabular}
    \caption{\small Example~\ref{ssec:liver}, reconstruction errors for example of liver deforming during respiration. For the average error $\textnormal{dist}(\bar{P},P)$ is as in~\eqref{eq:dist_particle_cloud}, and for the maximum error, $\textnormal{dist}(\bar{P},P) = \max_i \|\bar{p}^i - p^i\|$.}\label{fig:liver_error_and_svd}
\end{figure*}

\subsection{Discussion}

The results in Section~\ref{ssec:rigid-pointcloud} show that a MORLie ROM on a Lie group $SE(3)$ was successfully fit to pointcloud data undergoing a rigid motion, without assuming prior knowledge of the particular Lie group. Figure~\ref{sfig:rigidcloud_error_and_svd_d} shows that the dimension of the identified Lie group, 6, reflects in the singular values of the reduced snapshot matrix. Instead, Section~\ref{sssec:naive-POD-rigid-pointcloud} and Figure~\ref{sfig:rigidcloud_error_and_svd_c} show that $9(N_j+1)+3 = 93$ terms are required to capture most complexity of the data with a naive POD. Given more complicated underlying group dynamics requiring larger batches of experimental point cloud data, this is expected to make a naive POD unfeasible for data-based ROM identification~--~increasing the number of experiments should increase ROM accuracy, not the number of terms required. %This is especially important for experimental fluid data from wind-tunnels, where particle-seeding results in random point-clouds for each new batch. It also applies to video-based analysis of rigid and flexible objects, since point-identification is often noisy and unreliable.

We further note that the twist $\widetilde{T}(t)$ had to be chosen of sufficient complexity for Algorithm~\ref{app:alg_subalgebra_search} to correctly identify the subalgebra $se(3)$. For increasing measurement noise, the Algorithm became less reliable and no longer identified $se(3)$ correctly, instead terminating at $\textnormal{aff}(3)$. Future work may investigate more robust version of Algorithm~\ref{app:alg_subalgebra_search}. Similarly, identifying $se(3)$ from Algorithm~\ref{app:alg_subalgebra_search} still required expert knowledge, because the resulting basis elements are non-standard. This may be improved by e.g., comparing the identified sub-algebra against a library of subalgebras.

The results in Section~\ref{ssec:sheering-pointclouds} show that a MORLie ROM on a Lie group $\textnormal{Aff}^2(3)$ can be identified even when there are multiple affinely deforming and moving clusters of particles, and also when there is some overlap between the clusters. This is a first step towards more complex group-actions on clusters of particles, in which non-affine group actions may be considered. 

Both results of Section~\ref{ssec:rigid-pointcloud} and~\ref{ssec:sheering-pointclouds} support that the velocity-free optimization, albeit more computationally expensive, leads to more accurate results. The authors explain this by the velocity-free optimization being more robust to the Gaussian noise $\eta$, which is amplified in the velocity-based method through the numerical differentiation~\eqref{eq:approximate-velocity}. 

Finally, the Section~\ref{ssec:liver} shows that a rigid motion accurately models the livers motion during respiration for shallow breathing (Figures~\ref{sfig:liver_trajectories_a} and~\ref{sfig:liver_trajectories_b}), but not for deep breaths (Figure~\ref{sfig:liver_trajectories_c}), during which the liver significantly deforms. This reflects in the full-trajectory error in Figure~\ref{sfig:liver_error_and_svd_a}, which is accurate to within 2mm for large parts of the trajectory but increases up to 7mm when deformation plays a large role. Future work may investigate nonlinear group actions to more accurately model deformations during respiration.

\section{Conclusion}\label{sec:conclusion}
We presented a novel geometric framework for model order reduction via Lie groups (MORLie), giving a perspective on reduced order modeling that diverges from established linear subspace and submanifold methods for MOR. In this picture, the FOM evolves on a differentiable manifold and the ROM is cast on a Lie group acting on the FOM manifold. We defined a generalized notion of Kolmogorov $N$-widths~--~the group Kolmogorov $N$-width~--~showed how the familiar notion of Kolmogorov $N$-widths for linear subspaces is recovered as a special case, how the Kolmogorov $N$-width of the presented methods is generally lower than the linear subspace Kolmogorov $N$-width, and related it to recent nonlinear widths. We derived optimization methods for intrusive and non-intrusive optimization of the ROM and presented initial ideas for hyperreduction. In an extensive analytic examples, we apply MORLie to the transport equation and show that existing Lie group methods can be recast in the framework of MORLie. Finally, for the example of rigid and affinely evolving point-clouds we computationally implement the MOR process, and show that it scales better than linear subspace approximations, achieving a consistently low dimension.

\section{Acknowledgements}
We want to thank Ana Cord\'{o}n Avila for sharing the data on edge-tracking in MRI scans of a patient's liver during respiration.

% \input{sx_appendix.tex}

% \bibliography{bibliography.bib}
% \bibliographystyle{IEEEtran}

\appendix \label{sec:appendix}

\section{Product reductions}\label{ssec:MORLie_Products}

In Section~\ref{ssec:method_of_freezing}, we combine non-commuting actions $\Phi_G:G\times\M \rightarrow \M$ and $\Phi_H:H\times\M \rightarrow \M$ into an action $\widetilde{\Phi}_{G\times H}:(G\times H) \times \M_G \rightarrow \M_G$. Here, we present the corresponding Theorems, alongside a further technical extension that shows how MORs via different Lie groups, i.e., $(G,\Phi_G,\rho_G)$ and $(H,\Phi_H,\rho_H)$, can be combined into one product reduction $(G\times H,\Phi_{G\times H},\rho_{G\times H})$. %To this end, we present theorems on product MORLie in this Section. 
We first treat the case where $\Phi_G$ and $\Phi_H$ commute. %both $\Phi_G$ and $\Phi_H$ are left actions (see~\eqref{eq:group_homomorphism}), and they commute: 

% then $(G\times H, \Phi_{G\times H}, \rho_{G\times H})$ is a model order reduction of $(\M, X)$ via a Lie group and 
\begin{theorem}[Product ROM for commuting actions]\label{thm:product_reduction}
    Given two ROMs of $\M,\, X\in \Gamma(T\M)$ as $(G,\Phi_G,\rho_G)$ and $(H,\Phi_H,\rho_H)$ 
    \begin{align}
        \Phi_G&:G\times\M\rightarrow\M \,,\; \rho_G:\M\rightarrow\g \,, \\
        \Phi_H&:H\times\M \rightarrow\M\,,\; \rho_H:\M\rightarrow\mathfrak{h}\,,
    \end{align}
    such that $\Phi_H$ and $\Phi_G$ commute:
    \begin{equation}
        \forall g\in G, h \in H: \; \Phi_{G,g} \circ \Phi_{H,h} = \Phi_{H,h} \circ \Phi_{G,g} \,.
    \end{equation}    
    Define $\Phi_{G\times H}: G\times H \times \M \rightarrow \M$ and $\rho_{G\times H}:\M \rightarrow \g \times \mathfrak{h}$ as%i.e.,
    % \begin{equation}
    %     \forall g\in G, h \in H: \; \Phi_G(g,\cdot) \circ \Phi_H(h,\cdot) = \Phi_H(h,\cdot) \circ \Phi_G(g,\cdot) \,,
    % \end{equation}
    % then $(G\times H, \Phi_{G\times H}, \rho_{G\times H})$ is a model order reduction of $(\M, X)$ via a Lie group and
    \begin{align}
        \Phi_{G\times H} \big((g,h),\state\big) &:= \Phi_G\big(g,\Phi_H(h,\state)\big) \,, \label{eq:product_action} \\ % = \Phi_H\big(h,\Phi_G(g,\state)\big)
        \rho_{G\times H}(\state) &:= (\rho_G(\state),\rho_H(\state)) \label{eq:product_reduced_vector_field}\,.
    \end{align}
    Then:
    \begin{enumerate}
        \item $\Phi_{G\times H}: G\times H \times \M \rightarrow \M$ is a group action
        \item $(G\times H, \Phi_{G\times H}, \rho_{G\times H})$ induces the approximated dynamics (cf. Def.~\ref{def:AFOM}) 
        \begin{align}\label{eq:product_AFOM}
            \appX(\state) %= X_{\rho_{G\times H}(\state)}(\state) 
            = X_{\rho_G(\state)}(\state) + X_{\rho_H(\state)}(\state) \,.
            %&\Phi_{G,{*_1}}(e,\appstate)\rho_G(\appstate) \\ + &\Phi_{H,{*_1}}(e,\appstate) \rho_H(\appstate) \,. \nonumber 
        \end{align}
    \end{enumerate}
\end{theorem}
\begin{proof}[Theorem ~\ref{thm:product_reduction}]
    We begin by proving that $\Phi_{G\times H}: G\times H \times \M \rightarrow \M$ is a group action.
    We show that $\Phi_{G\times H}$ fulfills the homomorphism property~\eqref{eq:group_homomorphism}:
    \begin{align*}
         & \Phi_{G\times H}\big((g_1,h_1)\cdot(g_2,h_2), x\big)  = \Phi_{G\times H}\big((g_1 g_2,h_1h_2), x\big) \\
        & = \Phi_{G}\big(g_1g_2, \Phi_H\big(h_1h_2, x \big) \big) \\
        & = \Phi_{G}\big(g_1, \cdot\big)\circ \Phi_{G}\big(g_2, \cdot\big) \circ \Phi_{H}\big(h_1, \cdot\big) \circ \Phi_H\big(h_2, x \big) \\ 
        & = \Phi_{G}\big(g_1, \cdot\big) \circ \Phi_{H}\big(h_1, \cdot\big) \circ \Phi_{G}\big(g_2, \cdot\big) \circ \Phi_H\big(h_2, x \big) \\ 
        & = \Phi_{G\times H}\big((g_1,h_1),\cdot) \circ \Phi_{G\times H}\big((g_2,h_2),x\big)\,.
    \end{align*}
    The first equality holds by definition of the product group, the second and third equality hold by definition~\eqref{eq:product_action}, while the fourth equality makes use of the commutativity of $\Phi_G$ and $\Phi_H$. Finally, the fifth equality reapplies the definition~\eqref{eq:product_action}.
    Next, we show that $(G\times H, \Phi_{G\times H}, \rho_{G\times H})$ induces the approximated dynamics 
    \begin{align}
        \appX(\appstate) %= X_{\rho_{G\times H}(\appstate)}(\appstate) 
        = X_{\rho_G(\appstate)}(\appstate) + X_{\rho_H(\appstate)}(\appstate) \,.
        %&\Phi_{G,{*_1}}(e,\appstate)\rho_G(\appstate) \\ + &\Phi_{H,{*_1}}(e,\appstate) \rho_H(\appstate) \,. \nonumber 
    \end{align}
    By Definition~\ref{def:AFOM}, the approximated dynamics induced by $(G\times H, \Phi_{G\times H}, \rho_{G\times H})$ are given by the infinitesimal generator $X_{G\times H, \rho_{G\times H}(\appstate)}(\appstate)$:
    \begin{align*}
        \appX(\appstate) &= X_{G\times H, \rho_{G\times H}(\appstate)}(\appstate) \\ 
        &= \frac{\ext}{\ext s} \Phi_{G\times H}\big(e^{\rho_{G\times H}(\appstate)s},\appstate\big) \\
        &= \frac{\ext}{\ext s} \Phi_{G\times H}\big((e^{\rho_{G}(\appstate)s}, e^{\rho_{H}(\appstate)s}),\appstate\big) \\ 
        &= \frac{\ext}{\ext s} \Phi_{G}\big(e^{\rho_{G}(\appstate)s},\Phi_{H}(e^{\rho_{H}(\appstate)s},\appstate)\big) \\
        &= \frac{\ext}{\ext s} \Phi_{G}\big(e^{\rho_{G}(\appstate)s},\appstate\big) + \frac{\ext}{\ext s} \Phi_{H}\big(e^{\rho_{H}(\appstate)s},\appstate\big) \\
        &=  X_{\rho_G(\appstate)}(\appstate) + X_{\rho_H(\appstate)}(\appstate) \,.
        % &= \Phi_{G\times H,*_1}\big((e,e),\appstate\big)\rho_{G\times H}(\appstate) \\
        % &= \Phi_{G,*_1}\big(e,\appstate\big)\rho_{G}(\appstate) + \Phi_{H,*_1}\big(e,\appstate\big)\rho_{H}(\appstate) \\
        % &=  X_{\rho_G(x)}(x) + X_{\rho_H(x)}(x) \,.
    \end{align*}
\end{proof}

We call the tuple $(G\times H, \Phi_{G\times H}, \rho_{G\times H})$ the \textbf{product ROM} of $(G,\Phi_G,\rho_G)$ and $(H,\Phi_H,\rho_H)$, and we call~\eqref{eq:product_AFOM} the \textbf{product approximated dynamics}. A product reconstruction Theorem~\ref{thm:product_reconstruction} follows as an analog of Theorem~\ref{thm:MorLie reconstruction}:

\begin{theorem}[Product reconstruction]\label{thm:product_reconstruction}
    Given a product reduction $(G\times H, \Phi_{G\times H}, \rho_{G\times H})$ and the $\appX(\appstate)$ in~\eqref{eq:product_AFOM}, then the integral curves of 
    \begin{equation}
        \dot{\appstate} = \appX(\appstate) \,,\; \appstate(0) = \state_0\,,
    \end{equation}
    are given by 
    \begin{equation}\label{eq:product_reconstruction}
        \appstate = \Phi_{G\times H}\big((g,h)(t),\state_0\big)\,,
    \end{equation}
    where $g(t),h(t)$ satisfy (for $\Phi_G,\Phi_H$ left actions):
    \begin{align} 
        \dot{g} = {R_g}_* \rho_G\big(\appstate(t)\big)\,,\; g(0) = e\,, \label{eq:product_group_dynamics_g}\\
        \dot{h} = {R_h}_* \rho_H\big(\appstate(t)\big)\,,\; h(0) = e\,. \label{eq:product_group_dynamics_h}
    \end{align}
    When either $\Phi_G$ or $\Phi_H$ is a right action, ${R_g}_*$ needs to be replaced by ${L_g}_*$ in~\eqref{eq:product_group_dynamics_g} or ${R_h}_*$ needs to be replaced by ${L_h}_*$ in~\eqref{eq:product_group_dynamics_h}, respectively.
\end{theorem}

% \begin{proof}
%     See Appendix.
% \end{proof}

\begin{proof}[Theorem~\ref{thm:product_reconstruction}]
    To show that the reconstructed trajectory~\eqref{eq:product_reconstruction} is a solution of the dynamics~\eqref{eq:product_AFOM},
    we show that its tangent vector at the point $\appstate \in \M$ is given by $\bar{X}(\appstate)$. To this end, differentiate: %an arbitrary function $f \in C^\infty(\M)$ along $\appstate(t)$:
    \begin{align}
        \frac{\textnormal{d}}{\textnormal{d}t} \appstate(t)
        = &\, 
        \frac{\textnormal{d}}{\textnormal{d}t} \Phi_{G\times H}\big((g(t),h(t)),\state_0\big) \\ 
        = &\, \frac{\textnormal{d}}{\textnormal{d}t} \Phi_{G\times H}\big((g(t),h),\state_0\big) \nonumber \\
        &\quad + \frac{\textnormal{d}}{\textnormal{d}t} \Phi_{G\times H}\big((g,h(t)),\state_0\big) \nonumber \\
        = &\, \frac{\textnormal{d}}{\textnormal{d}t} \Phi_G\big(g(t),\Phi_H(h,\state_0)\big) \nonumber \\
        &\quad + \frac{\textnormal{d}}{\textnormal{d}t} \Phi_H\big(h(t),\Phi_G(g,\state_0)\big) \nonumber \\
        % = &\, \frac{\textnormal{d}}{\textnormal{d}s} \Phi_G\big(e^{\rho_G(\appstate)s}g(t),\Phi_H(h,\state_0)\big)_{|s=0} \nonumber \\
        % & + \frac{\textnormal{d}}{\textnormal{d}s} \Phi_H\big(e^{\rho_H(\appstate)s}h(t),\Phi_G(g,\state_0))_{|s=0} \nonumber \\
        %& = \frac{\textnormal{d}}{\textnormal{d}s} f\big(\Phi_G\big(\exp{\dot{g}(t)s,\Phi_H(h,\state_0)\big)\big)_{|s = 0}  + \frac{\textnormal{d}}{\textnormal{d}s} f\big(\Phi_H\big(\exp{\dot{h}(t)s,\Phi_G(g,\state_0)\big)\big)_{|s = 0} \\
        = &\, \frac{\textnormal{d}}{\textnormal{d}s} \Phi_G\big(e^{\rho_G(\appstate)s},\appstate\big)_{|s=0} \nonumber \\
        &\quad + \frac{\textnormal{d}}{\textnormal{d}s} \Phi_H\big(e^{\rho_H(\appstate)s},\appstate)_{|s=0} \nonumber \\
        = &\,  X_{\rho_G(\appstate)}(\appstate) + X_{\rho_H(\appstate)}(\appstate)\nonumber \\
        = &\, \appX(\appstate)\,. \nonumber
    \end{align}
    Where the third equality uses that the group actions commute, and the fourth equality reuses that 
    \begin{align}
        \frac{\textnormal{d}}{\textnormal{d}t} \Phi_G\big(g(t),\state\big) = \frac{\textnormal{d}}{\textnormal{d}s}\Phi_G\big(e^{\rho_G(\appstate)s},\Phi_G(g(t),\state)\big)_{|s=0}\,.
    \end{align}
\end{proof}

%Theorem~\ref{thm:product_reconstruction} is stated for the case where both $\Phi_G$ and $\Phi_H$ are left actions, corresponding to the use of ${R_g}_*$ in~\eqref{eq:product_group_dynamics_g} and ${R_h}_*$ in~\eqref{eq:product_group_dynamics_h}. 

Given Lie groups $G,H$ and actions $\Phi_G:G\times\M\rightarrow\M\,,\, \Phi_H:H\times\M\rightarrow\M$ that do not commute, 
Then Theorem~\ref{thm:product_reduction} does not immediately apply. We make use of the following lemma:

\begin{lemma}\label{lemma:commuting_actions_by_augmentation}
    Given any actions $\Phi_G:G\times\M\rightarrow\M$ and $\Phi_H: H\times \M \rightarrow\M$.
    % that do not commute:
    % \begin{equation}
    %     \exists g\in G, h\in H:\;\Phi_{H,h}\circ\Phi_{G,g} \neq \Phi_{G,g} \circ \Phi_{H,h}\,.
    % \end{equation}
    Define the \textbf{$G$-augmented} manifold 
    \begin{equation}
        \M_G := G \times \M\,.
    \end{equation} 
    Then the actions $\widetilde{\Phi}_G:G\times\M_G\rightarrow\M_G$ and $\widetilde{\Phi}_H:H\times\M_G\rightarrow\M_G$ defined by
    \begin{align}
        \widetilde{\Phi}_G \big(g_2,(g_1,\state)\big) &= (g_2 g_1, g_2 \state)\,, \\
        \widetilde{\Phi}_H \big(h, (g,\state)\big) &= (g, ghg^{-1}\state)
    \end{align}
    commute.
\end{lemma}
Here, $\widetilde{\Phi}_G$ acts on the second element $\M$ exactly as $\Phi_G$ does, and $\widetilde{\Phi}_H$ acts on the second element of $(e,x) \in \M_G$ exactly as $\Phi_H$ does. This construction of commuting actions is not unique, an alternative construction uses $\M_H$ and arrives at different augmented actions. 

Given this construction, a product of $(G,\Phi_G,\rho_G)$ and $(H,\Phi_H,\rho_H)$ can be found: 
\begin{theorem}[Product reduction for non-commuting actions]\label{thm:product_reduction_nc}
    Given two reductions of $\M,\, X\in \Gamma(T\M)$ as $(G,\Phi_G,\rho_G)$ and $(H,\Phi_H,\rho_H)$ 
    \begin{align}
        \Phi_G&:G\times\M\rightarrow\M \,,\; \rho_G:\M\rightarrow\g \,, \\
        \Phi_H&:H\times\M \rightarrow\M\,,\; \rho_H:\M\rightarrow\mathfrak{h}\,.
    \end{align}
    Define actions $\widetilde{\Phi}_G:G\times\M_G\rightarrow \M_G$ and $\widetilde{\Phi}_H:H\times\M_G\rightarrow \M_G$ as in Lemma~\ref{lemma:commuting_actions_by_augmentation}, and define $\widetilde{\Phi}_{G\times H}: G\times H \times \M_G \rightarrow \M_G$ and $\rho_{G\times H}:\M_G \rightarrow \g \times \mathfrak{h}$ as
    \begin{align}
        \widetilde{\Phi}_{G\times H} \big((g_2,h),(g_1,\state)\big) :=&\, \widetilde{\Phi}_G\big(g_2,\widetilde{\Phi}_H(h,(g_1,\state))\big) \label{eq:product_action_nc} \\ 
        %=&\, \widetilde{\Phi}_H\big(h,\widetilde{\Phi}_G(g_2,(g_1,\state))\big)\,, \nonumber \\
        \widetilde{\rho}_{G\times H}((g,\state)) :=&\, (\rho_G(\state),\rho_H(\state)) \label{eq:product_reduced_vector_field_nc}\,.
    \end{align}
    Then the approximated dynamics induced by $(G\times H, \widetilde{\Phi}_{G\times H}, \widetilde{\rho}_{G\times H})$ is $\appX\in\Gamma(\Delta)$ with $\Delta\subseteq T\M_G$: 
    \begin{equation}
        %\pmat{\dot{g}\\\dot{x}}
        \appX\big((g,\state)\big) = \pmat{ 
            {R_g}_*\rho_G(\state) \\ X_{\rho_G(\appstate)}(\appstate) + g\cdot X_{\rho_H(\appstate)}(g^{-1}\cdot\appstate) %X_{\big(\rho_G(\state),\rho_H(\state)\big)}(\state)
        } \,. 
    \end{equation}
    %Then $(G\times H, \widetilde{\Phi}_{G\times H}, \widetilde{\rho}_{G\times H})$ is a MOR of $(\M_G, (\Xi\rho_G,X))$ via a Lie group $(G\times H)$ called the \textbf{$\M_G$-related product} of $(G,\Phi_G,\rho_G)$ and $(H,\Phi_H,\rho_H)$.
\end{theorem} 
\begin{proof}
    By Definition~\ref{def:AFOM}, the approximated dynamics induced by $(G\times H, \widetilde{\Phi}_{G\times H}, \widetilde{\rho}_{G\times H})$ are given by the infinitesimal generator $X_{G\times H, \widetilde{\rho}_{G\times H}((g,\appstate))}((g,\appstate))$:
    \begin{align*}
        \appX\big((g,\appstate)\big) &= X_{G\times H, \widetilde{\rho}_{G\times H}((g,\appstate))}\big((g,\appstate)\big) \\ 
        &= \frac{\ext}{\ext s} \widetilde{\Phi}_{G\times H}\big(e^{\widetilde{\rho}_{G\times H}((g,\appstate))s},(g,\appstate)\big) \\
        &= \frac{\ext}{\ext s} \widetilde{\Phi}_{G\times H}\big((e^{\rho_{G}(\appstate)s}, e^{\rho_{H}(\appstate)s}),(g,\appstate)\big) \\ 
        &= \frac{\ext}{\ext s} \widetilde{\Phi}_{G}\big(e^{\rho_{G}(\appstate)s},\widetilde{\Phi}_{H}(e^{\rho_{H}(\appstate)s},(g,\appstate))\big) \\
        &= \frac{\ext}{\ext s} \widetilde{\Phi}_{G}\big(e^{\rho_{G}(\appstate)s},(g,\appstate)\big) \\ 
        & \quad 
        + \frac{\ext}{\ext s} \widetilde{\Phi}_{H}\big(e^{\rho_{H}(\appstate)s},(g,\appstate)\big) \\
        &= \frac{\ext}{\ext s} (e^{\rho_{G}(\appstate)s} g, \Phi_{G}\big(e^{\rho_{G}(\appstate)s},\appstate)\big) \\ 
        & \quad + \frac{\ext}{\ext s} (g, g\cdot \Phi_{H}\big(e^{\rho_{H}(\appstate)s},g^{-1} \cdot \appstate)\big) \\
        &= 
        \pmat{ 
            {R_g}_*\rho_G(\appstate) \\ 
            X_{\rho_G(\appstate)}(\appstate)
        }
        +
        \pmat{ 0 \\
            g\cdot X_{\rho_H(\appstate)}(g^{-1}\cdot\appstate)
         } \,.
    \end{align*}
\end{proof}
We call $(G\times H, \widetilde{\Phi}_{G\times H}, \widetilde{\rho}_{G\times H})$ the \textbf{$\M_G$-related product} of $(G,\Phi_G,\rho_G)$ and $(H,\Phi_H,\rho_H)$. 
Also the $\M_H$-related product may be found. The differences between the two are explored in Section~\ref{sec:analytic_examples}. 
The product reconstruction Theorem~\ref{thm:product_reconstruction} can be reused to compute integral curves on $\M_G$.

\begin{remark}
    The concept of a product reduction similarly induces larger group-action pairs (cf.\ Def.~\ref{def:induced_admissible_pairs}). Given group and action pairs $(G,\Phi_G)$, $(H,\Phi_H)$ whose actions commute, then $\Xi^{(G\times H, \Phi_{G\times H})} = \{ (\widetilde{G}\times\widetilde{H}, \Phi_{G\times H})\;|\; \widetilde{G} \subseteq G, \widetilde{H} \subseteq H \}$ has size $|\Xi^{(G\times H, \Phi_{G\times H})}| = |\Xi^{(G, \Phi_{G})}| |\Xi^{(H, \Phi_{H})}|$. Similar results hold when the actions do not commute.
\end{remark}

% To this end, define the $G$-augmented manifold 
% \begin{equation}
%     \M_G := G \times \M\,.
% \end{equation}

% It can be shown that a MorLie reduction of $\M_G$, that combines $(G,\Phi_G,\rho_G)$ and $(H,\Phi_H,\rho_H)$, can be constructed:
% \begin{theorem}
%     The actions 
%     \begin{equation}
%         \widetilde{\Phi}_{G,H}: H \times \M_G \rightarrow \M_G \,, (h, (g,\state)) \mapsto (g, \Phi_{G,g} \Phi_{H,h} \Phi_{G,g}^{-1} \state)
%     \end{equation}
%     and 
%     \begin{equation}
%         \widetilde{\Phi}_{G}: G \times \M_G \rightarrow \M_G\,, (g_2, (g_1,\state)) \mapsto (g_2 g_1, \Phi_G(g_2,\state))
%     \end{equation}
%     commute.
% \end{theorem}
% \begin{theorem}[Product reduction for non-commuting actions]
%     The reductions $(H,\widetilde{\Phi}_{G,H},\widetilde{\rho}_H)$ and $(G,\widetilde{\Phi}_G, \widetilde{\rho}_G)$ have commuting actions 
% \end{theorem}

\section{Code}\label{ssec:appendix-code}
\subsection{Subalgebra search}\label{ssec:code-subalgebra}
We present an algorithm that, given a finite collection $S_\g$ of elements in a Lie algebra $\g$ and an inner product on $\g$, finds a subalgebra $\mathfrak{h} \subseteq \g$ that contains the first $k$ singular vectors of $S_\g$.

\begin{algorithm}[H]
    \caption{Subalgebra search}\label{app:alg_subalgebra_search}
    \hspace*{\algorithmicindent} \textbf{Input}: $S_\g$ \\
    \hspace*{\algorithmicindent} \textbf{Output}: $\mathfrak{h}\subseteq \g$ 
    \begin{algorithmic}[1]
        \State $S_\g \gets \textnormal{Collection of algebra elements}$
        \State $d_{\textnormal{SVD}} \gets \textnormal{Threshold for singular vectors}$
        \State $k, \{\widetilde{A}_1,\cdots, \widetilde{A}_k\} \gets \textnormal{SVD}(S_\g, d_k)$
        \State $\h \gets \{\widetilde{A}_1,\cdots, \widetilde{A}_k\}$
        \Repeat
            \State $\h \gets \textnormal{Basis}(\textnormal{Bracket}(\h,\h))$
        \Until $\textnormal{span}(\h) == \textnormal{span}(\textnormal{Bracket}(\h,\h))$
    \end{algorithmic}
\end{algorithm}
Here, $\textnormal{Bracket}(\h,\h)$ returns $[\h,\h]\oplus\h$, i.e., it uses the Lie Bracket on $\g$ to bracket all elements in the basis of $\h$ with each other and expands the set $\h$ by the result. The algorithm is guaranteed to terminate with $\h = \g$, for finite $\g$, but may also terminate at $\h \subset \g$.
Accompanying code also shows how the basis of the resulting $\h$, together with the Lie-bracket on $\g$ can be used to construct explicitly maps such as $\exp:\h \rightarrow H$, $\log: U\subseteq H \rightarrow\h$ and $\ext\exp:T \h \rightarrow \h$ for numerical integration in local charts on $H$, for $G = GL(n)$. This can be more efficient than restricting the domain and co-domain of $\exp:\mathfrak{gl}(n)\rightarrow GL(n)$, $\log: U\subseteq GL(n)\rightarrow \mathfrak{gl}(n)$ and $\ext\exp:T \mathfrak{gl}(n)\rightarrow \mathfrak{gl}(n)$.

\subsection{Group and action search}\label{ssec:code-group-action}
We present an algorithm that, given a collection $S$ of points $P_{j,k}\in\R^{3N_i}$ (cf. Sec.~\ref{ssec:sheering-pointclouds}), splits them into distinct clusters and thus identifies the number of copies of the affine group $Aff(3)$ whose action can describe the motion of the points.

\begin{algorithm}[H]
    \caption{Clustering}\label{app:cluster_search}
    \hspace*{\algorithmicindent} \textbf{Input}: $S$ \\
    \hspace*{\algorithmicindent} \textbf{Output}: $n_G$, $S_1,\cdots, S_{n_G}$ 
    \begin{algorithmic}[1]
        \State $S \gets \textnormal{points } P_{j,k}\in \R^{3 N_i}$
        \State $N_n \gets \textnormal{Number of nearest neighbors}$
        \State $k \gets 0$
        \Repeat
            \State $S_k \gets \textnormal{RandomCluster}(S,N_n)$
            \State $X_k \gets \textnormal{FitInfinitesimalGenerator}(S_k)$
            \State $S_k \gets \textnormal{FilterByGenerator}(S,X_k)$
            \State $S \gets S\backslash S_k$
            \State $k \gets k + 1$
        \Until $|S| = 0$
        \State $n_G \gets k$
    \end{algorithmic}
\end{algorithm}
Here, $\textnormal{RandomCluster}(S,N_n)$ determines a local cluster of $N_n$ nearest neighbors, $\textnormal{FitInfinitesimalGenerator}(S_k)$ fits an infinitesimal generator to the velocities in $S_k$, and $\textnormal{FilterByGenerator}(S,X_k)$ returns the trajectories whose velocities are described by $X_k$. The output of the algorithm is that the set $S$ is split into distinct groups $S_1,\cdots, S_{n_G}$. % their reduced snapshot matrices are computed, applying Theorem $(w.r.t. $Aff(3)$)

\section{Singular values of rigid snapshots}\label{ssec:appendix-rigid-svd}

We prove Theorem~\ref{thm:singular-values-rigid-snapshots}. For convenience, we state the Theorem again:

\begin{theorem}
    Assume the following:
    \begin{enumerate}
        \item The number of columns $N_jN_k$ and the number of rows $3N_i+1$ of $S$ are sufficiently large: $N_jN_k \geq 9(N_j+1)+3$ and $3N_i+1 \geq 9(N_j+1)+3$
        %\item The noise is $\eta^i_{j,k} = 0$
        \item Each initial point cloud $P_{j,0} \in \R^{N_j}$ spans $\R^3$: $$\dim\textnormal{span}\{p^i_{j,0}\;|\;i \in \{0,\cdots,N_i\}\} = 3$$
        \item The matrices $R(t_k)$ span $\R^{3\times 3}$: $$\dim\textnormal{span}\{R(t_k)\;|\;k \in \{0,\cdots,N_k\}\} = 9$$
        \item The translations $b(t_k)$ span $\R^3$: $$\dim\textnormal{span}\{b(t_k)\;|\;k \in \{0,\cdots,N_k\}\} = 3$$
        \item The subspaces $$W_j := \textnormal{span}\{P_{j,k}\;|\; k \in \{0,\cdots,N_k\} \}$$ of $\R^{3N_i}$ associated with different point cloud trajectories $P_{j,k}$ are independent. 
    \end{enumerate}  
    Then the number of non-zero singular values of the snapshot set $S$ in~\eqref{eq:snapshots-rigid-pointcloud} is $9(N_j+1)+3$.
\end{theorem}

\begin{proof}
    The number of non-zero singular values of a matrix is bounded by the number of columns $N_jN_k$, the number of rows $3N_i+1$, and the column rank of $S$. Given Assumption 1. that $N_jN_k$ and $3N_i+1$ are sufficiently large, we show that the column rank of $S$ is $9(N_j+1)+3$.
    
    Each column of $S$ corresponding to trajectory $j$ at time $t_k$ has the form
    \begin{equation}
    P_{j,k}     =
    \begin{pmatrix}
    R(t_k)p^1_{j,0} \\ 
    \vdots \\
    R(t_k)p^{N_i}_{j,0}
    \end{pmatrix}
    +
    \begin{pmatrix}
    b(t_k) \\
    \vdots \\
    b(t_k)
    \end{pmatrix}.
    \end{equation}
    We analyze rotational and translational contributions separately, first focusing on the rotational term.

    Write as $e_n \in \R^3$ the $n$-th unit vector, and define vectors $\bar P^{(j)}_{nm} \in \mathbb{R}^{3N_i}$ by
    \begin{equation}
        \bar P^{(j)}_{nm}
        =
        \begin{pmatrix}
        (e_n^\top p^1_{j,0}) e_m \\
        \vdots \\
        (e_n^\top p^{N_i}_{j,0}) e_m
        \end{pmatrix}\,.
    \end{equation}
    Then
    \begin{align}
        P^{\mathrm{rot}}_{j,k} = 
            \pmat{
            R(t_k)p^1_{j,0} \\ 
            \vdots \\
            R(t_k)p^{N_i}_{j,0} }
        &= \sum_{n,m} R_{nm}(t_k)
            \pmat{ 
            (e_n^\top p^1_{j,0}) e_m \\
            \vdots \\
            (e_n^\top p^{N_i}_{j,0}) e_m } \\ 
        &= \sum_{n,m} R_{nm}(t_k) \bar{P}^{(j)}_{nm} \,.
    \end{align}
    Hence the rotational part lies in the span of at most $9$ vectors. We now show that these are linearly independent. Suppose that
    \begin{equation}
    \sum_{n,m=1}^{3} a_{nm}\bar P^{(j)}_{nm}=0.
    \end{equation}
    Then for each particle $p^i_{j,0}$
    \begin{equation}
    \sum_{n,m} a_{nm}(e_n^\top p^i_{j,0}) e_m = 0\,.
    \end{equation}
    Define
    \begin{equation}
    A := \sum_{n,m} a_{nm} e_m e_n^\top \in \mathbb{R}^{3\times3}.
    \end{equation}
    The previous equation is equivalent to
    \begin{equation}
    A p^i_{j,0} = 0 \quad \text{for all } i.
    \end{equation}
    By Assumption 2., the initial cloud spans $\mathbb{R}^3$, so this would imply $A=0$, and therefore $a_{nm}=0$ for all $n,m$. Thus, the nine vectors $\bar P^{(j)}_{nm}$ are linearly independent.

    By Assumption 3., the rotational component for trajectory $j$ spans a $9$-dimensional subspace of $\mathbb{R}^{3N_i}$.

    By Assumption 5., it follows that $9$ linearly independent components are required for the rotational component of each of the $N_j + 1$ trajectories, and we arrive at $9(N_j+1)$ linearly independent components to express the rotational component of $S$. 
    
    By Assumption 4., the translations can be expressed by three linearly independent vectors $\bar{P}_{b,1},\cdots, \bar{P}_{b,3}$ for all trajectories, bringing the maximum number of orthogonal components to $9(N_j+1)+3$. 
\end{proof}
 
\section{Model order reduction on manifolds}\label{sec:manimor}

\subsection{ManiMOR}
%Describe ManiMOR, especially general construction of ROM as vector field on a submanifold $\mathcal{N} \subseteq \M$.\@ 
We briefly describe the recent differential geometric framework for MOR on manifolds (ManiMOR)~\cite{Buchfink2023a}.

Here, the full order model is of the form~\eqref{eq:FOM}, along with a set of solution snapshots of the form~\eqref{eq:solution-snapshots}. The ROM is then described as a vector field on a reduced order manifold $\mathcal{N}$ which is related to the full order manifold $\mathcal{M}$ by an embedding $\varphi:\mathcal{N}\rightarrow\mathcal{M}$. 

The vector field $\appX\in\X(\mathcal{N})$ is given by a projection $\Pi \in C^\infty(T\M,T\mathcal{N})$ of the full order vector field, and is defined as
\begin{equation}
    \appX(\stateb) = \Pi X(\varphi(\stateb))\,,
\end{equation}
where the projection operator $\Pi$ satisfies the projection property:
\begin{equation}
    \Pi \circ \varphi_* = \textnormal{id}_{T\mathcal{N}}  \,.
\end{equation}
With respect to the main-text we note that the projection operator $\Pi_\Delta$ in Theorem~\ref{thm:explicit-rho-intrusive} corresponds to $\Pi$, picking for $\mathcal{N}$ any orbit $\mathcal{N} = \mathcal{O}(\state)$ and for $\varphi:\mathcal{O}(\state) \rightarrow \M$ the canonical embedding. 
For more information on ManiMOR we refer to~\cite{Buchfink2023a}.

\subsection{From submanifolds to distributions to Lie groups}
In ManiMOR, the ROM is described on a single submanifold $\varphi(\mathcal{N}) \subseteq \M$. Instead, we want to investigate if MOR can restrict dynamics in $\Gamma(T\M)$ to $\Gamma(\Delta)$ with $\Delta \subseteq T\M$ a distribution. If $\Delta$ is regular and integrable, there will be a family of $k$-dimensional manifolds $\mathcal{N}_\state \subseteq \M$ with constant $k = \dim \Delta(\state)$ that contains all solutions for dynamics in $\Gamma(\Delta)$.  

It is apriori unclear how to coordinatize a family of submanifolds $\mathcal{N}_\state$, unless they are all equivalent and thus admit equivalent coordinates. Thus, we are looking for distributions $\Delta$ that induce a foliation into equivalent submanifolds $\mathcal{N}_\state$, i.e., we are looking for a subset of regular, integrable distributions. 

Finite-dimensional Lie algebras on a manifold $\M$ are a special case of such regular, integrable distributions. They are induced by actions of a group $G$ on the manifold $\M$ that are both free and proper, in which case the infinitesimal generators of the Lie algebra elements in $\g$ form a Lie algebra of vector fields on $\M$ that spans a regular, integrable distribution. For free actions, the submanifolds $\mathcal{N}_\state = \mathcal{O}(\state)$ are all equivalent to $G$, and thus admit equivalent coordinatizations induced by $G$. This serves as a high-level motivation to study the action of Lie groups for MOR, pointing out that the decision to describe ROMs on Lie groups is still a special case from descriptions on a more general foliation into different $\mathcal{N}_\state$. % Relate to problem in Section~\ref{sec:problem_formulation}, and that various initial point-clouds all follow constant dimensional, integrable distributions.

% \subsection{From distributions to Lie groups}
% Highlight equivalance of integrable distributions and Lie sub-algebras (Frobenius theorem), and from there on to Lie group methods.

% \yannik{Possibly including a definition of a manifold Kolmogorov $N$-width?}

% \begin{equation}
%     d_N^{\textnormal{Mfd}}(S) = \inf_{\substack{\mathcal{N} \subseteq \M \\ \dim \mathcal{N} = N}}\;\sup_{s\in S}\;\inf_{v\in\mathcal{N}} \textnormal{dist}(s,v)\,.
% \end{equation}

% \yannik{
% Can then mention that $d_{\dim S}^{\textnormal{Mfd}}(S) = 0$. In line with the expectation that expansions of the type~\eqref{eq:non-separable-approximation} are of sufficient generality to break the Kolmogorov barrier that linear-subspace methods are subject to.
% }

\end{multicols}

\end{document}